\documentclass[9pt, oneside]{amsart}   	
\usepackage{geometry}                		
\usepackage{graphicx}				
\usepackage{amssymb}
\usepackage{amsmath}
\usepackage{amsfonts}
\usepackage{amsthm}
\usepackage{color}
\usepackage{amscd}
\usepackage[usenames,dvipsnames,svgnames,table]{xcolor}
\usepackage[utf8]{inputenc}
\usepackage[parfill]{parskip}
\usepackage[colorlinks=true, pdfstartview=FitV,linkcolor=ForestGreen,citecolor=ForestGreen, urlcolor=black]{hyperref}
\usepackage{enumerate}
\usepackage{esint}
\usepackage{cite}
\usepackage{bbm}
\usepackage{comment}
\usepackage{url} 
\usepackage{tikz}
\usepackage{mathrsfs}

\newcommand{\dd}{\, {\rm d}}

\newcommand{\abs}[1]{\left\vert#1\right\vert}
\newcommand{\Abs}[1]{\big\vert#1\big\vert}
\newcommand{\norm}[1]{\left\Vert#1\right\Vert}  
\newcommand{\Norm}[1]{\big\Vert#1\big\Vert}  

\newcommand{\R}{\ensuremath{{\mathbb R}}}
\newcommand{\N}{\ensuremath{{\mathbb N}}}
\newcommand{\Z}{\ensuremath{{\mathbb Z}}}

\newcommand{\beq}{\begin{equation}}
\newcommand{\eeq}{\end{equation}}
\newcommand{\beqs}{\begin{equation*}}
\newcommand{\eeqs}{\end{equation*}}
\newcommand{\bal}{\begin{equation}\begin{aligned}}
\newcommand{\eal}{\end{aligned}\end{equation}}
\newcommand{\bals}{\begin{equation*}\begin{aligned}}
\newcommand{\eals}{\end{aligned}\end{equation*}}

\newcommand{\supp}{\textrm{supp }}
\newcommand{\dist}{\textrm{dist }}

\newcounter{num} \numberwithin{num}{section}

\newtheorem{theorem}[num]{Theorem}
\newtheorem{proposition}[num]{Proposition}
\newtheorem{lemma}[num]{Lemma}
\newtheorem{corollary}[num]{Corollary}
\theoremstyle{definition}
\newtheorem{definition}[num]{Definition}

\theoremstyle{remark}
\newtheorem{remark}[num]{Remark}

\numberwithin{equation}{section}

\title{Quantitative Schauder estimates for hypoelliptic equations}
\author{Amélie Loher}
\date{Mar 11, 2026}

\address[Amélie Loher]{All Souls College, University of Oxford, OX1 4AL, UK}
\email{amelie.loher@all-souls.ox.ac.uk}


\begin{document}

\begin{abstract}
We derive Schauder estimates using ideas from Campanato's approach for a general class of local hypoelliptic operators and non-local kinetic equations. The method covers equations in divergence and non-divergence form. In particular our results are applicable to the inhomogeneous Landau and to the non-cutoff Boltzmann equation. The paper is self-contained.
\end{abstract}

\maketitle
\tableofcontents

\section{Introduction}
\subsection{Problem Formulation}
We consider functions $f: \R \times \R^d\times \R^d \to \R$ solving a kinetic Fokker-Planck-type equation either in divergence form
\beq
	\partial_t f + v \cdot \nabla_x f = \sum_{1 \leq i, j\leq d} \partial_{v_i} \big(a^{ij}\partial_{v_j}f \big) + \sum_{1 \leq i \leq d}b^i \partial_{v_i} f + c f+ h,
\label{eq:1.1loc_div}
\eeq 
or in non-divergence form 
\beq
	\partial_t f + v \cdot \nabla_x f = \sum_{1 \leq i, j\leq d}a^{ij}\partial^2_{v_iv_j}f + \sum_{1 \leq i \leq d}b^i \partial_{v_i} f + c f+ h,
\label{eq:1.1loc_nondiv}
\eeq
with diffusion coefficients $A = \big(a^{ij}(t, x, v)\big)_{i, j = 1, \dots, d}$, lower order terms $B = (b^i)_{i = 1, \dots, d}$, $c$, and source term $h$. We also consider a fractional analogue 
\beq
	\partial_t f + v \cdot \nabla_x f  = \mathcal L f + h,
\label{eq:1.1nonloc}
\eeq
where 
\beq
	\mathcal L f(t, x, v) = \int_{\R^d} \big[f(t,x,v') - f(t, x, v)\big]K(t, x, v, v') \dd v',
\label{eq:1.2}
\eeq
for some non-negative kernel $K = K(t, x, v, v')$. The integral is to be understood in a principal value sense. The solutions are functions of time, space and velocity $f = f(t, x, v)$.  In the local case \eqref{eq:1.1loc_nondiv}, we assume $A$ to be uniformly elliptic and Hölder continuous. Similarly, in the non-local case \eqref{eq:1.1nonloc} we require a suitably defined ellipticity condition on $K$ as well as Hölder continuity. In both cases, we also assume the source term $h$ and the lower order terms $B, c$ to be Hölder continuous. Our goal is to establish Schauder estimates for solutions of \eqref{eq:1.1loc_div}, \eqref{eq:1.1loc_nondiv} and \eqref{eq:1.1nonloc}, which means that we want to quantify the transfer of Hölder regularity from the coefficients onto the solution of the equation. 

The equation is set in a \textit{kinetic cylinder}
\beq\label{eq:cylinder}
	Q_R(z_0) := \left\{z = (t, x, v) : -R^{2s} \leq t - t_0 \leq 0, \abs{v - v_0} < R, \abs{x - x_0 -(t-t_0)v_0} < R^{1+2s}\right\}
\eeq
for some $R > 0$ and $z_0 = (t_0, x_0, v_0) \in \R\times \R^d\times \R^d$. The parameter $s \in (0, 1)$ will appear in the conditions on the non-local kernel below. 
It determines the non-locality of the operator $\mathcal L$ in \eqref{eq:1.2}. In the local case set $s = 1$. This choice of domain is motivated by the underlying Lie group structure of \eqref{eq:1.1loc_div}, \eqref{eq:1.1loc_nondiv}, \eqref{eq:1.1nonloc}. In fact, equation \eqref{eq:1.1nonloc} is invariant under the scaling defined by
\beq\label{eq:scaling}
	(t, x, v) \to \left(r^{2s}t, r^{1+2s}x, rv\right) =: (t, x, v)_r = z_r,
\eeq
in the sense that a function $f_r$ in these rescaled variables $f_r = f(z_r)$ is a solution to \eqref{eq:1.1nonloc} provided that $f = f(z)$ is, upon suitably rescaling the solution domain. This coincides with the scaling of the local analogues \eqref{eq:1.1loc_div} and \eqref{eq:1.1loc_nondiv} for $s = 1$. Furthermore, these equations \eqref{eq:1.1loc_div}, \eqref{eq:1.1loc_nondiv} and \eqref{eq:1.1nonloc} verify a Galilean invariance:
\beq\label{eq:galilean}
	(t_1, x_1, v_1) \circ (t_2, x_2, v_2) \to (t_1 + t_2, x_1 + x_2 +t_2v_1 , v_1 + v_2),
\eeq
for any two points $z_1, z_2 \in \R^{1+2d}$; that is a function $f_{z_2} = f(z_1 \circ z_2)$ translated according to this Galilean translation \eqref{eq:galilean} is a solution to \eqref{eq:1.1nonloc} (or \eqref{eq:1.1loc_div} / \eqref{eq:1.1loc_nondiv}), provided that $f = f(z_1)$ is, upon suitably translating the solution domain.

The notion of Hölder continuity that we work with takes these invariances into account. On the one hand, the Hölder norm in the velocity variable coincides with the usual notion of Hölder regularity, whereas the regularity in time and space directions is adjusted according to the scaling \eqref{eq:scaling}. On the other hand, we choose a Hölder norm with respect to a distance that is left-invariant by the underlying Lie group structure \eqref{eq:galilean}. We introduce the kinetic Hölder spaces, which defines a notion of Hölder continuity in all variables, in detail in Definition \ref{def:holder} below.

Before stating our main results, we discuss the assumptions that define the ellipticity class and the Hölder continuity of the coefficients. We want our results in the local case to be applicable to the inhomogeneous Landau equation; and in the fractional case, we work with a kernel general enough so that $\mathcal L f$ includes the non-cutoff Boltzmann collision operator.

\subsection{Assumptions and result: the local (non-fractional) case}\label{sec:assum_loc}
We consider \eqref{eq:1.1loc_div} and \eqref{eq:1.1loc_nondiv}, and we assume uniform ellipticity on the divergence coefficients, that is for some $\lambda_0 > 0$ there holds 
\beq
	\forall (t, x, v) \in \R\times\R^d\times \R^d, ~ \forall \xi \in \R^d, \quad \sum_{1 \leq i, j \leq d} a^{i, j}(t ,x, v)\xi_i\xi_j \geq \lambda_0 \abs\xi^2.
\label{eq:unifellip}
\eeq
Moreover, we work with coefficients $A, B, c$ and source term $h$ that are Hölder continuous in the sense of the kinetic Hölder regularity defined below in Definition \ref{def:holder}. 

\begin{theorem}[Schauder estimate for kinetic Fokker-Planck equations]\label{thm:non-div}
Let $\alpha \in (0, 1)$ be given. Let $m \geq 3$ be some integer. Suppose $A \in C_\ell^{m-3+ \alpha}(Q_1)$ satisfies \eqref{eq:unifellip} for some $\lambda_0 > 0$ and assume $B, c, h \in C_\ell^{m-3+\alpha}(Q_1)$. Let $f$ solve \eqref{eq:1.1loc_div} and further assume $\nabla_v A \in C_\ell^{m-3+ \alpha}(Q_1)$, or else let $f$ solve  \eqref{eq:1.1loc_nondiv} in $Q_1$. 
Then we have
\beqs
	\norm{f}_{C_\ell^{m-1 + \alpha}(Q_{1/4})} \leq C \left(\norm{f}_{L^\infty(Q_1)}  + \norm{h}_{C_\ell^{m-3+\alpha}(Q_1)}\right),
\eeqs
for some $C$ depending on $d, \lambda_0, \alpha, \norm{A}_{C_\ell^{m-3+\alpha}},  \norm{B}_{C_\ell^{m-3+\alpha}},  \norm{c}_{C_\ell^{m-3+\alpha}}$, and for the divergence form case also on $\norm{\nabla_v A}_{C_\ell^{m-3+\alpha}}$. 
\end{theorem}
\begin{remark}
In fact, since our approach is constructive, it is straightforward to check that the constant in Theorem \ref{thm:non-div} depends only on the upper bound of the Hölder continuity of the coefficients. 
\end{remark}
We recover Theorem 3.9 of Imbert and Mouhot \cite{IM} when $m = 3$ and Theorem 2.12 of Henderson-Snelson \cite{henderson-snelson} when $m \in \{3, 4\}$. Since we require $\nabla_v A \in C_\ell^{m-3+\alpha}(Q_1)$ for the divergence form equation \eqref{eq:1.1loc_div}, this is merely a sub-case of the non-divergence-form equation \eqref{eq:1.1loc_nondiv} with a Hölder continuous drift term. 
Our approach is robust enough to cover hypoelliptic equations with a higher number of commutators, or also Dini-regular coefficients; we refer to Theorem \ref{thm:hypo} and Theorem \ref{thm:dini} in Appendix \ref{app:hypo-sec}. 

\subsection{Assumptions and results: the non-local (fractional) case}\label{sec:assum_nonloc}
For the non-local equation \eqref{eq:1.1nonloc}, we specify the following notion of ellipticity and Hölder continuity. We consider some $s \in (0, 1)$. To be consistent with the previous work of Imbert-Silvestre \cite{ISschauder}, we consider a non-negative kernel $K = K(t, x, v, v')$ that maps $(t, x, v)$ into a non-negative Radon-Nykodim density $K_{(t, x, v)}$ with respect to Lebesque
measure in $\R^d \setminus \{0\}$ via 
\beqs
	K_{(t, x, v)}(w) := K(t, x, v, v+w).
\eeqs
For any $(t, x, v) \in \R^{1+2d}$ we require the existence of some $0 < \lambda_0 < \Lambda_0$ such that the following conditions hold true (see Remark \ref{rmk:kinetics} for discussion):
For all $r > 0$, we assume the upper bound
\beq
	\int_{B_r}\abs{w}^2K_{(t, x, v)}(w)\dd w \leq \Lambda_0 r^{2-2s}.
\label{eq:upperbound}
\eeq
We further require a coercivity condition for any $r > 0$ and any $\varphi \in C^2(B_{2r})$
\beq
	\lambda_0 \int_{B_r}\int_{B_r} \frac{\abs{\varphi(v)-\varphi(v')}^2}{\abs{v-v'}^{d+2s}} \dd v\dd v' \leq \int_{B_{2r}}\int_{B_{2r}} \big[\varphi(v)-\varphi(v')\big]K_{(t, x, v)} (v'-v) \varphi(v) \dd v'\dd v + \Lambda_0 \norm{\varphi}_{L^2(B_{2r})}.
\label{eq:coercivity}
\eeq
Moreover, we will impose a certain notion of symmetry on the kernel, which can be understood as the distinction between divergence and non-divergence form equations in the fractional case. We either work with the following symmetry condition, which is the non-local analogue of \textit{non-divergence form} equations
\beq
	K_{(t, x, v)} (w) = K_{(t, x, v)}(-w).
\label{eq:symmetry}
\eeq
Or else, if we consider the \textit{divergence form} analogue instead, we require
\beq
	\forall v \in \R^d \quad \Bigg\vert \textrm{PV} \int_{\R^d} \big(K(v, v') - K(v', v)\big)\dd v'\Bigg\vert \leq \Lambda_0,
\label{eq:cancellation1}
\eeq
and if $s \geq \frac{1}{2}$ we assume that  for all $r > 0$
\beq
	\forall v \in \R^d \quad \Bigg\lvert \textrm{PV} \int_{B_r(v)} (v -v') K(v, v')\dd v'\Bigg\rvert \leq \Lambda_0 r^{1-2s}.
\label{eq:cancellation2}
\eeq

Finally we want $K$ to be Hölder continuous with exponent $\alpha \in (0, +\infty)$: given $z_1 = (t_1, x_1, v_1)$ and $z_2 = (t_2, x_2, v_2)$ we assume that there is some $A_0 > 0$ such that for any $r > 0$ 
\beq
	\int_{B_r} \Abs{K_{z_1}(w) - K_{z_2}(w)}\abs{w}^2 \dd w \leq A_0r^{2 - 2s} d_\ell(z_1, z_2)^\alpha,
\label{eq:holder}
\eeq
where $d_\ell$ denotes the kinetic distance defined below in Definition \ref{def:kin_distance}. In the divergence form case, we require in addition to \eqref{eq:holder} for any $r > 0$
\beq
	\Bigg\vert\textrm{PV}\int_{B_r} w \big(K_{z_1}(w) - K_{z_2}(w)\big) \dd w\Bigg\vert \leq A_0r^{1 - 2s} d_\ell(z_1, z_2)^\alpha.
\label{eq:holder_div}
\eeq

\begin{remark}
We observe that, as a consequence of \eqref{eq:upperbound} and \eqref{eq:holder}, we obtain for all $r > 0$ and some $C > 0$
\beqs
	\int_{B_r\setminus B_{r/2}} \Abs{K_{z_1}(w) - K_{z_2}(w)}\dd w \leq C A_0 r^{-2s}d_\ell(z_1, z_2)^\alpha,
\eeqs
which in turn implies
\bal
	&\int_{B_1} \abs{w}^{2s+\alpha} \Abs{K_{z_1}(w) - K_{z_2}(w)}\dd w \leq C  A_0 d_\ell(z_1, z_2)^\alpha,\\
	&\int_{\R^d \setminus B_1} \Abs{K_{z_1}(w) - K_{z_2}(w)}\dd w \leq C  A_0 d_\ell(z_1, z_2)^\alpha.
\label{eq:3.10holder}
\eal
\end{remark}

For integro-differential equations in non-divergence form we recover Theorem 1.6 of \cite{ISschauder}, but in contrast to the methods employed by Imbert-Silvestre, our proof is quantitative. 
\begin{theorem}[Imbert-Silvestre\protect{\cite[Theorem 1.6]{ISschauder}}]\label{thm:non-div-fractional}
Let $0 < s < 1$ and let $0 < \gamma < \min(1, 2s)$. Assume $K$ is a non-negative kernel that is elliptic and Hölder continuous in the sense that it satisfies \eqref{eq:upperbound}-\eqref{eq:symmetry} for some $0 < \lambda_0 < \Lambda_0$ and \eqref{eq:holder} for $\alpha = \frac{2s}{1+2s}\gamma$, for some $A_0 > 0$ and for each $z \in Q_1$. Then any solution $f \in C_\ell^{\gamma}([-1, 0] \times B_1 \times \R^d)$ of \eqref{eq:1.1nonloc} in $Q_1$ satisfies
\beqs
	\norm{f}_{C_\ell^{2s + \alpha}(Q_{1/4})} \leq C\big(\norm{f}_{C_\ell^{\gamma}([-1, 0] \times B_1 \times \R^d)} + \norm{h}_{C_\ell^{\alpha}(Q_1)}\big),
\eeqs
for some constant $C = C(d, s, \lambda_0, \Lambda_0, A_0)$.
\end{theorem}
For divergence form kinetic integro-differential equations we establish the following result.
\begin{theorem}[Schauder estimates for kinetic integro-differential equations in divergence form]\label{thm:div-fractional}
Let $0 < s < 1$ and let $0 < \gamma < \min(1, 2s)$. Assume $K$ is a non-negative kernel that is elliptic in the sense that it satisfies \eqref{eq:upperbound}, \eqref{eq:coercivity}, the (weak) divergence form symmetry \eqref{eq:cancellation1}-\eqref{eq:cancellation2} for some $0 < \lambda_0 < \Lambda_0$. Assume also that $K$ is Hölder continuous in the sense that \eqref{eq:holder}-\eqref{eq:holder_div} are satisfied for $\alpha = \frac{2s}{1+2s}\gamma$, for some $A_0 > 0$ and for each $z \in Q_1$. Then any solution $f \in C_\ell^{\gamma}([-1, 0] \times B_1 \times \R^d)$ of \eqref{eq:1.1nonloc} in $Q_1$ satisfies
\beqs
	\norm{f}_{C_\ell^{2s + \alpha}(Q_{1/4})} \leq C\big(\norm{f}_{C_\ell^{\gamma}([-1, 0] \times B_1 \times \R^d)} + \norm{h}_{C_\ell^{\alpha}(Q_1)}\big),
\eeqs
for some constant $C = C(d, s, \lambda_0, \Lambda_0, A_0)$.
\end{theorem}
\begin{remark}\label{rmk:kinetics}
We emphasise that Theorem \ref{thm:non-div}, Theorem \ref{thm:non-div-fractional} and Theorem \ref{thm:div-fractional} are applicable to the inhomogeneous Landau and the Boltzmann equation without cut-off, respectively. On the one hand, the Landau equation is given by
\beq\label{eq:landau}
	\partial_t f + v \cdot \nabla_x f = \nabla_v \cdot \Bigg(\int_{\R^d} a(v-w) \big[f(w) \nabla f(v) - f(v) \nabla f(w) \big] \dd w \Bigg),
\eeq
where
\beqs
	a(z) = a_{d, \gamma} \abs{z}^{\gamma + 2} \Big( I - \frac{z \otimes z}{\abs{z}^2}\Big),
\eeqs
for $\gamma \geq - d$, $a_{d, \gamma} > 0$. It can be rewritten in divergence \eqref{eq:1.1loc_div} or non-divergence form \eqref{eq:1.1loc_nondiv} for suitable coefficients $A, B, c$, as stated, for example, on page one in \cite{henderson-snelson}.
The Boltzmann equation, on the other hand, is given by
\beq\label{eq:boltzmann}
	\partial_t f + v \cdot \nabla_x f = \int_{\R^d} \int_{\mathbb S^{d-1}}  \big[f(w_*)f(w)  - f(v_*)f(v)  \big] B(\abs{v-v_*}, \cos \theta) \dd v_* \dd \sigma,
\eeq
where 
\beqs
	w = \frac{v+v_*}{2} + \frac{\abs{v-v_*}}{2} \sigma, \qquad w_* =  \frac{v+v_*}{2}  -  \frac{\abs{v-v_*}}{2} \sigma,
\eeqs
and where $\theta$ is the deviation angle between $v$ and $w$. The non-cutoff kernels $B$ are given by 
\beqs
	B(r, \cos \theta) = r^\gamma b(\cos \theta), \qquad b(\cos \theta) \approx \abs{\sin(\theta/2)}^{-d + 1 - 2s},
\eeqs
for $\gamma > -d$ and $s \in (0,1)$. Using Carleman coordinates and the cancellation lemma, we can rewrite this as \eqref{eq:1.1nonloc}, for some specific kernel $K$. 

In a certain conditional regime upon which we do not elaborate here, we can check that the coefficients in the Landau equation and the kernel of the Boltzmann equation satisfy the ellipticity assumptions made in Section \ref{sec:assum_loc} and Section \ref{sec:assum_nonloc}, respectively. In particular, any Hölder continuous solution $f$ of \eqref{eq:landau} or \eqref{eq:boltzmann} with mass, energy and entropy bounded above, and mass bounded below, satisfies the Schauder estimate in Theorem \ref{thm:non-div} or \ref{thm:non-div-fractional}, respectively. We refer the reader to \cite[Theorem 1.2]{henderson-snelson} for the Landau equation, and \cite[Section 4]{ISglobal} for the Boltzmann equation.
\end{remark}

\subsection{Contribution}
Our contribution consists of a quantitative and unified approach to Schauder estimates for kinetic equations with either non-fractional or fractional coefficients, in either non-divergence or divergence form. In this respect it improves upon the previous results on kinetic Schauder estimates in the local case by Imbert-Mouhot \cite{IM} and Henderson-Snelson \cite{henderson-snelson}, and in the non-local case by Imbert-Silvestre \cite{ISschauder}. On the one hand, in the non-fractional case we manage to gain two orders of Hölder regularity \textit{at any smoothness $m -3\geq 0$}. On the other hand, we establish Schauder estimates for \textit{divergence form equations} in Theorem \ref{thm:non-div} and \ref{thm:div-fractional}, which, to the best of our knowledge, is a novelty in the fractional case. Moreover, our approach is fully quantitative, which, in the fractional case, avoids the blow-up argument used in \cite{ISschauder}. Finally, in the non-fractional case, the method is robust enough to deal with hypoelliptic operators of any order, and it works even more generally for Dini-regular coefficients, see Theorem \ref{thm:hypo} and Theorem \ref{thm:dini}, respectively. To the best of our knowledge this is the first use of Campanato spaces in a kinetic context to deduce Schauder estimates in all variables. The robustness of the methods permits to deal with a variety of problems with a similar structure, from local to non-local equations, from one Hörmander commutator to any number of commutators, and from Hölder-continuous coefficients to mere Dini-continuity. 

\subsection{Previous Literature Results}
All the works on Schauder estimates have to be classified according to the notion of Hölder continuity that is used and the assumptions on the coefficients that are made. 

In the local case, there is the work by Imbert and Mouhot \cite{IM}, which adapts Krylov's approach \cite{krylov} to the kinetic setting. Furthermore, in \cite{henderson-snelson}, Henderson-Snelson discuss a $C^\infty$-smoothing estimate for the Landau equation by iteratively applying their Schauder estimates. There are also two articles \cite{henderson-wang, dong-yas} for kinetic Fokker-Planck equations, which assume less regularity in time, and deduce partial Schauder estimates for space and velocity only. Their goal is to reduce the regularity assumptions needed on time. 
However, the Hölder norms defined in \cite{henderson-wang} differ from our notion of Hölder continuity, since theirs do not take the Hölder continuity in the temporal variable into account. In \cite{dong-yas}, they start with merely measurable temporal coefficients to end up with an estimate that still bounds 
the $C^{2+\alpha}$ norm in velocity (but not in time). 



In the non-local case, the work that inspired us most is Imbert and Silvestre \cite{ISschauder}. In particular, the definition of kinetic Hölder spaces, the notion of distance and degree of a kinetic polynomial all stem from their seminal contribution on regularity for the non-cutoff Boltzmann equation \cite{ISschauder, ISglobal, IS, IS?}.Their approach to Schauder estimates consists of first proving a Liouville-type theorem, where from an assumed growth rate on the Hölder norm of a solution to a parabolic PDE, one concludes that the solution is a polynomial with an order implied by the growth rate; second using a blow-up argument. Their work is inspired from Ros-Oton-Serra \cite{rosoton-serra}, who have used these techniques for non-local operators that are generators of stable and symmetric Lévy processes. Note, however, that this method is non-constructive, as it relies on compactness arguments.
The structure of this argument comes from Simon \cite{simon}, who used a scaling argument to derive a Liouville theorem for general hypoelliptic operators, which states that the solution is a polynomial of a certain order, from which he finally deduces the Schauder estimate by a compactness argument.

We follow Campanato's approach. This method was first established for elliptic equations. A nice reference is the book by Giaquinta and Martinazzi \cite[Chapter 5]{GiaquintaMartinazzi}. We would also like to refer to the book of Krylov \cite[Section 3.3]{krylov}, where he introduces higher order Campanato norms, without specifically naming them. 
The idea is to use the scaling stemming from a combination of a Poincaré inequality, Sobolev and regularity estimates on the constant coefficient equation. In contrast, Simon's scaling argument \cite[Lemma 1]{simon} replaces the Sobolev inequality and regularity estimates by a reasoning of Hörmander \cite[Theorem 3.7]{hormander} based on the closed graph theorem and the homogeneity of the operator; let us refer the reader to Appendix \ref{app:hypo}. Through the characterisation of Hölder norms by Campanato norms, we replace the blow-up argument of Simon by a constructive method.

\subsection{Strategy}

We consider a solution of either the local or non-local equation, and freeze coefficients: the part which solves a constant coefficient equation with zero source term is considered separately from the rest. The latter can be viewed as a \textit{lower order source term} with the expected bounds due to the Hölder continuity of the coefficients.  For the constant coefficient solution, we subtract a certain polynomial constructed from the vector fields of the equation of degree up to the order of our equation, such that we have a zero-averaged function. We then apply Poincaré's inequality repeatedly as long as the zero-average condition is satisfied and the integrand is orthogonal to the kernel of the Poincaré inequality, that is one order higher than the equation itself. We then use an $L^\infty$-bound and Sobolev's embedding. But then, since we consider a solution to a constant coefficient equation, regularity estimates yield a bound uniform in the Hölder norm of the coefficients. These regularity estimates are proved by using Hölder's inequality in Fourier variables, and they rely on a transfer of regularity from the velocity variable onto the spatial variable due to the hypoelliptic character of the equation. Eventually, the combination of all these ideas results in a higher order Campanato norm on the left hand side, which characterises Hölder norms. The transfer of regularity from the coefficients onto the solution arises from the scaling of the equation.

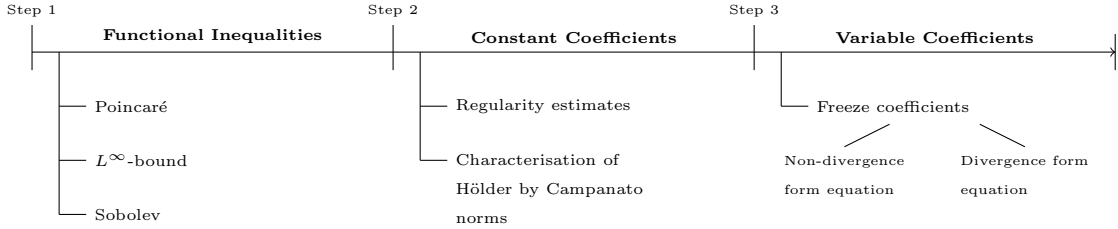
\begin{tikzpicture}[scale =1.2]
\draw (0, 0) -- node[above] {\scriptsize{\textbf{Functional Inequalities}}}(4, 0);
\draw (4, 0) --  node[above] {\scriptsize{\textbf{Constant Coefficients}}}(8, 0);
\draw [->](8, 0) --  node[above] {\scriptsize{\textbf{Variable Coefficients}}}(12, 0);
\draw (0, -0.2) -- (0, 0.3) node[above] {\tiny{Step 1}};
\draw (4, -0.2) -- (4, 0.3) node[above] {\tiny{Step 2}};
\draw (8, -0.2) -- (8, 0.3) node[above] {\tiny{Step 3}};
\draw (12, -0.2) -- (12, 0.2);
\draw (0.3, 0) -- (0.3, -1.8);
\draw (0.3, -0.6) -- (0.6, -0.6)node[anchor=west] {\scriptsize{Poincaré}};
\draw (0.3, -1.2) -- (0.6, -1.2)node[anchor=west] {\scriptsize{$L^\infty$-bound}};
\draw (0.3, -1.8) -- (0.6, -1.8)node[anchor=west] {\scriptsize{Sobolev}};
\draw (4.3, 0) -- (4.3, -1.2);
\draw (4.3, -0.6) -- (4.6, -0.6) node[anchor=west] {\scriptsize{Regularity estimates}};
\draw (4.3, -1.2) -- (4.6, -1.2);
\node[align=left, anchor=west]  at (4.6, -1.51)  {\scriptsize{Characterisation of}\\ \scriptsize{Hölder by Campanato}\\ \scriptsize{norms}};
\draw (8.3, 0) -- (8.3, -0.6);
\draw (8.3, -0.6) -- (8.6, -0.6)node[anchor=west] {\scriptsize{Freeze coefficients}};
\draw (9.5, -0.8) -- (9, -1.05);
\draw (10.5, -0.8) -- (11, -1.05);
\node[align=left, anchor=north] at (9, -1.05) {\tiny{Non-divergence} \\ \tiny{form equation}};
\node[align=left, anchor=north] at (11, -1.05) {\tiny{Divergence form}\\ \tiny{equation}};
\end{tikzpicture}

Section \ref{sec:prelims} introduces the notion of Hölder spaces that we work with. We state the equivalence of Hölder and Campanato norms in Theorem \ref{thm:equicampholderhigh}, whose proof is postponed to the Appendix \ref{app:campanato}. In Section \ref{sec:toolbox} we assemble tools that are setting the framework for Campanato's approach. In particular, we derive regularity estimates \ref{prop:fractional_energy} for the constant coefficient equation.
Section \ref{sec:campanato-inequality} is devoted to the proof of Campanato's inequality. Section \ref{sec:schauder_nonfrac} completes the Schauder estimates in the non-fractional case where $s = 1$, proving Theorem \ref{thm:non-div}, whereas Section \ref{sec:schauder_frac} treats the fractional case $s \in (0, 1)$, proving Theorems \ref{thm:non-div-fractional} and \ref{thm:div-fractional}. 

\subsection{Notation}
Whenever a statement holds both in the local and the non-local case, we will state the non-local result and we ask the reader to set $s = 1$ to obtain the local analogue. 

We write $z = (t, x, v)$ for an element of $\R \times \R^d \times \R^d$. Moreover, we let $n = 2s +2d(s +1)$ denote the total dimension respecting the scaling of the equation.

The transport operator will be denoted as $\mathcal T = \partial_t + v\cdot \nabla_x$. 

We use the floor function $\lfloor a\rfloor$ for $a\in \R$ to denote the greatest integer $k \in \Z$ such that $k \leq a$. We further use the abbreviation $a \lesssim b$ for $a, b \in \R$ if there exists a constant $C > 0$ such that $a \leq C b$. Similarly, $a \gtrsim b$ denotes $a \geq C b$ for some $C >0$. Finally, $a \approx b$ if $a \lesssim b$ and $a \gtrsim b$.

For a domain $\Omega \subset \R^{1+2d}$ we denote by $\Omega^v$ the temporal and spatial domain at fixed velocity $v$, that is for $z = (t, x, v) \in \Omega$ we have $(t, x) \in \Omega^v$ for any $v$ in $\pi_3(\Omega)$ for $\pi_3(t,x,v) = v$.

\section{Preliminaries}\label{sec:prelims}
\subsection{Definition of kinetic Hölder spaces}
To define the Hölder spaces that we are working with, we first need to understand the underlying Lie group structure of \eqref{eq:1.1loc_div}, \eqref{eq:1.1loc_nondiv} and \eqref{eq:1.1nonloc}. These equations are invariant under Galilean transformations \eqref{eq:galilean}, in the sense that if $f$ solves \eqref{eq:1.1loc_div}, \eqref{eq:1.1loc_nondiv} or \eqref{eq:1.1nonloc} then $f(z_1 \circ z_2)$ is also a solution of the respective equation with a translated right hand side and a translated kernel. The translated kernel will still be elliptic. Furthermore, both equations are invariant under scaling \eqref{eq:scaling} for a rescaled right hand side. The rescaled kernel will again be elliptic. The notion of distance that we introduce respects these invariances. It has been used by Imbert-Silvestre \cite[Def. 2.1]{ISschauder} before.
\begin{definition}[Kinetic distance]\label{def:kin_distance}
For $z_1 = (t_1, x_1, v_1), z_2 = (t_2, x_2, v_2) \in \R^{1+2d}$ we define
\beqs
	d_\ell(z_1, z_2) := \min_{w \in \R^d} \Big\{\max\Big[ \abs{t_1-t_2}^{\frac{1}{2s}}, \abs{x_1 - x_2 - (t_1-t_2)w}^{\frac{1}{2s}}, \abs{v_1-w}, \abs{v_2 - w}\Big]\Big\}.
\eeqs
Moreover we define
\beqs
	\norm{z} = \max\Big\{\abs{t}^{\frac{1}{2s}}, \abs{x}^{\frac{1}{1+2s}}, \abs{v}\Big\}.
\eeqs
This is a quasi-norm, but not a norm in the mathematical sense.
\end{definition} 
\begin{remark}
This notion of distance should not be confused with the distance function towards the grazing set as introduced in \cite[Def. 1]{Guo}, which apart from the name does not have any connection to this distance here.
\end{remark}
Let us observe that this distance is left invariant in the sense that $d_\ell(z \circ z_1, z \circ z_2) = d_\ell(z_1, z_2)$ for any $z, z_1, z_2 \in \R^{1+2d}$. We can also reformulate it as $d_\ell$ being the infimum value of $r > 0$ such that both $z_1, z_2$ belong to $Q_r(z_0)$ for some $z_0 \in \R^{1+2d}$. Other equivalent formulations are
\bals
	d_\ell(z_1, z_2) \approx \norm{z_2^{-1} \circ z_1} \approx\norm{z_1^{-1} \circ z_2}\approx \inf_{w\in\R^d} \abs{t_2 - t_1}^{\frac{1}{2s}} + \abs{x_2 - x_1-(t_2 - t_1)w}^{\frac{1}{1+2s}} + \abs{v_1 - w} + \abs{v_2 - w}.
\eals
For more remarks on this distance we refer the reader to \cite[Section 2]{ISschauder}. 

In addition to the kinetic distance, we use the notion of kinetic degree of a monomial $m_j \in \R[t, x, v]$ introduced in \cite[Subsection 2.2]{ISschauder} as
\beqs
	\textrm{deg}_{\textrm{kin}} m_j = 2s\cdot j_0 + (1+2s)\Bigg(\sum_{i=1}^{d} j_i \Bigg)+ \sum_{i = d+1}^{2d} j_i = 2s\cdot j_0 + (1+2s)\cdot\abs{J_1} + \abs{J_2} =: \abs{J},
\eeqs	
where we denote a multi-index $j \in \N^{1+2d}$ with $j = (j_0, J_1, J_2)$ where $J_1 = (j_1, \dots, j_d)$ and $J_2 = (j_{d+1}, \dots, j_{2d})$. 
Under scaling $z \to z_R$ as in \eqref{eq:scaling}, a monomial $m_j$ behaves as
\beqs
	m_j(z_R) = R^{2sj_0}t^{j_0}R^{(1+2s)\abs{J_1}}x^{J_1} R^{\abs{J_2}}v^{J_2} = R^{\abs{J}} z^j, \quad R > 0, 
\eeqs
and its degree is precisely $\abs{J} = 2sj_0 + (1+2s)\abs{J_1} + \abs{J_2}$. 
We denote with $\mathcal P_k$ the space of $k$ degree polynomials. Note that in the non-local case $k$ is in the discrete set $k \in \N + 2s\N$, and we will write $k = 2s\cdot k_0 + (1+2s)\cdot k_1 + k_2$ for $k_0, k_1, k_2 \in \N$. An element $p \in \mathcal P_k$ is written as
\beq\label{eq:poly}
	p(t, x, v) = \sum_{\substack{j \in \N^{1+2d},\\ \abs{J} \leq k}} a_j m_j(z).
\eeq
The sum is taken over $j_0\in [0, k_0], \abs{J_1}\in [0, k_1], \abs{J_2} \in [0, k_2]$. We will abbreviate this and write $\abs{J} \leq k$. In the local case there is no ambiguity.

Our notion of Hölder continuity leans on \cite[Def. 2.2]{IM} and \cite[Def. 2.3]{ISschauder}. 
\begin{definition}[Hölder spaces]
Given an open set $\Omega \subset \R \times \R^d \times \R^d$ and $\beta \in (0, \infty)$ we say that $f: \Omega \to \R$ is $C_\ell^{\beta}(\Omega)$ at a point $z_0 \in \R^{1+2d}$ if there is a polynomial $p \in \R[t, x, v]$ with kinetic degree $\textrm{deg}_{\textrm{kin }} p < \beta$ and a constant $C > 0$ such that
\beq
	\forall r > 0 \qquad \norm{f - p}_{L^\infty(Q_r(z_0)\cap \Omega)} \leq C r^{\beta}.
\label{eq:holderdef}
\eeq
When this property holds at every point $z_0 \in \Omega$ we say that $f \in C_\ell^{\beta}(\Omega)$. The semi-norm $[f]_{C_\ell^{\beta}(\Omega)}$ is the smallest $C$ such that \eqref{eq:holderdef} holds for all $z_0 \in \Omega$. We equip $C_\ell^{\beta}(\Omega)$ with the norm
\beqs
	\norm{f}_{C_\ell^{\beta}(\Omega)} = \norm{f}_{L^\infty(\Omega)} + [f]_{C_\ell^{\beta}(\Omega)}.
\eeqs
\label{def:holder}
\end{definition}
\begin{remark}
This definition coincides with the definition of \cite[Def. 2.2]{IM}. As the authors point out, it is equivalent to ask that for any $z \in \Omega$
\beqs
	\abs{f(z) - p(z)} \leq Cd_\ell(z, z_0)^{\beta}.
\eeqs
We can further rephrase Hölder regularity of $f$ at $z_0$ due to the left-invariance as follows \cite{ISschauder}. For any $z \in \R^{1+2d}$ such that $z_0 \circ z \in \Omega$ we have
\beqs
	\abs{f(z_0 \circ z) - p_{z_0}(z)} \leq C\norm{z}^{\beta},
\eeqs
where $p_{z_0}(z) = p(z_0 \circ z)$. The polynomial $p_{z_0}$ will be the expansion of $f$ at $z_0$. 
\end{remark}

Hölder spaces can also be characterised in terms of Campanato spaces. These have been introduced by Campanato himself \cite{Camp, Camp1, Camp2} in the elliptic context. We adapt his notion to the kinetic setting. 
\begin{definition}[Higher order Campanato spaces]
Let $\Omega \subset \R^{1+2d}$ be an open subset. 
For $1 \leq p \leq \infty, ~ \lambda \geq 0, ~ k \geq 0$ we define the Campanato space $\mathcal L^{p, \lambda}_k\big(\Omega\big)$ as 
\begin{equation}
    \mathcal L_k^{p, \lambda}\big(\Omega\big):= \Bigg\{ f \in L^p\big(\Omega\big) : \sup_{z \in \Omega, r > 0} r^{-\lambda}\inf_{P \in \mathcal P_k}\int_{Q_r(z)\cap \Omega}\abs{f -  P}^p\dd z < +\infty\Bigg\}
\end{equation}
where $\mathcal P_k$ is the space of polynomials of kinetic degree less or equal $k$. We endow it with the seminorm 
\begin{equation}
    [f]_{\mathcal L^{p, \lambda}_k}^p := \sup_{z \in \Omega, r > 0} r^{-\lambda}\inf_{P \in \mathcal P_k}\int_{Q_r(z)\cap \Omega}\abs{f - P}^p\dd z 
\end{equation}
and the norm
\begin{equation}
    \norm{f}_{\mathcal L^{p, \lambda}_k} = [f]_{\mathcal L^{p, \lambda}_k}+ \norm{f}_{L^p}.
\end{equation}
\label{def:campanatoH}
\end{definition}
\begin{remark}
\begin{enumerate}[i.]
	\item We observe that for the local case $k \in \N$, whereas in the non-local case $k \in \N + 2s\N$. 
	\item Campanato's spaces are most commonly known for $k = 0$. Such spaces have been used for Schauder estimates in the elliptic context \cite{GiaquintaMartinazzi} and in the parabolic context by Krylov \cite{krylov}. To gain higher Hölder continuity ($k \geq 1$) the equation was just differentiated. This would not work as easily for our equations. A method inspired from Campanato's approach with $k=0$ has been developed for partial Schauder estimates in the kinetic setting in \cite{dong-yas}, however without addressing Hölder continuity in time. Even if the use of the higher-order Campanato spaces are a natural step if the equation cannot be differentiated easily, and thus presumably was employed in the literature in such settings, we are only aware of one other work, \cite{safonov}, (which was recently pointed out to us) that employs these spaces to derive higher-order Schauder estimates: there it was in the context of fully non-linear elliptic equations.
\end{enumerate}
\end{remark}
The next subsection states a characterisation of Hölder continuity in terms of Campanato's norms. 

\subsection{Relation between Hölder and Campanato spaces}
Hölder spaces can be characterised through Campanato spaces, and vice versa. This equivalence has been established by Campanato himself in \cite{Camp1} for the lowest order Campanato space, and in \cite{Camp} for higher order Campanato spaces. Following Campanato's arguments, we can show the following relation between Campanato and Hölder spaces defined in Definition \ref{def:holder} and Definition \ref{def:campanatoH}. We refer the reader to the proof in Appendix \ref{app:campanato}.
\begin{theorem}[Campanato]
Let $\tilde z_0 \in \R^{1+2d}$ and $R > 0$, and write $\Omega = Q_R(\tilde z_0)$. Then, 
for $n + kp < \lambda \leq n + (k+1)p$ and $\beta = \frac{\lambda - n}{p}$ we have $\mathcal L_k^{p, \lambda}(\Omega) \cong C_\ell^{\beta}(\bar \Omega)$, where $n = 2s + 2d(s+1)$.
\label{thm:equicampholderhigh}
\end{theorem}
\begin{remark}
For the local case, setting $s = 1$ yields the same result.
\end{remark}

\subsection{Differential operators}\label{sec:diffops}
In this section, we show how to relate Hölder norms to kinetic differential operators. We reprove Lemma 2.7 of \cite{ISschauder} to make our paper self-contained. 
\begin{lemma}[Imbert-Silvestre\protect{\cite[Lemma 2.7]{ISschauder}}]
Let $D = \mathcal T, D = \nabla_x$ or $D = \nabla_v$. Let $f \in C_\ell^{\beta}(Q)$ for $\beta \in (0, \infty)$ and $Q$ some kinetic cylinder. Then $D^l f \in C_\ell^{\beta - k}(Q)$ where $k$ is the kinetic degree of $D^l$, $l \in \N$, and
\beqs
	[D^lf]_{C_\ell^{\beta - k}(Q)} \leq C[f]_{C_\ell^{\beta}(Q)}. 
\eeqs
\label{lem:2.7}
\end{lemma}
\begin{proof}
Let $z_1, z_2 \in Q$. Since $f \in C_\ell^{\beta}(Q)$ there exists a polynomial $p$ with degree $k = \rm{deg}_{\rm{kin}}p <\beta$  so that for $z \in Q$ with $\norm{z}  \leq d_\ell(z_1, z_2) = r$
\bal
	&\abs{f(z_1 \circ z) - p(z_1 \circ z)} \leq C r^{\beta},\\
	&\abs{f(z_2 \circ z) - p(z_2 \circ z)} \leq C r^{\beta},
\label{eq:fholderaux}
\eal
where $C = [f]_{C_\ell^{\beta}(Q)}$.
We can compute that 
\beqs
	p(z_1 \circ z) = f(z_1) + \mathcal Tf(z_1)  t + \nabla_x f(z_1)\cdot x + \nabla_v f(z_1) \cdot v + \dots
\eeqs
By equivalence of norms in finite dimensional spaces, we know that if $\sup_{\abs{z} \leq 1}\abs{p(z)} \leq C_0$ then the coefficients of $p$ denoted by $a_j$ will satisfy $\sup_{j}\abs{a_j} \leq C C_0$ for some constant $C$ depending on $k$ and $n$. Scaling this argument yields together with \eqref{eq:fholderaux}
\bals
	\Abs{D^lf(z_1) - D^lf(z_2)} r^{k} \leq C r^{\beta}, 
\eals
where $D^l$ is the differential operator of degree $k$. 
\end{proof}
We will need a similar estimate for the fractional operator \eqref{eq:1.2}. We start with a global bound, see \cite[Lemma 3.6]{ISschauder} for kernels in non-divergence form \eqref{eq:symmetry}.
\begin{lemma}\label{lem:3.6}
Assume $0< \alpha < \min(1, 2s)$. For any non-negative kernel $K$ satisfying \eqref{eq:upperbound}, and either satisfy \eqref{eq:symmetry} or \eqref{eq:cancellation1}, \eqref{eq:cancellation2}. Then for $f \in C_\ell^{2s +\alpha}(\R^{1+2d})$ there holds
\beqs
	[\mathcal Lf]_{C_\ell^{\alpha}(\R^{2d+1})} \leq C [f]_{C_\ell^{2s + \alpha}(\R^{2d+1})}. 
\eeqs
\end{lemma}
\begin{proof}
Let $z, \xi \in \R^{1+2d}$. We know that
\beq
	\abs{f(z \circ \xi) - p_z(\xi)} \leq [f]_{C_\ell^{2s + \alpha}}\norm{\xi}^{2s+\alpha}.
\label{eq:auxfholder}
\eeq
We need to estimate
\bals
	\mathcal Lf(z\circ \xi) - \mathcal Lf(z) &= \int_{\R^d} \big[f(z \circ \xi \circ (0, 0, v' - v - \xi_v)) - f(z\circ \xi)\big]K(z\circ \xi, v') \dd v' \\
	&\quad- \int_{\R^d} \big[f\big(z \circ (0, 0, v' - v)\big) - f(z) \big]K(z, v')\dd v'.
\eals
We distinguish the close and the far part. Let $R > 0$ and write for ease of notation $\phi = (0, 0, v' - v - \xi_v)$ and $\psi = (0, 0, v' - v)$ for $\xi = (\xi_t, \xi_x, \xi_v)$.

If we assume symmetry in the non-divergence form \eqref{eq:symmetry}, then we can symmetrise the integral and remove the principal value. We find
\bal\label{eq:Lf-aux}
	\textrm{PV}\int_{B_R(v)} &\big[f(z \circ \psi) - f(z) \big]K(z, v')\dd v' \\
	&= \frac{1}{2} \int_{B_R(v)} \big[f(z \circ \psi) +f(z \circ -\psi)  - 2f(z) \big]K(z, v')\dd v' \\
	&= \frac{1}{2}\int_{B_R(v)} \left[f(z \circ \psi)) - p_z(\psi)\right]K(z, v')\dd v' + \frac{1}{2}\int_{B_R(v)} \left[p_z(\psi) - f(z)\right]K(z, v')\dd v'\\
	&\quad + \frac{1}{2}\int_{B_R(v)} \left[f(z \circ -\psi) - p_z(-\psi)\right]K(z, v')\dd v' + \frac{1}{2}\int_{B_R(v)} \left[p_z(-\psi) - f(z)\right]K(z, v')\dd v'.
\eal
The polynomial $p_z(\psi)$ is given by 
\[
	p_z(\psi) = f(z) + \nabla_v f(z)\cdot (v'-v) + (v'-v)^T\cdot\nabla_v^2 f(z) \cdot (v'-v).
\]
Any higher order terms vanish since $\textrm{deg } p < 2s+\alpha$. The terms involving $t$ or $x$ vanish when evaluated at $\psi$. The first order terms in the integrand above will vanish due to \eqref{eq:symmetry}. Thus we further bound \eqref{eq:Lf-aux}
\bals
	\textrm{PV}\int_{B_R(v)} &\big[f(z \circ \psi) - f(z) \big]K(z, v')\dd v' \\
	&\leq [f]_{C_\ell^{2s + \alpha}}\int_{B_R(v)}\abs{v'-v}^{2s+\alpha}K(z, v') \dd v' + \Abs{\nabla^2_v f(z)} \int_{B_R(v)} \abs{v'-v}^2 K(z, v') \dd v'\\
	&\lesssim_\Lambda [f]_{C_\ell^{2s + \alpha}}R^\alpha + \Abs{\nabla^2_v f(z)} R^{2-2s}.
\eals
The last inequality uses for the second order term, the upper bound \eqref{eq:upperbound}. 
All estimates are independent of $z \in \R^{1+2d}$ so that we similarly obtain
\bals
	\textrm{PV} \int_{B_R(v + \xi_v)} \big[f(z \circ \xi \circ\phi) - f(z\circ \xi) \big]K(z\circ \xi, v')\dd v' \lesssim_\Lambda [f]_{C_\ell^{2s + \alpha}}R^\alpha + \Abs{\nabla^2_v f(z\circ \xi)} R^{2-2s}.
\eals
Therefore
\bals
	\int_{B_R(v+\xi_v)} \big[f(z \circ \xi \circ \phi) - f(z\circ \xi)\big]K(z\circ \xi, v') \dd v' &- \int_{B_R(v)} \big[f(z \circ \psi) - f(z) \big]K(z, v')\dd v' \\
	&\lesssim_\Lambda [f]_{C_\ell^{2s + \alpha}}R^\alpha +  \Abs{\nabla^2_v f(z\circ \xi) - \nabla^2_v f(z)} R^{2-2s}\\
	&\lesssim_\Lambda [f]_{C_\ell^{2s + \alpha}}R^\alpha +  \norm{\xi}^{2s+\alpha - 2} R^{2-2s}[f]_{C_\ell^{2s + \alpha}}.
\eals
We used Lemma \ref{lem:2.7} for the last inequality.
Choosing $R = \norm{\xi}$ therefore yields
\bals
	\int_{B_R(v+\xi_v)} \big[f(z \circ \xi \circ \phi) - f(z\circ \xi)\big]K(z\circ \xi, v') \dd v' &- \int_{B_R(v)} \big[f(z \circ \psi) - f(z) \big]K(z, v')\dd v' \lesssim_\Lambda [f]_{C_\ell^{2s + \alpha}}R^\alpha.
\eals

If, instead of \eqref{eq:symmetry}, we assume \eqref{eq:cancellation1} and \eqref{eq:cancellation2}, then we bound
\bals
	\Bigg\vert\textrm{PV}\int_{B_R(v)} \big[f(z \circ \psi) - f(z) \big]K(z, v')\dd v'\Bigg\vert &=\Bigg\vert\textrm{PV}\int_{B_R(v)} \big[f(z \circ \psi) -p_z(\psi) - \big(f(z)-p_z(\psi)\big) \big]K(z, v')\dd v'\Bigg\vert \\
	&\leq \Bigg\vert\textrm{PV}\int_{B_R(v)}\abs{f(z \circ \psi) - f(z)}K(z, v') \dd v'\Bigg\vert\ \\
	&\quad+ \Bigg\vert\textrm{PV}\int_{B_R(v)} D_v f(z)\cdot \big(v-v'\big)  K(z, v') \dd v'\Bigg\vert \\
	&\quad+  \Bigg\vert\textrm{PV}\int_{B_R(v)} \Abs{D_v^2 f(z)} \abs{v-v'}^2 K(z, v') \dd v'\Bigg\vert\\
	&\leq [f]_{C^{2s+\alpha}_\ell} \int_{B_R(v)}\abs{v'-v}^{2s+\alpha} K(z, v') \dd v' \\
	&\quad+ C\Lambda  \Abs{D_v f(z)} R^{1-2s}+ C \Lambda\Abs{D_v^2 f(z)} R^{2-2s}\\
	&\lesssim_\Lambda [f]_{C^{2s+\alpha}_\ell} R^\alpha +  \Abs{D_v f(z)} R^{1-2s}+ \Abs{D_v^2 f(z)} R^{2-2s}.
\eals
We again used \eqref{eq:upperbound} and \eqref{eq:auxfholder}. The same computations yield
\bals
	\Bigg\vert\textrm{PV} \int_{B_R(v + \xi_v)} &\big[f(z \circ \xi \circ\phi) - f(z\circ \xi) \big]K(z\circ \xi, v')\dd v'\Bigg\vert \\
	&\lesssim_\Lambda [f]_{C_\ell^{2s + \alpha}}R^\alpha +  \Abs{D_v f(z\circ \xi)} R^{1-2s}+ \Abs{D_v^2 f(z\circ \xi)} R^{2-2s},
\eals
so that as before, we obtain with Lemma \ref{lem:2.7}
\bals
	\int_{B_R(v+\xi_v)} \big[f(z \circ \xi \circ \phi) &- f(z\circ \xi)\big]K(z\circ \xi, v') \dd v' - \int_{B_R(v)} \big[f(z \circ \psi) - f(z) \big]K(z, v')\dd v' \\
	&\lesssim_\Lambda [f]_{C_\ell^{2s + \alpha}}R^\alpha +  \Abs{\nabla_v f(z\circ \xi) - \nabla_v f(z)} R^{1-2s}+  \Abs{\nabla^2_v f(z\circ \xi) - \nabla^2_v f(z)} R^{2-2s}\\
	&\lesssim_\Lambda [f]_{C_\ell^{2s + \alpha}}R^\alpha +  \norm{\xi}^{2s+\alpha - 1} R^{1-2s}[f]_{C_\ell^{2s + \alpha}} + \norm{\xi}^{2s+\alpha - 2} R^{2-2s}[f]_{C_\ell^{2s + \alpha}}\\
	&\lesssim_\Lambda [f]_{C_\ell^{2s + \alpha}}R^\alpha,
\eals	 
by choosing $\norm{\xi} = R$.

For the far part we do not need to distinguish non-divergence form from divergence form. In both cases we separate the integral into different terms
\beqs
	\int_{\R^d \setminus B_R(v + \xi_v)} \big[f(z \circ \xi \circ \phi) - f(z\circ \xi)\big]K(z\circ \xi, v') \dd v' - \int_{\R^d \setminus B_R(v)} \big[f(z \circ \psi) - f(z) \big]K(z, v')\dd v' \leq \sum_{i = 1}^5 I_i,
\eeqs
with 
\bals
	&I_1 = \int_{\R^d \setminus B_R(v + \xi_v)} \Abs{f(z \circ \xi \circ \phi) - p_{z \circ \phi}(\phi^{-1}\circ \xi\circ \phi)}K(z \circ \xi, v') \dd v',\\
	&I_2 = \int_{\R^d \setminus B_R(v + \xi_v)} \Abs{f(z \circ \xi) - p_{z}(\xi)}K(z \circ \xi, v') \dd v',\\
	&I_3 = \int_{\R^d \setminus B_R(v + \xi_v)} \Abs{p_{z \circ \phi}(\phi^{-1}\circ \xi\circ \phi) - p_{z}(\xi)}K(z \circ \xi, v') \dd v',\\
	&I_4 = \int_{\R^d \setminus B_R(v)} \Abs{p_{z\circ\psi}(\xi) - p_{z}(\xi) - f(z\circ \psi) + f(z)} K(z, v') \dd v',\\
	&I_5 = \int_{\R^d \setminus B_R(v)} \Abs{p_{z\circ\psi}(\xi) - p_{z}(\xi)} K(z, v') \dd v'.
\eals
Using that 
\bals
	\Abs{f(z \circ \xi \circ \phi) - p_{z \circ \phi}(\phi^{-1}\circ \xi\circ \phi)} &\leq [f]_{C_\ell^{2s + \alpha}}\norm{\phi^{-1}\circ \xi \circ \phi}^{2s+\alpha}\\
 	&\leq [f]_{C_\ell^{2s + \alpha}}\Big(\norm{\xi} +\abs{v'-v-\xi_v}^{\frac{1}{1+2s}}\norm{\xi_t}^{\frac{2s}{1+2s}}\Big)^{2s+\alpha},
\eals
we bound the first term with \eqref{eq:upperbound} by
\bals
	I_1 &\leq C [f]_{C_\ell^{2s + \alpha}}\Bigg(\norm{\xi}^{2s+\alpha}R^{-2s} +\norm{\xi}^{\frac{2s(2s+\alpha)}{1+2s}} \int_{\R^d\setminus B_R(v + \xi_v)} \abs{v'-v-\xi_v}^{\frac{2s+\alpha}{1+2s}}K(z \circ \xi, v')\dd v'\Bigg)\\
	&\leq C [f]_{C_\ell^{2s + \alpha}}R^{-2s}\Big(\norm{\xi}^{2s+\alpha} + \norm{\xi}^{\frac{2s(2s+\alpha)}{1+2s}} R^{\frac{2s+\alpha}{1+2s}}\Big).
\eals
For $I_2$ we get 
\bals
	I_2 &\leq C [f]_{C_\ell^{2s + \alpha}}\norm{\xi}^{2s+\alpha}R^{-2s}.
\eals
We further notice that $I_4$ is the same as $I_5$ without the lowest order term of $p_{z\circ \psi} - p_z$. To estimate $I_5$ we write $p_z(\xi) = \sum a_j(z) m_j(\xi)$. Note that by Lemma \ref{lem:2.7} the coefficients $a_j$ satisfy
\beqs
	[a_j]_{C_\ell^{2s - j + \alpha}} \leq C[f]_{C_\ell^{2s + \alpha}},
\eeqs
where $j$ is the degree of the corresponding monomial. Thus
\beqs
	I_5 \leq C[f]_{C_\ell^{2s + \alpha}}\Big(R^{\alpha} + R^{\alpha-1}\norm{\xi} +R^{\alpha-2s}\norm{\xi}^{2s} + R^{\alpha - 2}\norm{\xi}^2\Big),
\eeqs
and
\beqs
	I_4 \leq C[f]_{C_\ell^{2s + \alpha}}\Big(R^{\alpha-1}\norm{\xi} +R^{\alpha-2s}\norm{\xi}^{2s} + R^{\alpha - 2}\norm{\xi}^2\Big).
\eeqs
For $I_3$ we notice that $\phi^{-1} \circ \xi \circ \phi^{-1} = \big(\xi_t, \xi_x + \xi_t(v' - v-\xi_v), \xi_v\big)$. Apart from the space variable this coincides with $\xi$. But since we only conisder polynomial expansion up to order $2s+\alpha < 2s + 1$ the space variable won't appear, so that in fact $\abs{I_3} = \abs{I_5}$. 
We now choose $R = \norm{\xi}$ so that all terms are bounded by $I_i \leq C [f]_{C_\ell^{2s + \alpha}}\norm{\xi}^{\alpha}$ for all $i = 1, \dots, 5$. 
\end{proof}
To localise Lemma \ref{lem:3.6} we follow the proof of Imbert and Silvestre in \cite[Lemma 3.7]{ISschauder}. Here we also cover the non-divergence form symmetry \eqref{eq:cancellation1}-\eqref{eq:cancellation2}.
\begin{lemma}[Imbert-Silvestre\protect{\cite[Lemma 3.7]{ISschauder}}]
Let $0 < \alpha \leq \gamma < \min(1, 2s)$ and let $K$ satisfy \eqref{eq:upperbound} and either \eqref{eq:symmetry} or \eqref{eq:cancellation1}, \eqref{eq:cancellation2}. Then 
\beqs
	[\mathcal Lf]_{C_\ell^\alpha(Q_{\frac{1}{2}})} \leq C\Big( [f]_{C_\ell^{2s+\alpha}(Q_{\frac{1}{2}})} +  [f]_{C_\ell^\gamma((-1, 0] \times B_1 \times \R^d)}\Big),
\eeqs
for some $C$ depending on $n, s, \Lambda_0$ and $A_0$. 
\label{lem:3.7}
\end{lemma}
\begin{proof}
We write $\mathcal L f(z) = \tilde{\mathcal L}f(z) + C(z)$ where $\tilde{\mathcal L}f(z)$ corresponds to the non-local operator in \eqref{eq:1.2} with kernel $\tilde K(v, v') = \mathbbm 1_{B_\rho(v)}(v')K(v, v')$ and $C(z)$ corresponds to $\mathcal Lf$ with kernel $\big[1-\mathbbm 1_{B_\rho(v)}(v') \big]K(v, v')$ for some small $\rho>0$. Then by Lemma \ref{lem:3.6} we have
\beqs
	[\tilde{\mathcal L} f]_{C_\ell^{\alpha}(Q_{\frac{1}{2}})} \leq C [f]_{C_\ell^{2s+\alpha}(Q_1)}.
\eeqs
Now we consider $z_0, z \in Q_{\frac{1}{2}}$ such that $z_0 \circ z \in Q_{\frac{1}{2}}$. If we write $\phi = (0, 0, v'-v - v_0)$ and $\psi = (0, 0, v'-v)$ we have for $K(w) = K(v, v+w)$
\bals
	C(z_0 \circ z) &- C(z) \\
	&= \int_{\R^d\setminus B_\rho(v+v_0)} \big[ f(z_0 \circ z \circ \phi)- f(z_0 \circ z)\big]K(z_0 \circ z, v') \dd v' - \int_{\R^d\setminus B_\rho(v)} \big[ f(z \circ\psi) -f(z)\big]K(z, v') \dd v'\\
	&= \int_{\R^d\setminus B_\rho} \big[ f(z)- f(z_0 \circ z)\big]K(w) \dd w - \int_{\R^d \setminus B_\rho} \big[f\big(z \circ(0, 0, w)\big) -f\big(z_0 \circ z \circ(0, 0, w)\big)\big]K(w) \dd w\\
	&\leq C \Lambda_0 \rho^{-2s} [f]_{C_\ell^\gamma} d_\ell(z, z_0 \circ z)^\alpha +C [f]_{C_\ell^\gamma}\int_{\R^d\setminus B_\rho} d_\ell\big(z\circ (0, 0, w), z_0 \circ z\circ (0, 0, w)\big)^\gamma K(w) \dd w, 
\eals
since $\alpha \leq \gamma$. But now we compute 
\bals
	d_\ell(z\circ (0, 0, w), z_0 \circ z\circ (0, 0, w)) &= \Norm{ (0, 0, w)^{-1} \circ z^{-1} \circ z_0^{-1} \circ z\circ (0, 0, w)}\\ 
	&= d_\ell((z_0 \circ z)^{-1}, z) - (0, t_0w, 0)\\ 
	&\lesssim d_\ell(z, z_0 \circ z) + \abs{t-t_0}^{\frac{1}{1+2s}} \abs{w}^{\frac{1}{1+2s}}\\
	&\lesssim d_\ell(z, z_0 \circ z)^{\frac{2s}{1+2s}}\big(1 +  \abs{w}^{\frac{1}{1+2s}}\big).
\eals
Therefore, since $\alpha \leq \frac{2s}{1+2s}$ and since $K$ satisfies the upper bound \eqref{eq:upperbound} we find
\bals
	C(z_0 \circ z) - C(z) &\leq  C \Lambda_0 [f]_{C_\ell^\gamma} \rho^{-2s}d_\ell(z, z_0 \circ z)^\alpha.
\eals
This concludes the proof.
\end{proof}
\subsection{Interpolation}
We also have an interpolation inequality, see \cite[Prop. 2.10]{ISschauder}. Unlike the other preliminary results that we have stated in Subsection \ref{sec:diffops}, the proof of the following proposition is verbatim the same as in \cite[Prop. 2.10]{ISschauder}. For the sake of self-containment we recall it in Appendix \ref{app:interpolation}.
\begin{proposition}[Imbert-Silvestre\protect{\cite[Prop. 2.10]{ISschauder}}]
Given $\beta_1 < \beta_2 < \beta_3$ so that $\beta_2 = \theta\beta_1 +(1-\theta)\beta_3$, then for any $f \in C_{\ell}^{\beta_3}(Q_1)$ there holds
\beqs
	[f]_{C_{\ell}^{\beta_2}(Q_1)} \leq [f]^\theta_{C_{\ell}^{\beta_1}(Q_1)}[f]^{1-\theta}_{C_\ell^{\beta_3}(Q_1)} + [f]_{C_{\ell}^{\beta_1}(Q_1)}.
\eeqs
In particular for all $\varepsilon > 0$
\beqs
	[f]_{C_\ell^{\beta_2}(Q_1)} \leq C(\varepsilon) [f]_{C_\ell^{\beta_1}(Q_1)}+ \varepsilon [f]_{C_\ell^{\beta_3}(Q_1)}.
\eeqs
\label{prop:interpolation}
\end{proposition}

\subsection{Commutator estimates}\label{sec:comm}
We denote by 
\[
	\abs{D_v}^{ks} = (-\Delta_v)^{\frac{ks}{2}}.
\]

This subsection is concerned with commutator estimates of general order. We start with the fractional Leibniz rule.
\begin{lemma}[Fractional Leibniz rule]\label{lem:product-cutoff}
Let $\beta \geq 0$ and let $\Omega \subset \R^d$ be a bounded Lipschitz domain.
Let $\rho \in (0, 1]$, $v_0 \in \R^d$, and let $\theta_1 \in C_c^\infty(\Omega)$ satisfy $\theta_1(v) = \theta_0\big(\frac{v - v_0}{\rho}\big)$ 
for some fixed $\theta_0 \in C_c^\infty(\R^d)$. 
Then for every $f \in H^\beta(\Omega)$,
\beq\label{eq:product-cutoff}
	\norm{\theta_1 f}_{H^\beta(\R^d)}
	\leq C\norm{f}_{H^\beta(\Omega)} + C\rho^{-\beta}\norm{f}_{L^2(\Omega)},
\eeq
where $C = C(d, \beta, \theta_0, \Omega)$.
\end{lemma}
\begin{proof}
Since $\Omega$ is bounded and Lipschitz, there exists a bounded linear 
extension operator (see e.g. \cite[Theorem~5.4]{Adams75})
\[
	E \colon H^\beta(\Omega) \to H^\beta(\R^d), 
	\qquad 
	\norm{Ef}_{H^\beta(\R^d)} \leq C(\Omega, d, \beta)\,\norm{f}_{H^\beta(\Omega)},
\]
with $Ef = f$ a.e.\ on $\Omega$.

Since $\supp(\theta_1) \subset \Omega$ and $Ef = f$ on $\Omega$, we have
\beqs
	\theta_1 f = \theta_1 (Ef) \qquad \text{on } \R^d.
\eeqs
It therefore suffices to prove the global estimate: 
for every $g \in H^\beta(\R^d)$,
\beq\label{eq:global-product}
	\norm{\theta_1 g}_{H^\beta(\R^d)} 
	\leq C\norm{g}_{H^\beta(\R^d)} + C\rho^{-\beta}\norm{g}_{L^2(\R^d)}.
\eeq
Then, applying \eqref{eq:global-product} to $g = Ef$ and using 
$\norm{Ef}_{H^\beta(\R^d)} \leq C\norm{f}_{H^\beta(\Omega)}$ and $\norm{Ef}_{L^2(\R^d)} \leq C\norm{f}_{L^2(\Omega)}$ yields the result \eqref{eq:product-cutoff}.

To prove \eqref{eq:global-product}, we write $\beta = n + \sigma$ with $n = \lfloor \beta \rfloor \in \N_0$ 
and $\sigma \in [0, 1)$. The $H^\beta$~norm satisfies
\beq\label{eq:Hbeta-equiv}
	\norm{g}_{H^\beta(\R^d)}^2 
	\approx \norm{g}_{H^n(\R^d)}^2 
	+ \mathbf 1_{\sigma > 0}\sum_{\abs{\gamma} = n} [D^\gamma g]_{W^{\sigma,2}(\R^d)}^2.
\eeq

\textit{Step 1: Integer part.} 
By the Leibniz rule, for $\abs{\gamma} \leq n$:
\[
	D^\gamma(\theta_1 g) 
	= \sum_{\gamma_1 + \gamma_2 = \gamma} 
		\binom{\gamma}{\gamma_1}
		(D^{\gamma_1}\theta_1)(D^{\gamma_2}g).
\]
Since $\norm{D^{\gamma_1}\theta_1}_{L^\infty} 
\leq C\rho^{-\abs{\gamma_1}}$, each term satisfies
\[
	\norm{(D^{\gamma_1}\theta_1)(D^{\gamma_2}g)}_{L^2}
	\leq C\rho^{-\abs{\gamma_1}}\norm{g}_{H^{n-\abs{\gamma_1}}}.
\]
By the standard interpolation inequality 
$\norm{g}_{H^{n-k}} 
\leq C\norm{g}_{H^n}^{1-k/n}\norm{g}_{L^2}^{k/n}$
for $k \in \{1, \dots, n\}$ and by Young's inequality,
\[
	\rho^{-k}\norm{g}_{H^{n-k}} 
	\leq \varepsilon \norm{g}_{H^n} + C_\varepsilon\,\rho^{-n}\norm{g}_{L^2},
\]
for any $\varepsilon > 0$.
Summing over all multi-indices and choosing $\varepsilon$ sufficiently small, 
\beq\label{eq:integer-bound}
	\norm{\theta_1 g}_{H^n} 
	\leq C\norm{g}_{H^n} + C\rho^{-n}\norm{g}_{L^2}.
\eeq

\medskip
\textit{Step 2: Fractional part.}
Let $\abs{\gamma} = n$. By the Leibniz rule,
\[
	D^\gamma(\theta_1 g) = \theta_1\,D^\gamma g 
		+ \sum_{\substack{\gamma_1 + \gamma_2 = \gamma \\ \abs{\gamma_1}\geq 1}}
		c_{\gamma_1,\gamma_2}\,(D^{\gamma_1}\theta_1)(D^{\gamma_2}g).
\]

\textit{Highest-order term.} 
Setting $u = D^\gamma g$ and using 
$$\theta_1(x)u(x) - \theta_1(y)u(y) 
= \theta_1(x)(u(x)-u(y)) + (\theta_1(x)-\theta_1(y))u(y),$$ we estimate
\bals
	[\theta_1 u]_{W^{\sigma,2}(\R^d)}^2 
	&\leq 2\norm{\theta_1}_{L^\infty}^2 [u]_{W^{\sigma,2}(\R^d)}^2 
	+ 2\int_{\R^d}\abs{u(y)}^2 
		\underbrace{\int_{\R^d} 
		\frac{\abs{\theta_1(x)-\theta_1(y)}^2}{\abs{x-y}^{d+2\sigma}}\dd x}_{=:\,I(y)}
		\dd y.
\eals
We bound $I(y)$ by splitting the integration domain. 
For $\abs{x-y} < \rho$, we use 
$\abs{\theta_1(x)-\theta_1(y)} 
\leq \norm{\nabla\theta_1}_{L^\infty}\abs{x-y} \leq C\rho^{-1}\abs{x-y}$, so that for $\sigma < 1$,
\[
	\int_{\abs{x-y}<\rho} \frac{\abs{\theta_1(x)-\theta_1(y)}^2}{\abs{x-y}^{d+2\sigma}}\dd x
	\leq C\rho^{-2}\int_{\abs{z}<\rho}\abs{z}^{2-d-2\sigma}\dd z
	= C\rho^{-2\sigma}.
\]
For $\abs{x-y}\geq \rho$: using 
$\abs{\theta_1(x)-\theta_1(y)} \leq 2\norm{\theta_1}_{L^\infty}$,
\[
	\int_{\abs{x-y}\geq \rho} \frac{\abs{\theta_1(x)-\theta_1(y)}^2}{\abs{x-y}^{d+2\sigma}}\dd x
	\leq C\int_{\abs{z}\geq \rho}\abs{z}^{-d-2\sigma}\dd z
	= C\rho^{-2\sigma}.
\]
Therefore $I(y) \leq C \rho^{-2\sigma}$, and
\beq\label{eq:gagliardo-leading}
	[\theta_1 u]_{W^{\sigma,2}} 
	\leq C[u]_{W^{\sigma,2}} + C\rho^{-\sigma}\norm{u}_{L^2}
	= C[D^\gamma g]_{W^{\sigma,2}} + C\rho^{-\sigma}\norm{g}_{H^n}.
\eeq

\textit{Lower-order terms.} 
For $\abs{\gamma_1} = k \geq 1$, $\abs{\gamma_2} = n - k$, 
set $\Theta = D^{\gamma_1}\theta_1$ (so $\norm{\Theta}_{L^\infty} \leq C\rho^{-k}$,
$\norm{\nabla\Theta}_{L^\infty} \leq C\rho^{-k-1}$) 
and $w = D^{\gamma_2}g$. Then by the same splitting argument:
\[
	[\Theta w]_{W^{\sigma,2}} 
	\leq C\norm{\Theta}_{L^\infty}[w]_{W^{\sigma,2}} 
		+ C(\norm{\nabla\Theta}_{L^\infty} + \norm{\Theta}_{L^\infty})\norm{w}_{L^2}
	\leq C\rho^{-k}[w]_{W^{\sigma,2}} + C\rho^{-k-1}\norm{w}_{L^2}.
\]
Hence
\[
	[\Theta w]_{W^{\sigma,2}} 
	\leq C\rho^{-k}\norm{g}_{H^{n-k+\sigma}} + C\rho^{-k-1}\norm{g}_{H^{n-k}}.
\]
By the interpolation inequality 
$\norm{g}_{H^{n-k+\sigma}} 
\leq C\norm{g}_{H^\beta}^{1-\frac{k}{\beta}}\norm{g}_{L^2}^{\frac{k}{\beta}}$
and Young's inequality as above, for any $\varepsilon > 0$:
\beq\label{eq:lower-order-young}
	\rho^{-k}\norm{g}_{H^{n-k+\sigma}} 
	\leq \varepsilon\norm{g}_{H^\beta} 
		+ C_\varepsilon\,\rho^{-\beta}\norm{g}_{L^2}.
\eeq
The same argument applies to $\rho^{-k-1}\norm{g}_{H^{n-k}}$.

\medskip
\textit{Step 3: Conclusion.}
Collecting \eqref{eq:integer-bound}, \eqref{eq:gagliardo-leading}, 
and \eqref{eq:lower-order-young}, and choosing $\varepsilon$ 
sufficiently small, yields
\[
	\norm{\theta_1 g}_{H^\beta(\R^d)} 
	\leq C\norm{g}_{H^\beta(\R^d)} 
		+ C\rho^{-\beta}\norm{g}_{L^2(\R^d)}.
\]
Applying this to $g = Ef$ concludes the proof.
\end{proof}

We also record an auxiliary result on local Sobolev norms. 
\begin{lemma}[Local Sobolev bound for nonlocal operators]\label{lem:local-nonlocal}
Let $\beta \geq 0$, $\sigma > 0$ with $\sigma \notin 2\N_0$, 
and let $\Omega' \Subset \Omega \subset \R^d$ be bounded Lipschitz domains 
with $\dist(\Omega', \R^d \setminus \Omega) \geq c_0\rho$ for some $c_0, \rho > 0$.
If $f \in H^{\sigma+\beta}(\Omega) \cap L^\infty(\R^d)$, then
\beq\label{eq:local-nonlocal}
	\norm{\abs{D}^{\sigma}f}_{H^{\beta}(\Omega')} 
	\leq C\norm{f}_{H^{\sigma+\beta}(\Omega)} 
		+ C\rho^{-(\sigma+\beta)}\norm{f}_{L^\infty(\R^d)},
\eeq
where $C = C(d, \sigma, \beta, c_0, \Omega, \Omega')$.
\end{lemma}

\begin{proof}
Choose $\theta_2 \in C_c^\infty(\Omega)$ with 
$\theta_2 = 1$ on $\Omega' + B_{c_0\rho/2}$, 
$\supp(\theta_2) \subset \Omega$, 
and $\norm{D^k\theta_2}_{L^\infty} \leq C\rho^{-k}$.
On $\Omega'$ we decompose
\[
	\abs{D}^{\sigma} f 
	= \abs{D}^{\sigma}(\theta_2 f) 
		+ \abs{D}^{\sigma}\big((1-\theta_2) f\big).
\]

Since $\theta_2 f$ is compactly supported in~$\Omega$,
Lemma~\ref{lem:product-cutoff} gives
$$\norm{\theta_2 f}_{H^{\sigma+\beta}(\R^d)} 
\leq C\norm{f}_{H^{\sigma+\beta}(\Omega)} 
+ C\rho^{-(\sigma+\beta)}\norm{f}_{L^2(\Omega)}.$$
By Plancherel (since $\theta_2 f \in H^{\sigma+\beta}(\R^d)$):
\[
	\norm{\abs{D}^{\sigma}(\theta_2 f)}_{H^{\beta}(\R^d)}
	\leq \norm{\theta_2 f}_{H^{\sigma+\beta}(\R^d)}
	\leq C\norm{f}_{H^{\sigma+\beta}(\Omega)} 
		+ C\rho^{-(\sigma+\beta)}\norm{f}_{L^\infty(\R^d)}.
\]

Since $(1-\theta_2) f$ vanishes on $\Omega' + B_{c_0\rho/2}$, 
for every $v \in \Omega'$ the kernel of $\abs{D}^{\sigma}$ is 
non-singular and all derivatives are controlled by the tail:
for any $N \in \N_0$,
\[
	\abs{D^N \abs{D}^{\sigma}\big((1-\theta_2) f\big)(v)}
	\leq C_N \norm{f}_{L^\infty}
		\int_{\abs{z} > c_0\rho/2} \abs{z}^{-(d+\sigma+N)}\dd z
	= C_N \rho^{-(\sigma+N)}\norm{f}_{L^\infty}.
\]
In particular,
$\norm{\abs{D}^{\sigma}\big((1-\theta_2) f\big)}_{H^{\beta}(\Omega')}
\leq C\rho^{-(\sigma+\beta)}\norm{f}_{L^\infty}$.
Combining both terms gives~\eqref{eq:local-nonlocal}.
\end{proof}

\begin{proposition}[Higher order commutator estimates]\label{lem:commutator}
Let $s \in (0,1)$. Let $k \geq 1$. Let $D$ be a closed set and $\Omega$ open such that $D \Subset \Omega \subset B_{\bar R/2} \subset \R^d$ for $0 < \bar R \leq 2$. Let $\varphi$ be a smooth function with support in $D$, let $f \in H^{(k-1)s}(\Omega) \cap L^\infty(\R^d)$ and let $$\rho := \frac{\textrm{dist}(D, \R^d\setminus \Omega)}{4}.$$
Then there exists $h_1 \in L^2(\R^d \setminus \Omega) \cap L^\infty(\R^d)$ and $h_2 \in L^2(\R^d)$ such that
\beqs
	\abs{D_v}^{ks} [\varphi f] - \varphi \abs{D_v}^{ks} f = h_1 + h_2.
\eeqs
Moreover, $h_1$ satisfies
\beq\label{eq:h1}
	\norm{h_1}_{L^2(\R^d\setminus \Omega)} \leq C  \rho^{-ks} \norm{\varphi}_{L^\infty} \norm{f}_{L^2(D)};
\eeq
whereas $h_2 = 0$ outside $\Omega$, and it satisfies for $s \in \left(0, \frac{1}{k}\right)$
\beq\label{eq:h2}
	\norm{h_2}_{L^2(\R^d)} \leq C \rho^{1-ks} \norm{\varphi}_{C^1} \norm{f}_{L^2(\Omega)} \leq C \rho^{-ks}  \norm{f}_{L^2(\Omega)}.
\eeq
Else if $s \in \left[\frac{1}{k}, 1\right)$, then 
\bal\label{eq:h2-gen}
	&\norm{h_1 + h_2}_{L^2(\R^d)} \\
	&\quad\leq C  \rho^{1-s}  \norm{\varphi}_{C^1}\norm{f}_{H^{(k-1)s}(\Omega)} + C\rho^{2-ks}   \norm{\varphi}_{C^2} \norm{f}_{L^2(\Omega)} +C  \rho^{-ks}\norm{\varphi}_{L^\infty} \norm{f}_{L^\infty(\R^d)} \\
	&\quad\leq C  \rho^{-s}  \norm{f}_{H^{(k-1)s}(\Omega)} + C\rho^{-ks}   \norm{f}_{L^2(\Omega)} +C  \rho^{-ks} \norm{f}_{L^\infty(\R^d)}.
\eal
\end{proposition}
\begin{proof}[Proof of Proposition \ref{lem:commutator}]
Since $\rho = \frac{1}{4}\dist(D, \R^d \setminus \Omega)$, the domain $E := D + B_{2\rho}$ satisfies $D \Subset E \Subset \Omega$ with $\textrm{dist}(D, \R^d \setminus E) = 2\rho$ and $\textrm{dist}(E, \R^d \setminus \Omega) = 2\rho$.

Assume first that $k \in \{1, 2\}$. In this case, $h_1, h_2$ are defined pointwisely by
\bal\label{eq:def-pw-h1-h2}
	&h_1(v) = \int_{\R^d\setminus B_\rho(v)} f(w)\frac{ \big(\varphi(w) - \varphi(v)\big)}{\abs{v-w}^{d+ks}} \dd w, \\
	&h_2(v) = \int_{B_\rho(v)} f(w)\frac{ \big(\varphi(w) - \varphi(v)\big) }{\abs{v-w}^{d+ks}} \dd w.
\eal
It is easy to see that $h_2$ vanishes outside $\Omega$: since $\supp \varphi \subset D$, we have $B_\rho(v) \cap \supp \varphi = \emptyset$ for any $v \notin \Omega$, so $h_2(v) = 0$ for $v \notin \Omega$.

To bound $h_1$, we notice that $\varphi(v) = 0$ for $v \notin D$. Thus if $v \notin \Omega \supset D$, then
\beqs
	h_1 =\int_{\R^d\setminus B_\rho(v)} \frac{ f(w)\varphi(w) }{\abs{v-w}^{d+ks}} \dd w = \int_{D}  \frac{ f(w)\varphi(w) }{\abs{v-w}^{d+ks}} \dd w.
\eeqs
Therefore, using Cauchy-Schwarz and Fubini's theorem
\bals
	\int_{\R^d \setminus  \Omega} h_1^2 \dd v &= \int_{\R^d \setminus  \Omega} \Bigg( \int_{D}  \frac{ f(w)\varphi(w) }{\abs{v-w}^{d+ks}}  \dd w\Bigg)^2 \dd v\\
	&\leq \norm{\varphi}_{L^\infty}^2 \int_{\R^d \setminus  \Omega}\Bigg(\int_{D}   \frac{ f^2(w) }{\abs{v-w}^{d+ks}}  \dd w\Bigg)\Bigg(\int_{D}  \frac{1}{\abs{v-w}^{d+ks}} \dd w\Bigg)  \dd v\\
	&\leq C \rho^{-ks} \norm{\varphi}_{L^\infty}^2\int_{D} f^2(w) \int_{ \abs{v-w} \geq 2\rho}  \frac{ 1 }{\abs{v-w}^{d+ks}} \dd v \dd w\\
	&\leq C \rho^{-2ks} \norm{\varphi}_{L^\infty}^2 \norm{f}_{L^2(D)}^2.
\eals
This yields \eqref{eq:h1}.

To bound $h_2$ we first consider $s \in \left(0, \frac{1}{k}\right)$. We use Cauchy-Schwarz and Fubini
\bals
	\norm{h_2}_{L^2(\R^d)}^2&= \int_E \Bigg(\int_{B_\rho(v)} f(w) \frac{\big(\varphi(w) - \varphi(v)\big)}{\abs{v-w}^{d+ks}} \dd w\Bigg)^2 \dd v\\
	&\leq   \int_E \Bigg(\int_{B_\rho(v)} f^2(w)\frac{\Abs{\varphi(w) - \varphi(v)} }{\abs{v-w}^{d+ks}} \dd w\Bigg) \Bigg(\int_{B_\rho(v)}  \frac{\abs{\varphi(w) - \varphi(v)}}{\abs{v-w}^{d+ks}} \dd w\Bigg)  \dd v\\
	&\leq  \norm{\varphi}_{C^1}^2 \int_E \Bigg(\int_{B_\rho(v)}  \frac{f^2(w)}{\abs{v-w}^{d+ks-1}}\dd w\Bigg) \Bigg(\int_{B_\rho(v)} \frac{1}{\abs{v-w}^{d+ks-1}} \dd w\Bigg)  \dd v\\
	&\leq  C \rho^{1-ks} \norm{\varphi}_{C^1}^2 \int_{\Omega} f^2(w)\int_{E \cap B_\rho(w)}  \frac{1}{\abs{v-w}^{d+ks-1}} \dd v \dd w\\
	&\leq C \rho^{2-2ks} \norm{\varphi}_{C^1}^2\norm{f}_{L^2(\Omega)}^2.
\eals
This yields \eqref{eq:h2} for $k \in \{1, 2\}$. The same argument works for general $k > 2$, as long as $ks < 1$, since $h_1$ and $h_2$ defined as in \eqref{eq:def-pw-h1-h2} still make sense, and therefore we conclude \eqref{eq:h2}.

To prove \eqref{eq:h2-gen}, we consider first $k = 2$: let $s \in \left[\frac{1}{2}, 1\right)$.
We estimate $h_2$ by duality. Let $g \in L^{2}(\R^d)$. Then, since $\textrm{supp }h_2 \subseteq E$, we have
\bals
	\int_E h_2(v) g(v) \dd v &= \int_E \int_{B_\rho(v)} g(v) f(w) \frac{\big(\varphi(w) - \varphi(v)\big)}{\abs{v-w}^{d+ks}} \dd w \dd v\\
	&= \int_E \int_{B_\rho(v)} g(v) f(v) \frac{\big(\varphi(w) - \varphi(v)\big)}{\abs{v-w}^{d+ks}} \dd w \dd v\\
	&\quad + \int_{E} \int_{B_\rho(v)} g(v) \big(f(w) - f(v)\big)  \frac{\big(\varphi(w) - \varphi(v)\big)}{\abs{v-w}^{d+ks}}\dd w \dd v.
\eals
By Taylor's theorem and by oddness of the integrand, we find
\bals
	\int_E \int_{B_\rho(v)} g(v) f(v) \frac{\big(\varphi(w) - \varphi(v)\big)}{\abs{v-w}^{d+ks}} \dd w \dd v &= \int_E  g(v) f(v)\nabla \varphi(v) \Bigg(\int_{B_\rho(v)} \frac{\big(w-v\big)}{\abs{v-w}^{d+ks}} \dd w\Bigg) \dd v \\
	&\quad+ \norm{\varphi}_{C^2} \int_E  g(v) f(v) \Bigg(\int_{B_\rho(v)} \frac{\abs{v-w}^2}{\abs{v-w}^{d+ks}} \dd w\Bigg) \dd v\\
 	&\leq C \rho^{2-2s} \norm{\varphi}_{C^2} \norm{f}_{L^2(E)} \norm{g}_{L^2(E)}.
\eals
For the remaining term, we use Cauchy-Schwarz twice, 
\bals
	\int_{E} &\int_{B_\rho(v)} g(v) \big(f(w) - f(v)\big)  \frac{\big(\varphi(w) - \varphi(v)\big)}{\abs{v-w}^{d+ks}}\dd w \dd v \\
	&\leq \norm{g}_{L^2(E)} \Bigg\{\int_{E} \Bigg(\int_{B_\rho(v)} \big(f(w) - f(v)\big)  \frac{\big(\varphi(w) - \varphi(v)\big)}{\abs{v-w}^{d+ks}}\dd w\Bigg)^2 \dd v\Bigg\}^{\frac{1}{2}}\\
	&\leq \norm{g}_{L^2(E)} \Bigg\{\int_{E} \int_{B_\rho(v)}    \frac{\abs{f(w) - f(v)}^2}{\abs{v-w}^{d+ks}}\dd w  \int_{B_\rho(v)}  \frac{\abs{\varphi(w) - \varphi(v)}^2}{\abs{v-w}^{d+ks}} \dd w \dd v\Bigg\}^{\frac{1}{2}}\\
	&\leq C \rho^{1-s} \norm{\varphi}_{C^1} \norm{g}_{L^2(E)} \Bigg\{\int_{E} \int_{B_\rho(v)}    \frac{\abs{f(w) - f(v)}^2}{\abs{v-w}^{d+ks}}\dd w  \dd v\Bigg\}^{\frac{1}{2}}\\
	&\leq C \rho^{1-s} \norm{\varphi}_{C^1} \norm{g}_{L^2(E)}\norm{f}_{H^s(\Omega)}. 
\eals
Thus
\[
	\int_E h_2(v) g(v) \dd v  \leq  C \rho^{2-2s} \norm{\varphi}_{C^2} \norm{f}_{L^2(E)} \norm{g}_{L^2(E)} + C \rho^{1-s} \norm{\varphi}_{C^1} \norm{g}_{L^2(E)}\norm{f}_{H^s(\Omega)}. 
\]
This implies 
\[
	\norm{h_2}_{L^2(\R^d)} \leq  C \rho^{2-2s} \norm{\varphi}_{C^2} \norm{f}_{L^2(\Omega)} + C \rho^{1-s} \norm{\varphi}_{C^1} \norm{f}_{H^s(\Omega)}.
\]
and therefore proves \eqref{eq:h2-gen} for $k = 2$.

For $k > 2$, we need a more elaborate argument. We require several auxiliary results, which follow in the remainder of this subsection. 
\end{proof}

\begin{proposition}[Kato--Ponce commutator estimate]\label{prop:kato-ponce}
Let $\alpha > 0$, $s \in (0, 1)$ with $\alpha \geq s$ and $\varphi \in \mathcal S(\R^d)$.  
Then for every $f \in H^{\alpha-s}(\R^d)$,
\beq\label{eq:KP}
	\norm{[\abs{D}^\alpha, \varphi] f}_{L^2(\R^d)}
	\leq C\,\big\Vert\abs{\cdot}^s\,\abs{\hat\varphi}\big\Vert_{L^1}\,
		\norm{f}_{\dot H^{\alpha-s}(\R^d)}
	+ C\,\big\Vert\abs{\cdot}^{\alpha}\abs{\hat\varphi}\big\Vert_{L^1}\,
		\norm{f}_{L^2(\R^d)},
\eeq
where $C = C(d, \alpha)$.  
In particular, if $\varphi$ is a smooth cutoff at spatial scale~$\rho$,
i.e.\ $\varphi(v)=\varphi_1\big(\frac{v}{\rho}\big)$ for a fixed $\varphi_1 \in C_c^\infty(\R^d)$,
then
\beq\label{eq:KP-scaling}
	\norm{[\abs{D}^\alpha, \varphi] f}_{L^2(\R^d)}
	\leq C(\varphi_1, d, \alpha)\,\Big(
		\rho^{-s}\norm{f}_{\dot H^{\alpha-s}(\R^d)}
		+ \rho^{-\alpha}\norm{f}_{L^2(\R^d)}\Big).
\eeq
\end{proposition}

\begin{remark}
The estimate \eqref{eq:KP} is a consequence of the Kato--Ponce inequality 
\cite{KatoPonce88}; the version with $L^\infty$ norms on $\varphi$ 
(i.e.\ with $\norm{\nabla \varphi}_{L^\infty}$ and $\norm{\abs{D}^\alpha\varphi}_{L^\infty}$
in place of the $L^1$~Fourier norms) was established in \cite{GrafakosOh14}.
Since we only apply \eqref{eq:KP} to cutoff functions of a fixed smooth profile,
the formulation with $L^1$~Fourier norms is sufficient and admits
a short self-contained proof, which we give below.
\end{remark}

\begin{proof}[Proof of Proposition~\ref{prop:kato-ponce}]
We proceed in four steps.

\medskip
\textit{Step 1: Fourier representation of the commutator.}
Denote by $\hat f$ the Fourier transform of $f$. 
Since $\widehat{\varphi f}(\xi) = (\hat\varphi \ast \hat f)(\xi) 
= \int_{\R^d} \hat\varphi(\xi-\eta)\hat f(\eta)\dd\eta$, we have
\beq\label{eq:KP-fourier}
	\widehat{[\abs{D}^\alpha, \varphi] f}(\xi)
	= \int_{\R^d} \Big(\abs{\xi}^\alpha - \abs{\eta}^\alpha\Big)\,
		\hat\varphi(\xi-\eta)\,\hat f(\eta)\dd\eta.
\eeq

\medskip
\textit{Step 2: Multiplier bound.}
We claim that for all $\xi, \eta \in \R^d$, $\alpha > 0$, 
and $s \in (0, \min(1, \alpha)]$:
\beq\label{eq:KP-multiplier}
	\Big\vert\abs{\xi}^\alpha - \abs{\eta}^\alpha\Big\vert
	\leq C_\alpha\,\Big(\abs{\xi-\eta}^\alpha 
		+ \abs{\xi-\eta}^{s}\,\abs{\eta}^{\alpha-s}\Big).
\eeq

If $\alpha \in (0, 1)$, concavity of $t \mapsto t^\alpha$ gives 
$\big\vert\abs{\xi}^\alpha - \abs{\eta}^\alpha\big\vert 
\leq \abs{\xi - \eta}^\alpha$
(see e.g.\ \cite[Lemma~A.1]{GrafakosOh14}), 
so \eqref{eq:KP-multiplier} follows immediately.

Assume now $\alpha \geq 1$. We consider two cases.

\textit{Case 1: $\abs{\xi - \eta} \leq \abs{\eta}$.}
For $a, b \geq 0$ and $\alpha \geq 1$, the mean value theorem gives
\[
	\abs{a^\alpha - b^\alpha}
	\leq \alpha\,\big(a^{\alpha-1} + b^{\alpha-1}\big)\,\abs{a-b}.
\]
Applied to $a = \abs{\xi}$, $b = \abs{\eta}$:
\[
	\Big\vert\abs{\xi}^\alpha - \abs{\eta}^\alpha\Big\vert
	\leq \alpha\,\big(\abs{\xi}^{\alpha-1} + \abs{\eta}^{\alpha-1}\big)\,
		\abs{\xi-\eta}.
\]
Since $\abs{\xi - \eta} \leq \abs{\eta}$, 
we have $\abs{\xi} \leq 2\abs{\eta}$, so 
$\abs{\xi}^{\alpha-1} \leq 2^{\alpha-1}\abs{\eta}^{\alpha-1}$, 
and it suffices to bound $\abs{\eta}^{\alpha-1}\abs{\xi-\eta}$.
We write
\[
	\abs{\eta}^{\alpha-1}\abs{\xi-\eta}
	= \abs{\eta}^{\alpha-s}\cdot
		\underbrace{\abs{\eta}^{s-1}\abs{\xi-\eta}^{1-s}}_{
			= \big(\frac{\abs{\xi-\eta}}{\abs{\eta}}\big)^{1-s} \leq 1}
		\cdot\;\abs{\xi-\eta}^{s}
	\leq \abs{\eta}^{\alpha-s}\abs{\xi-\eta}^{s},
\]
where we used $\abs{\xi-\eta}/\abs{\eta} \leq 1$ and $1-s \geq 0$.

\textit{Case 2: $\abs{\xi - \eta} > \abs{\eta}$.}
Using the triangle inequality 
$\abs{\xi} \leq \abs{\xi-\eta}+\abs{\eta} \leq 2\abs{\xi-\eta}$,
we get $\abs{\xi}^\alpha \leq 2^\alpha\abs{\xi-\eta}^\alpha$.
And since
$\abs{\eta}^\alpha < \abs{\xi-\eta}^\alpha$, we therefore find
\[
	\Big\vert\abs{\xi}^\alpha - \abs{\eta}^\alpha\Big\vert
	\leq \abs{\xi}^\alpha + \abs{\eta}^\alpha
	\leq (2^\alpha + 1)\abs{\xi-\eta}^\alpha.
\]
Combining both cases yields \eqref{eq:KP-multiplier}
with $C_\alpha = 2\alpha\cdot 2^{\alpha-1} + 2^\alpha + 1$.

\medskip
\textit{Step 3: Young's convolution inequality.}
Inserting \eqref{eq:KP-multiplier} into \eqref{eq:KP-fourier}, we find
\bals
	\Big\vert\widehat{[\abs{D}^\alpha, \varphi] f}(\xi)\Big\vert
	&\leq C_\alpha \underbrace{\int_{\R^d} \abs{\xi-\eta}^\alpha\,
		\abs{\hat\varphi(\xi-\eta)}\,\abs{\hat f(\eta)}\dd\eta}_{=:\, I_1(\xi)}\\
	&\quad + C_\alpha \underbrace{\int_{\R^d} \abs{\xi-\eta}^s\,
		\abs{\hat\varphi(\xi-\eta)}\,
		\abs{\eta}^{\alpha-s}\abs{\hat f(\eta)}\dd\eta}_{=:\, I_2(\xi)}.
\eals
Both $I_1$ and $I_2$ are convolutions. By Young's inequality 
$\norm{p\ast q}_{L^2}\leq \norm{p}_{L^1}\norm{q}_{L^2}$:
\[
	\norm{I_1}_{L^2}
	\leq \Big\Vert\abs{\cdot}^\alpha\abs{\hat\varphi}\Big\Vert_{L^1}\,
		\norm{\hat f}_{L^2}
	= \Big\Vert\abs{\cdot}^\alpha\abs{\hat\varphi}\Big\Vert_{L^1}\,
		\norm{f}_{L^2},
\]
and 
\[
	\norm{I_2}_{L^2}
	\leq \Big\Vert\abs{\cdot}^s\,\abs{\hat\varphi}\Big\Vert_{L^1}\,
		\big\Vert\abs{\cdot}^{\alpha-s}\hat f\big\Vert_{L^2}
	= \Big\Vert\abs{\cdot}^s\,\abs{\hat\varphi}\Big\Vert_{L^1}\,
		\norm{f}_{\dot H^{\alpha-s}}.
\]
By Plancherel, $\norm{[\abs{D}^\alpha, \varphi] f}_{L^2}
= \norm{\widehat{[\abs{D}^\alpha, \varphi] f}}_{L^2}
\leq C_\alpha(\norm{I_1}_{L^2} + \norm{I_2}_{L^2})$,
which gives \eqref{eq:KP}.

\medskip
\textit{Step 4: Scaling.}
If $\varphi_\rho(v) = \varphi_1(v/\rho)$, then 
$\hat\varphi_\rho(\xi) = \rho^d\,\hat\varphi_1(\rho\xi)$.
Hence, by the change of variables $z = \rho\xi$,
\[
	\big\Vert\abs{\cdot}^s\,\abs{\hat\varphi_\rho}\big\Vert_{L^1}
	= \int_{\R^d} \abs{\xi}^s\,\rho^d\abs{\hat\varphi_1(\rho\xi)}\dd\xi
	= \rho^{-s}\int_{\R^d} \abs{z}^s\,\abs{\hat\varphi_1(z)}\dd z
	= C(\varphi_1)\,\rho^{-s},
\]
and similarly
$\big\Vert\abs{\cdot}^{\alpha}\abs{\hat\varphi_\rho}\big\Vert_{L^1}
= C(\varphi_1, \alpha)\,\rho^{-\alpha}$.
The integrals $\int \abs{z}^s\abs{\hat\varphi_1(z)}\dd z$ 
and $\int \abs{z}^\alpha\abs{\hat\varphi_1(z)}\dd z$ are finite 
since $\varphi_1 \in C_c^\infty \subset \mathcal S$.
Substituting into \eqref{eq:KP} yields~\eqref{eq:KP-scaling}.
\end{proof}

We now formulate the local version that is needed
in the proof of Proposition~\ref{prop:fractional_energy}.

\begin{corollary}[Local commutator estimate]\label{cor:local-KP}
Let $\alpha \geq s$ and $s \in (0, 1)$.
Let $D \Subset \Omega \subset \R^d$ with 
$$\rho := \frac{1}{4}\operatorname{dist}(D, \R^d\setminus\Omega) > 0.$$
Let $\varphi \in C_c^\infty(D)$ be a cutoff function at scale~$\rho$,
i.e.\ $\varphi(v) = \varphi_1\big(\frac{(v-v_0)}{\rho}\big)$ for some fixed profile
$\varphi_1 \in C_c^\infty(\R^d)$ and centre~$v_0$.
Suppose $f \in H^{\alpha-s}(\Omega) \cap L^\infty(\R^d)$. Then
\beq\label{eq:local-KP}
	\norm{[\abs{D}^\alpha, \varphi] f}_{L^2(\R^d)}
	\leq C\rho^{-s}\norm{f}_{H^{\alpha-s}(\Omega)}
		+ C\rho^{-\alpha}\norm{f}_{L^\infty(\R^d)},
\eeq
where $C = C(d, \alpha, \varphi_1)$.
\end{corollary}

\begin{proof}
\textit{Step 1: Decomposition.}
Choose $\theta_1 \in C_c^\infty(\Omega)$ with $\theta_1 = 1$ on 
$D + B_{2\rho}$ and $\supp(\theta_1) \subset D + B_{3\rho} \subset \Omega$, 
and $\norm{\theta_1}_{C^N} \leq C_N \rho^{-N}$ for all $N \in \N$. 
Since $\supp(\varphi) \subset D$ and $\theta_1 = 1$ on $D$, 
we have $\varphi \theta_1 = \varphi$, hence $\varphi f = \varphi(\theta_1 f)$.
By linearity of the commutator:
\beqs
	[\abs{D}^\alpha, \varphi] f
	= [\abs{D}^\alpha, \varphi](\theta_1 f)
		+ \big[\abs{D}^\alpha(\varphi(1-\theta_1)f) 
			- \varphi\abs{D}^\alpha\big((1-\theta_1)f\big)\big].
\eeqs
Since $\varphi(1-\theta_1) = 0$, the first term in the square bracket vanishes,
giving
\beq\label{eq:local-KP-split2}
	[\abs{D}^\alpha, \varphi] f
	= [\abs{D}^\alpha, \varphi](\theta_1 f)
		- \varphi\,\abs{D}^\alpha\big((1-\theta_1)f\big).
\eeq

\medskip
\textit{Step 2: Global estimate applied to $\theta_1 f$.}
Since $\theta_1 \in C_c^\infty(\Omega)$ and $f \in H^{\alpha-1}(\Omega)$,
the product $\theta_1 f$, extended by zero outside $\Omega$,
belongs to $H^{\alpha-1}(\R^d)$: indeed $\theta_1 f$ vanishes 
in a neighbourhood of $\partial \Omega$, so the zero extension 
preserves $H^{\alpha-1}$ regularity.
More precisely, the fractional Leibniz rule, Lemma \ref{lem:product-cutoff}, gives
\beq\label{eq:product-bound}
	\norm{\theta_1 f}_{H^{\alpha-s}(\R^d)}
	\leq C\Big(\rho^{-\alpha+s}\norm{f}_{L^2(\Omega)} 
		+ \norm{f}_{H^{\alpha-s}(\Omega)}\Big).
\eeq
Applying Proposition~\ref{prop:kato-ponce}, equation \eqref{eq:KP-scaling}:
\bals
	\norm{[\abs{D}^\alpha, \varphi](\theta_1 f)}_{L^2}
	\leq C\rho^{-s}\norm{\theta_1 f}_{\dot H^{\alpha-s}(\R^d)}
		+ C\rho^{-\alpha}\norm{\theta_1 f}_{L^2(\R^d)}\leq C\rho^{-s}\norm{f}_{H^{\alpha-s}(\Omega)}
		+ C\rho^{-\alpha}\norm{f}_{L^2(\Omega)}.
\eals

\medskip
\textit{Step 3: Tail estimate.}
It remains to bound 
$\varphi\,\abs{D}^\alpha\big((1-\theta_1)f\big)$
on $\supp(\varphi) \subset D$. 
Let $v \in D$ and set $g := (1-\theta_1)f$. 
By construction, $g = 0$ on $D + B_{2\rho}$, so in particular 
$g = 0$ on $B_{2\rho}(v)$ for all $v \in D$.

We claim that 
\beq\label{eq:tail-claim}
	\abs{\abs{D}^\alpha g(v)} \leq C \rho^{-\alpha}\norm{g}_{L^\infty(\R^d)},
	\qquad v \in D.
\eeq
From \eqref{eq:tail-claim}, we conclude
$\norm{\varphi\abs{D}^\alpha g}_{L^2} 
\leq C\rho^{-\alpha}\norm{\varphi}_{L^\infty}\norm{f}_{L^\infty(\R^d)}$.

The claim follows, since for $\alpha \notin 2\N_0$, the operator $\abs{D}^\alpha$ 
has Schwartz kernel 
$K_\alpha(z) = c_{d,\alpha}\abs{z}^{-(d+\alpha)}$
away from the origin.
Since $g$ vanishes on $B_{2\rho}(v)$, there is no singularity at $z = 0$
and the integral converges absolutely:
\[
	\abs{D}^\alpha g(v)
	= c_{d,\alpha} \int_{\abs{v-w} > 2\rho} 
		\frac{g(w)}{\abs{v-w}^{d+\alpha}}\dd w.
\]
Thus,
\[
	\norm{\varphi\abs{D}^\alpha g}_{L^2(D)}
	\leq C\rho^{-\alpha}\norm{\varphi}_{L^\infty}\norm{g}_{L^\infty}.
\]
For $\alpha \in 2\N_0$, $\abs{D}^\alpha = (-\Delta)^{\alpha/2}$ 
is a local differential operator, so $\abs{D}^\alpha g(v) = 0$ for 
$v \in D$ since $g$ vanishes on $D + B_{2\rho}$. 
Hence \eqref{eq:tail-claim} holds trivially.

Combining Steps~2 and~3 with \eqref{eq:local-KP-split2}
gives \eqref{eq:local-KP}.
\end{proof}
\begin{proof}[Proof of \eqref{eq:h2-gen}]
As a consequence of  \eqref{eq:local-KP}, we conclude the proof of \eqref{eq:h2-gen} in Proposition \ref{lem:commutator} for the case that $k > 2$ and $s \in \big[\frac{1}{k}, 1\big)$, by picking $\alpha = ks \geq 1$.
\end{proof}


\section{Toolbox}\label{sec:toolbox}
Campanato's approach is a scaling argument, consisting of a clever combination of several tools that permit to gain as much regularity as can be gained from the equation. In short, we combine Poincaré's inequality with Sobolev embedding, and close the argument with regularity estimates. In this section we assemble the tools that are used in both the non-fractional and the fractional case. 

\subsection{Functional inequalities}

Similar to the elliptic case in \cite{Giaquinta}, for $f \in W^{m, p}(Q_R(z_0))$ there exists a unique polynomial $p_{m-1} = p_{m-1}(z_0, R, f, z)$ of degree less or equal to $m-1$ so that 
\bal
	\fint_{Q_R(z_0)} D^\phi(f - p_{m-1})\dd z = 0 \qquad \forall \Phi \textrm{ with } \abs \Phi \leq m-1.
\label{eq:giaq1}
\eal
Here $m \in \N + 2s\N$ and $D^\phi$ is a kinetic differential whose order is in the discrete set $\N + 2s\N$ as well. The polynomial is given by 
\beqs
	p_{m-1}(z) = \sum_{\psi \in \N^{1+2d}, \abs \Psi \leq m-1} \frac{c_\psi}{\psi!} (z - z_0)^\psi
\eeqs
with 
\beqs
	c_\psi = \sum_{\substack{\phi \in \N^{1+2d}, \\2\abs \Phi\leq m-1-\abs{\Psi}}} c_{\psi, \phi}R^{-n+2\abs\phi}\int_{Q_R(z_0)}D^{\psi+2\phi}f \dd z,
\eeqs
where $n = 2s + 2d(s+1)$. Recall that for $\psi = (\psi_0, \Psi_1, \Psi_2) \in \N^{1+2d}$ we denote by $\abs{\Psi}$ the size of $\psi$ respecting the scaling, i.e. $\abs{\Psi} = 2s\cdot\psi_0 + (1+2s)\abs{\Psi_1} + \abs{\Psi_2}$. Here $\psi!$ denotes the element-wise operation $\psi! = \psi_0!\psi_1! \cdots \psi_{2d}!$. 

The idea is to use \eqref{eq:giaq1} in order to apply the standard Poincaré-inequality \cite[Prop 3.12]{GiaquintaMartinazzi} to $D^\phi(f - p_{m-1})$ for $\abs \Phi = 0, \dots, m-1$. Moreover, we have for any non-negative function $f \in L^2(Q_r(z_0))$ 
\beq
	\int_{Q_r(z_0)} f^2 \dd z \leq Cr^n \norm{f}^2_{L^\infty(Q_r(z_0))},
\label{eq:campnonloc1}
\eeq
where $r > 0 $ and $n = 2s + 2d(s+1)$.
Combined with Sobolev's embedding and regularity estimates, we obtain an estimate commonly referred to as Campanato's (first) inequality, which will be the first tool to tackle the Schauder estimates. For reference, in the elliptic case, Campanato's first inequality reads
\beqs
	\int_{B_r} \abs{u}^2 \dd x \leq C \Big(\frac{r}{R}\Big)^d \int_{B_R} \abs{u}^2 \dd x,
\eeqs
for a solution $u: \R^d \to \R$ of a second order elliptic equation, see \cite[Section 5]{GiaquintaMartinazzi}.

\subsection{Regularity estimates}
The second key step are regularity estimates for the constant coefficient equation. 
We consider solutions $f$ of the constant coefficient Kolmogorov equation
\beq\label{eq:kolmogorov_energy_loc}
	\partial_t f + v \cdot \nabla_x f - A^0\Delta_v f = h
\eeq
in $Q_R(z_0)$ for some $z_0 \in \R^{1+2d}$ and $R > 0$. Here $A^0$ is some constant such that $A^0 \geq \lambda_0$ with $\lambda_0$ from \eqref{eq:unifellip}.
The fractional analogue reads
\beq\label{eq:kolmogorov_energy}
	\partial_t f + v \cdot \nabla_x f +\mathcal L_0f = h,
\eeq
where $\mathcal L_0$ is the non-local operator \eqref{eq:1.2} associated to a non-negative, translation-invariant kernel $K_0$ such that
\beq\label{eq:cc_kernel}
	\frac{\lambda_0}{\abs{w}^{d+2s}} \leq K_0(w) \leq \frac{\Lambda_0}{\abs{w}^{d+2s}},
\eeq
and $K_0(w) = K_0(-w)$ is independent of $z$.
We derive inductive regularity estimates relying on Bouchut's Proposition \ref{prop:bouchut}, which captures the regularising effect of the transport operator in the space variable. 
For the sake of brevity we 
introduce $\abs{D_\omega}^\gamma = (-\Delta_\omega)^{\frac{\gamma}{2}}$ in $\R^m_w$ with $(m,\omega) \in\big\{ (1,t),  (d,x),  (d, v)\big\}$.
\begin{proposition}[Local (non-fractional) regularity estimates]\label{prop:nonfractional_energy}
Let $f$ be a non-negative solution in $Q_R(z_0)$ of \eqref{eq:kolmogorov_energy_loc} with $s = 1$. Let $l \in \N_0$, $0 < r < R \leq 1$ and write $\delta := R-r > 0$.
Let $D^l$ be a pseudo-differential operator of order $l \geq 0$, equal to one of the following
\[
	D^{l} \in \Big\{\abs{D_t}^{\frac{l}{2}}, \abs{D_x}^{\frac{l}{3}}, \abs{D_v}^{l}, \abs{D_t}^{\frac{l+1}{2}}\abs{D_v}, \abs{D_v}^{(j+1)}\abs{D_x}^{\frac{l + 1- j}{3}}\Big\}, \quad j \in \{0, \dots, l\}.
\]
Then we have
\bals
	\Norm{D^{l+2} f}_{L^2(Q_r(z_0))} \leq C \delta^{-(l+2)} \norm{f}_{L^2(Q_R(z_0))} + C\delta^{-l}  \norm{D^l h}_{L^2(Q_R(z_0))},
\eals
where $C = C(n, \lambda_0)$. In particular, if $h = 0$, then
\bals
	\norm{\abs{D_v}^{l+2} f}_{L^2(Q_r(z_0))} &+\norm{\abs{D_t}^{\frac{l+2}{2}} f}_{L^2(Q_r(z_0))}  + \norm{\abs{D_x}^{\frac{l+2}{3}} f}_{L^2(Q_r(z_0))} \lesssim   \delta^{-(l+2)} \norm{f}_{L^2(Q_R(z_0))}.
\eals
\end{proposition}
For the fractional case, the right hand side involves a norm on the whole velocity space.
\begin{proposition}[Non-local (fractional) regularity estimates]\label{prop:fractional_energy}
Let $l \in \N_0$, $0 < r < R \leq 1$ and write $\delta = R-r > 0$. Let $Q_R(z_0)$ be the kinetic cylinder defined in \eqref{eq:cylinder} and write $Q_R(z_0) =: \mathcal I \times \Omega_x \times \Omega_v$. 
Suppose $f \in L^\infty(\R^{1+2d})$ is a non-negative solution in $Q_R(z_0)$ of \eqref{eq:kolmogorov_energy} with $s \in (0, 1)$.  
Let $D^{ls}$ be a pseudo-differential operator of order $ls \geq 0$, 
equal to one of the following
\[
	D^{ls} \in \Big\{\abs{D_t}^{\frac{ls}{2s}}, \abs{D_x}^{\frac{ls}{1+2s}}, \abs{D_v}^{ls}, \abs{D_t}^{\frac{l+1}{2}}\abs{D_v}^s, \abs{D_v}^{(j+1)s}\abs{D_x}^{\frac{ls + s- js}{1+2s}}\Big\}, \quad j \in \{0, \dots, l\}.
\]
Then there holds
\bal\label{eq:frac-regularity-estimate}
	\norm{D^{(l+2)s} f}_{L^2(Q_r(z_0))} \leq C \delta^{-(l+2)s} \norm{f}_{L^\infty( \R^{1+2d})}+ C\delta^{-ls}\left( \norm{D^{ls} h}_{L^2(Q_R(z_0))} + \norm{h}_{L^\infty(\R^{1+2d})}\right),
\eal
where $C = C(n, s, \Lambda_0, \lambda_0)$.
\end{proposition}
\begin{remark}\begin{enumerate}[i.]
\item The proof of Proposition \ref{prop:nonfractional_energy} is similar to the proof of Proposition \ref{prop:fractional_energy}. In fact, for Step 1 in the proof Proposition \ref{prop:fractional_energy}, we can just set $s = 1$ and obtain the global version of the energy estimate for the non-fractional case. Steps 2, 3 and 4 are much simpler for the non-fractional case: it suffices to localise with some smooth cut-off $\theta \in C^\infty_c(Q_R(z_0))$, and then consider the equation solved by $g := f \theta$. Since the equation solved by $f$ is non-fractional, $g$ solves an equation with a right hand side that can be bounded by $\norm{f}_{L^2(Q_R(z_0))}$ using the induction hypothesis. Since this case is comparatively simpler, we will focus on the proof of Proposition \ref{prop:fractional_energy}. 
\item With slightly more work, we would possibly also be able to deduce a similar result for a general diffusion coefficient that is uniformly elliptic and satisfies $D^l A \in L^2(Q_R(z_0))$ with $l \in \N_0$ as in the statement. For our purposes, the constant coefficient case suffices.
\end{enumerate}
\end{remark}
The proof builds upon the work of Alexandre and Bouchut \cite{Alexandre, Bouchut}. In particular, we will make use of the following proposition\cite[Proposition 1.1]{Bouchut}.
\begin{proposition}[Bouchut]\label{prop:bouchut}
Assume that $f, S \in L^2(\R^{1+2d})$ satisfy
\beq\label{eq:transport}
	\partial_t f + v \cdot \nabla_x f = S,
\eeq
and $\abs{D_v}^{\beta}f \in L^2(\R^{1+2d})$ for some $\beta \geq 0$. Then $\abs{D_x}^{\frac{\beta}{1+\beta}} f \in L^2(\R^{1+2d})$, and
\beq\label{eq:bouchut}
	\norm{\abs{D_x}^{\frac{\beta}{1+\beta}} f}_{L^2(\R^{1+2d})} \leq C \norm{\abs{D_v}^\beta f}_{L^2(\R^{1+2d})}^{\frac{1}{1+\beta}}\norm{S}_{L^2(\R^{1+2d})}^{\frac{\beta}{1+\beta}}, 
\eeq
for some universal constant $C > 0$.
\end{proposition}
We recall the proof of Proposition \ref{prop:bouchut} in Appendix \ref{app:bouchut}. 

We also need a direct consequence of the commutator estimates of the previous Subsection \ref{sec:comm}, which we record in the following corollary in the form that we apply it in the proof of Proposition \ref{prop:fractional_energy}. 


\begin{corollary}[Localised commutator estimates]\label{cor:comm-localisation}
Let $l \in \N_0$, $s \in (0, 1)$, and let $\theta \in C_c^\infty(\Omega_v)$  
be a cutoff function in~$v$  
with $\norm{D^k\theta}_{L^\infty}\leq C\rho^{-k}$,
where $\rho = \frac{(R-r)}{C_0}$ and $C_0$ is a sufficiently large 
universal geometric constant.
Let $D = \supp(\theta)$ and let $\Omega_v$ be a bounded Lipschitz domain
with $D \Subset \Omega_v$ and 
$\dist(D, \R^d\setminus\Omega_v) \geq C_0\rho$.
Let $f \in H^{(l+1)s}_v(\Omega_v) \cap L^\infty(\R^d)$. Then 
\beq\label{eq:Dls-commutator}
	\norm{\abs{D_v}^{ls}\big([\abs{D_v}^{2s}, \theta]\,f\big)}_{L^2(\R^d)}
	\leq C\rho^{-s}\norm{f}_{H^{(l+1)s}(\Omega_v)}
		+ C\rho^{-(l+2)s}\norm{f}_{L^\infty(\R^d)},
\eeq
where $C = C(d, s, l)$.
\end{corollary}

\begin{proof}
Choose $\theta_1 \in C_c^\infty(\Omega_v)$ with 
$\theta_1 = 1$ on $D + B_{2\rho}$, 
$\supp(\theta_1) \subset D + B_{3\rho} \subset \Omega_v$, 
and $\norm{D^k\theta_1}_{L^\infty} \leq C\rho^{-k}$.

\medskip
\textit{Step 1: Telescoping identity.}
We write
\beq\label{eq:telescope}
	\abs{D_v}^{ls}\big([\abs{D_v}^{2s}, \theta]f\big)
	= \underbrace{[\abs{D_v}^{(l+2)s}, \theta]f}_{\mathrm{I}}
	\;-\; \underbrace{[\abs{D_v}^{ls}, \theta](\abs{D_v}^{2s}f)}_{\mathrm{II}}.
\eeq
This identity follows from expanding both sides:
\bals
	[\abs{D_v}^{(l+2)s}, \theta]f 
	&= \abs{D_v}^{(l+2)s}(\theta f) - \theta \abs{D_v}^{(l+2)s} f
	= \abs{D_v}^{ls}\abs{D_v}^{2s}(\theta f) - \theta\abs{D_v}^{ls}\abs{D_v}^{2s} f,\\
	[\abs{D_v}^{ls}, \theta](\abs{D_v}^{2s}f) 
	&= \abs{D_v}^{ls}(\theta \abs{D_v}^{2s} f) - \theta\abs{D_v}^{ls}\abs{D_v}^{2s} f.
\eals
Thus,
\[
	[\abs{D_v}^{(l+2)s}, \theta]f - [\abs{D_v}^{ls}, \theta](\abs{D_v}^{2s}f)
	= \abs{D_v}^{ls}\big(\abs{D_v}^{2s}(\theta f) - \theta \abs{D_v}^{2s} f\big)
	= \abs{D_v}^{ls}\big(\big[\abs{D_v}^{2s}, \theta\big] f\big).
\]

\medskip
\textit{Step 2: Bound on Term~$\mathrm{I}$.}
We apply Proposition \ref{lem:commutator} with $k = (l+2)$. 
If $(l+2)s < 1$, we use \eqref{eq:h1} and \eqref{eq:h2}. If $(l+2)s \geq 1$, we use \eqref{eq:h2-gen}. This yields
\beq\label{eq:I-bound}
	\norm{\mathrm{I}}_{L^2}
	\leq C\rho^{-s}\norm{f}_{H^{(l+1)s}(\Omega_v)}
	+ C\rho^{-(l+2)s}\norm{f}_{L^\infty(\R^d)}.
\eeq

\medskip
\textit{Step 3: Bound on Term~$\mathrm{II}$.}
Let $g := \abs{D_v}^{2s}f$.
Since $\theta(1-\theta_1) = 0$, 
the decomposition~\eqref{eq:local-KP-split2} gives
\beq\label{eq:II-decomp}
	\mathrm{II} = [\abs{D_v}^{ls}, \theta]g
	= \underbrace{[\abs{D_v}^{ls}, \theta](\theta_1 g)}_{\mathrm{II}_a}
	\;-\; \underbrace{\theta\,\abs{D_v}^{ls}\big((1-\theta_1) g\big)}_{\mathrm{II}_b}.
\eeq
Since $\theta_1$ is supported in $\Omega_v$, the function $\theta_1 \abs{D_v}^{2s}f$ is compactly supported in $\Omega_v$. By Lemma \ref{lem:local-nonlocal}, $\abs{D_v}^{2s}f$ belongs to $H^{(l-1)s}(\Omega_v)$, and thus by Lemma ~\ref{lem:product-cutoff}, $\theta_1 \abs{D_v}^{2s}f$ belongs to $H^{(l-1)s}(\R^d)$. Therefore, 
\beq\label{eq:theta1F-bound}
	\norm{\theta_1 \abs{D_v}^{2s}f }_{H^{(l-1)s}(\R^d)}
	\leq C\norm{\abs{D_v}^{2s}f }_{H^{(l-1)s}(\Omega_v)}
		+ C\rho^{-(l-1)s}\norm{\abs{D_v}^{2s}f }_{L^2(\Omega_v)}.
\eeq
In particular, $\theta_1 \abs{D_v}^{2s}f  \in L^2(\R^d)$ with 
$\norm{\theta_1 \abs{D_v}^{2s}f }_{L^2} \leq C\norm{f }_{H^{2s}(\Omega_v)} + C \norm{f}_{L^\infty(\R^d)}$.

\textit{Case $ls \geq 1$.}
By Proposition \ref{prop:kato-ponce}, and \eqref{eq:theta1F-bound}, we get
\bal\label{eq:IIa-bound}
	\norm{\mathrm{II}_a}_{L^2}
	&\leq C\rho^{-s}\norm{\theta_1 \abs{D_v}^{2s}f }_{H^{(l-1)s}(\R^d)}
		+ C\rho^{-ls}\norm{\theta_1 \abs{D_v}^{2s}f }_{L^2(\R^d)}\\
	&\leq C\rho^{-s}\norm{\theta_1 \abs{D_v}^{2s}f }_{H^{(l-1)s}(\R^d)}
		+ C\rho^{-ls}\norm{\abs{D_v}^{2s}f }_{L^2(\Omega_v)}\\
	&\leq C\rho^{-s}\norm{f }_{H^{(l+1)s}(\Omega_v)}
		+ C\rho^{-(l+2)s}\norm{f }_{L^2(\Omega_v)}+ C\rho^{-(l+2)s}\norm{f }_{L^\infty(\R^d)}.
\eal

\textit{Case $ls < 1$.}
Proposition~\ref{lem:commutator} with $k = l$ gives
\[
	\norm{\mathrm{II}_a}_{L^2}
	\leq C\rho^{-ls}\norm{\theta_1 \abs{D_v}^{2s}f}_{L^2}
	\leq C\rho^{-ls}\norm{\abs{D_v}^{2s}f}_{L^2(\Omega_v)},
\]
which is bounded by the right-hand side 
of~\eqref{eq:IIa-bound}.

For $\mathrm{II}_b$, since $(1-\theta_1)$ vanishes on $D + B_{2\rho}$ 
and $v \in D = \supp(\theta)$, 
for $ls \notin 2\N_0$ the kernel representation gives
\[
	\abs{D_v}^{ls}\big((1-\theta_1)g\big)(v)
	= c_{d,ls}\int_{\abs{v-w} > 2\rho} 
		\frac{(1-\theta_1)(w)\,g(w)}{\abs{v-w}^{d+ls}}\dd w.
\]
For fixed $v \in D$, define the test kernel
\[
	\psi_v(w) := c_{d,ls}\,(1-\theta_1)(w)\,
		\frac{\chi_{\abs{v-w}>2\rho}}{\abs{v-w}^{d+ls}}.
\]
Then $\psi_v$ is smooth away from $v$,
vanishes on $B_{2\rho}(v)$,
and satisfies $\psi_v \in L^1(\R^d) \cap L^2(\R^d)$
with $\norm{\psi_v}_{L^1} \leq C\rho^{-ls}$.
By self-adjointness of $\abs{D_v}^{2s}$:
\beq\label{eq:duality-step}
	\abs{D_v}^{ls}\big((1-\theta_1)g\big)(v)
	= \langle g, \psi_v\rangle
	= \langle \abs{D_v}^{2s}f, \psi_v\rangle
	= \langle f, \abs{D_v}^{2s}\psi_v\rangle.
\eeq

It remains to bound 
$\norm{\abs{D_v}^{2s}\psi_v}_{L^1(\R^d)}$.
By definition, $\psi_v$ is supported in $\R^d \setminus B_{2\rho}(v)$ 
and satisfies $\abs{\psi_v(w)} \leq C\abs{v-w}^{-(d+ls)}$ 
for $\abs{v-w} > 2\rho$. Since $(1-\theta_1)(w)$ vanishes identically 
on $D + B_{2\rho} \supset B_{2\rho}(v)$ for $v \in D$, $\psi_v$ is smooth on $\R^d$.
Hence the product 
$(1-\theta_1)(w)\chi_{\abs{v-w}>2\rho}$ is $C^\infty$ a.e.
and $\psi_v \in C^\infty(\R^d)$.

Define the rescaled function
\beqs
	\tilde\psi(\tilde w) := \rho^{d+ls}\,\psi_v(v + \rho\tilde w),
	\qquad \tilde w \in \R^d.
\eeqs
Then $\tilde\psi$ is smooth on $\R^d$, vanishes on $B_2(0)$, 
satisfies $\abs{\tilde\psi(\tilde w)} \leq C\abs{\tilde w}^{-(d+ls)}$
for $\abs{\tilde w} > 2$, and its shape is independent of~$\rho$. In particular,
\beqs
	\norm{\tilde\psi}_{L^1(\R^d)} 
	= \rho^{ls}\norm{\psi_v}_{L^1} = C,
\eeqs
with $C$ independent of $\rho$.

Since $\abs{D}^{2s}$ scales homogeneously:
\beqs
	(\abs{D}^{2s}\psi_v)(v + \rho\tilde w) 
	= \rho^{-(d+ls+2s)}\,(\abs{D}^{2s}\tilde\psi)(\tilde w),
\eeqs
so that, by the change of variables $u = v + \rho\tilde u$:
\beq\label{eq:L1-scaling}
	\norm{\abs{D}^{2s}\psi_v}_{L^1(\R^d)}
	= \rho^{-(ls+2s)}\norm{\abs{D}^{2s}\tilde\psi}_{L^1(\R^d)}.
\eeq
To see that
$\norm{\abs{D}^{2s}\tilde\psi}_{L^1(\R^d)} < \infty$, 
we use that $\tilde\psi$ is smooth almost everywhere, it
vanishes on $B_2(0)$ and decays like $\abs{\tilde w}^{-(d+ls)}$
at infinity, with $ls > 0$. More precisely, we split:
\[
	\norm{\abs{D}^{2s}\tilde\psi}_{L^1}
	\leq \norm{\abs{D}^{2s}\tilde\psi}_{L^1(B_R)}
		+ \norm{\abs{D}^{2s}\tilde\psi}_{L^1(\R^d\setminus B_R)}
\]
for some fixed $R > 4$.
We bound the first term as
$\norm{\abs{D}^{2s}\tilde\psi}_{L^1(B_R)}
\leq C_R\norm{\abs{D}^{2s}\tilde\psi}_{L^2(\R^d)}
= C_R\norm{\tilde\psi}_{\dot H^{2s}(\R^d)}$,
which is finite since $\tilde\psi$ is smooth and in $L^2(\R^d)$ 
due to its decay $\abs{\tilde\psi(\tilde w)} 
\leq C\abs{\tilde w}^{-(d+ls)}$ with $ls > 0$.

For the second term, we use the second-order difference 
representation:
for $\abs{\tilde u} > R$ and $\abs{z} < \frac{R}{4}$, 
both $\tilde u + z$ and $\tilde u - z$ lie outside $B_{R/2}$,
where $\tilde\psi$ is smooth with 
$\abs{D^2\tilde\psi(\tilde w)} \leq C\abs{\tilde w}^{-(d+ls+2)}$.
Taylor expanding:
$\abs{2\tilde\psi(\tilde u) - \tilde\psi(\tilde u+z) - \tilde\psi(\tilde u-z)}
\leq C\abs{z}^2\abs{\tilde u}^{-(d+ls+2)}$,
so the $\abs{z} < \frac{R}{4}$ contribution gives
\[
	\int_{\abs{z}<\frac{R}{4}} 
		\frac{C\abs{z}^2\abs{\tilde u}^{-(d+ls+2)}}{\abs{z}^{d+2s}}\dd z
	= C\abs{\tilde u}^{-(d+ls+2)}
		\int_{\abs{z}<\frac{R}{4}}\abs{z}^{2-d-2s}\dd z
	\leq C\abs{\tilde u}^{-(d+ls+2)}.
\]
The $\abs{z} \geq \frac{R}{4}$ contribution is bounded by 
$C\abs{\tilde u}^{-(d+ls)}\int_{\abs{z}\geq R/4}\abs{z}^{-(d+2s)}\dd z
\leq C\abs{\tilde u}^{-(d+ls)}$.
Integrating over $\abs{\tilde u} > R$:
\[
	\int_{\abs{\tilde u}>R} 
		{\abs{D}^{2s}\tilde\psi(\tilde u)}\dd\tilde u
	\leq C\int_{\abs{\tilde u}>R} \abs{\tilde u}^{-(d+ls)}\dd\tilde u
	< \infty,
\]
since $ls > 0$. Therefore 
$\norm{\abs{D}^{2s}\tilde\psi}_{L^1(\R^d)} \leq C < \infty$,
with $C$ depending only on $d$, $s$, $l$, and the fixed profiles.

\textit{Step 4: Conclusion.}
Combined with~\eqref{eq:L1-scaling}, we find
\beqs
	\norm{\abs{D_v}^{2s}\psi_v}_{L^1(\R^d)} 
	\leq C\rho^{-(l+2)s}.
\eeqs
Therefore, for each $v \in D$, 
\[
	{\abs{D_v}^{ls}\big((1-\theta_1)g\big)(v)}
	= \abs{\langle f, \abs{D_v}^{2s}\psi_v\rangle}
	\leq \norm{f}_{L^\infty}\,\norm{\abs{D_v}^{2s}\psi_v}_{L^1}
	\leq C\rho^{-(l+2)s}\norm{f}_{L^\infty},
\]
so that when we integrate over $v \in D$, we end up with
\beq\label{eq:IIb-bound}
	\norm{\mathrm{II}_b}_{L^2} 
	= \norm{\theta\,\abs{D_v}^{ls}\big((1-\theta_1)g\big)}_{L^2}
	\leq C\rho^{-(l+2)s}\norm{f}_{L^\infty(\R^d)}.
\eeq
(For $ls \in 2\N_0$: $\abs{D_v}^{ls}$ is a local operator, 
so $\abs{D_v}^{ls}\big((1-\theta_1)g\big)(v) = 0$ for $v \in D$,
and \eqref{eq:IIb-bound} holds trivially.)

\medskip
\textit{Conclusion.}
Combining \eqref{eq:telescope}, \eqref{eq:I-bound}, 
\eqref{eq:II-decomp}, \eqref{eq:IIa-bound}, 
and~\eqref{eq:IIb-bound}:
\[
	\norm{\abs{D_v}^{ls}\big([\abs{D_v}^{2s}, \theta]f\big)}_{L^2}
	\leq C\rho^{-s}\norm{f}_{H^{(l+1)s}(\Omega_v)}
		+ C\rho^{-(l+2)s}\norm{f}_{L^\infty(\R^d)}
		+ C\rho^{-ls}\norm{f}_{L^2(\Omega_v)}.
		\qedhere
\]
\end{proof}

We can now prove Proposition \ref{prop:fractional_energy}. 
\begin{proof}[Proof of Proposition \ref{prop:fractional_energy}]
With no loss of generality, we assume $A^0 = 1$ and $K_0(w) = \frac{1}{\abs{w}^{d+2s}}$ (otherwise we can either perform a constant change of variable or just use the pointwise bounds on the kernel). We start with global estimates, and then we localise the result.

We will also assume for now that $f$ is sufficiently smooth, so that all the computations below are justified. To make it work for general $f \in L^2$, we will use a density argument.

\textit{Step 1: Global estimate}. Assume for now that $f$ solves \eqref{eq:kolmogorov_energy} on $\R^{1+2d}$ with a source term $h \in L^2(\R^{1+2d})$, that is 
\beq\label{eq:kolmogorov_global}
	\mathcal T f + \abs{D_v}^{2s}f = h.
\eeq
To prove the global statement \eqref{eq:frac-regularity-estimate} in its full generality, we will need to assume that $\abs{D_v}^{ls} h, \abs{D_x}^{\frac{ls}{1+2s}} h, \abs{D_t}^{\frac{ls}{2s}} h\in L^2(\R^{1+2d})$. 

First note that testing \eqref{eq:kolmogorov_global} with $f$ yields
\beqs
	\norm{\abs{D_v}^s f}^2_{L^2(\R^{1+2d})} \leq \norm{h}_{L^2(\R^{1+2d})}\norm{f}_{L^2(\R^{1+2d})}. 
\eeqs

Second, we note that any solution $f$ of \eqref{eq:kolmogorov_global} satisfies 
\beqs
	\mathcal T \big(\abs{D_x}^{\frac{ls}{1+2s}} f\big) = - \abs{D_v}^{2s} \abs{D_x}^{\frac{ls}{1+2s}} f + \abs{D_x}^{\frac{ls}{1+2s}} h.
\eeqs
Then Bouchut's Proposition \ref{prop:bouchut} applied to $\abs{D_x}^{\frac{ls}{1+2s}} f$ yields for $\beta = 2s\geq 0$
\beqs
	\norm{\abs{D_x}^{\frac{(l+2)s}{1+2s}} f}_{L^2(\R^{1+2d})} \lesssim \norm{\abs{D_v}^{2s}\abs{D_x}^{\frac{ls}{1+2s}} f}_{L^2(\R^{1+2d})} + \norm{\abs{D_v}^{2s} \abs{D_x}^{\frac{ls}{1+2s}} f}_{L^2(\R^{1+2d})}^{\frac{1}{1+2s}}\norm{\abs{D_x}^{\frac{ls}{1+2s}} h}^{\frac{2s}{1+2s}}_{L^2(\R^{1+2d})}.
\eeqs
Now we use Hölder's inequality in Fourier variables $(k, \xi)$ of $(x, v)$ to bound
\beqs
	\norm{\abs{D_v}^{2s}\abs{D_x}^{\frac{ls}{1+2s}} f}_{L^2} = \norm{\abs{D_x}^{\frac{\theta \cdot (l+2)s}{1+2s}}\abs{D_v}^{(1-\theta) \cdot (l+2)s} f}_{L^2} \leq \norm{\abs{D_x}^{\frac{(l+2)s}{1+2s}} f}^{\theta}_{L^2}  \norm{\abs{D_v}^{(l+2)s} f}^{1-\theta}_{L^2},
\eeqs
where $\theta = \frac{l}{l+2}$. Thus 
\bals
	\norm{\abs{D_x}^{\frac{(l+2)s}{1+2s}} f}_{L^2} &\lesssim \norm{\abs{D_x}^{\frac{(l+2)s}{1+2s}} f}^{\theta}_{L^2}  \Norm{\abs{D_v}^{(l+2)s} f}^{1-\theta}_{L^2} \\
	&\quad+ \norm{\abs{D_x}^{\frac{(l+2)s}{1+2s}} f}^{\frac{\theta}{1+2s}}_{L^2}  \norm{\abs{D_v}^{(l+2)s} f}_{L^2}^{\frac{1-\theta}{1+2s}}\norm{\abs{D_x}^{\frac{ls}{1+2s}} h}^{\frac{2s}{1+2s}}_{L^2},
\eals
from which we deduce by dividing by $\norm{\abs{D_x}^{\frac{(l+2)s}{1+2s}} f}^{\frac{\theta}{1+2s}}$ and using Hölder for some $\varepsilon \in (0, 1)$
\bals
	\norm{\abs{D_x}^{\frac{(l+2)s}{1+2s}} f}_{L^2} &\lesssim  \norm{\abs{D_x}^{\frac{(l+2)s}{1+2s}} f}_{L^2}^{\frac{2s\theta}{1+2s-\theta}} \norm{\abs{D_v}^{(l+2)s} f}^{\frac{(1-\theta)(1+2s)}{1+2s-\theta}}_{L^2}\\
	&\quad + \norm{\abs{D_v}^{(l+2)s} f}^{\frac{1-\theta}{1+2s-\theta}}_{L^2}\norm{\abs{D_x}^{\frac{ls}{1+2s}} h}^{\frac{2s}{1+2s-\theta}}_{L^2} \\
	&\leq \varepsilon^{\frac{1+2s-\theta}{2s\theta}}\norm{\abs{D_x}^{\frac{(l+2)s}{1+2s}} f}_{L^2} + C_\varepsilon \norm{\abs{D_v}^{(l+2)s} f}_{L^2} \\
	&\quad+ \norm{\abs{D_v}^{(l+2)s} f}_{L^2}^{\frac{1-\theta}{1+2s-\theta}}\norm{\abs{D_x}^{\frac{ls}{1+2s}} h}^{\frac{2s}{1+2s-\theta}}_{L^2}.
\eals
Thus absorbing the first term on the right hand side to the left hand side and using $\theta = \frac{l}{l+2}$ we have
\beqs
	\norm{\abs{D_x}^{\frac{(l+2)s}{1+2s}} f}_{L^2}\lesssim \norm{\abs{D_v}^{(l+2)s} f}_{L^2} + \norm{\abs{D_v}^{(l+2)s} f}^{\frac{1}{1+sl + 2s}}_{L^2} \norm{\abs{D_x}^{\frac{ls}{1+2s}} h}^{\frac{s(l+2)}{1+2s+sl}}_{L^2}.
\eeqs

Third, we test \eqref{eq:kolmogorov_global} with $\abs{D_x}^{\frac{(l+1)2s}{1+2s}}f$. Then
\beqs
	\norm{\abs{D_v}^{s}D_x^{\frac{(l+1)s}{1+2s}} f}_{L^2} \leq \norm{\abs{D_x}^{\frac{l s}{1+2s}} h}_{L^2}^{\frac{1}{2}}\norm{\abs{D_x}^{\frac{(l+2)s}{1+2s}} f}^{\frac{1}{2}}_{L^2}.
\eeqs


Since we will use these three observations to proceed, we collect them here:
\begin{itemize}
\item There holds 
\beqs
	\norm{\abs{D_v}^s f}^2_{L^2} \leq \norm{h}_{L^2}\norm{f}_{L^2}.
\eeqs

\item Moreover,  \beq\label{eq:step2_l}
	\norm{\abs{D_x}^{\frac{(l+2)s}{1+2s}} f}_{L^2} \lesssim \norm{\abs{D_v}^{(l+2)s} f}_{L^2} + \norm{\abs{D_v}^{(l+2)s} f}^{\frac{1}{1+sl + 2s}}_{L^2}\norm{\abs{D_x}^{\frac{ls}{1+2s}} h}^{\frac{s(l+2)}{1+2s+sl}}_{L^2}.
\eeq

\item Finally, \beq\label{eq:step3_l}
	\norm{\abs{D_v}^{s}D_x^{\frac{(l+1)s}{1+2s}} f}_{L^2}\leq \norm{\abs{D_x}^{\frac{l s}{1+2s}} h}_{L^2}^{\frac{1}{2}}\norm{\abs{D_x}^{\frac{(l+2)s}{1+2s}} f}^{\frac{1}{2}}_{L^2}.
\eeq
\end{itemize}

Now let $\varepsilon_0\in(0,1)$, and set
\[
\mathscr P
:= |D_v|^{2(l+1)} + |D_x|^{\frac{2(l+1)}{1+2s}} + \sum_{j=1}^l |D_v|^{2j}\,|D_x|^{\frac{2l+2-2j}{1+2s}},
\qquad
\mathscr M := (\varepsilon_0+\mathscr P)^s + |D_t|^{l}\partial_t .
\]
We test \eqref{eq:kolmogorov_global} with 
\beq\label{eq:test-fct}
	\mathscr M f= \left(\varepsilon_0 + \abs{D_v}^{2(l+1)} + \abs{D_x}^{\frac{2(l+1)}{1+2s}} + \sum_{j =1}^l \abs{D_v}^{2j} \abs{D_x}^{\frac{2l + 2 - 2j}{1+2s}}\right)^sf + \abs{D_t}^l\partial_tf,
\eeq
so that we get, after taking real parts,
\bal\label{eq:aux0_l}
	\Re\int\mathscr M f\cdot \left(\abs{D_v}^{2s} f +\partial_t f \right)\dd z &= - \Re\int \mathscr M f  v \cdot \nabla_x f \dd z +\Re \int\mathscr Mf h \dd z=: I_1 + I_2.
\eal

For the left hand side of \eqref{eq:aux0_l} we find
\bals	\Re\int\mathscr M f\cdot \left(\abs{D_v}^{2s} f +\partial_t f \right)\dd z 
	&=\int_{\R^{1+2d}} \big(\varepsilon_0+\hat{\mathscr P} \big)^{s}\abs{\xi}^{2s}\abs{\hat f}^2\dd \eta\dd k\dd \xi
   +\int_{\R^{1+2d}} |\eta|^{l+2}\,|\hat f|^2\dd\eta\dd k \dd\xi,
\eals
since $\Re\big\langle(\varepsilon_0+\mathscr P)^s f,\partial_t f\big\rangle=0$ and
$\Re\big\langle |D_t|^l\partial_t f,|D_v|^{2s}f\big\rangle=0$. By concavity of $t\mapsto t^s$ on $\R_+$ for $s < 1$, 
\[
(a_0+\cdots+a_N)^s \ge (N+1)^{s-1}\sum_{m=0}^N a_m^{s}, \qquad (a_m\ge0),
\]
which we apply with $a_0=\varepsilon_0$, $a_1=|\xi|^{2(l+1)}$, $a_2=|k|^{\frac{2(l+1)}{1+2s}}$ and
$a_{2+j}=|\xi|^{2j}|k|^{\frac{2l+2-2j}{1+2s}}$ for $j=1,\dots,l$, so that
\bals
\Re\int \mathscr M f\cdot\big(|D_v|^{2s}f+\partial_t f\big)\dd z
\gtrsim_l
&\norm{ \abs{D_t}^{\frac l2}\partial_t f}_{L^2(\R^{1+2d})}^2 + \norm{|D_v|^{(l+2)s} f}_{L^2(\R^{1+2d})}^2
 + \norm{|D_v|^{s}|D_x|^{\frac{(l+1)s}{1+2s}} f}_{L^2(\R^{1+2d})}^2 \\
&\quad+ \sum_{j=1}^l \norm{|D_v|^{(j+1)s}|D_x|^{\frac{(l+1-j)s}{1+2s}} f}_{L^2(\R^{1+2d})}^2.
\eals
Finally, we recover also a mixed time-velocity derivative by interpolation: 
\[
	\abs{\eta}^{l+1} \abs{\xi}^{2s} = \big(\abs{\eta}^{l+2}\big)^{\frac{l+1}{l+2}} \big(\abs{\xi}^{2(l+2)s}\big)^{\frac{1}{l+2}},
\]
so that Hölder in $(\eta, \xi)$ yields, 
\[
	\norm{\abs{D_t}^{\frac{l+1}{2}} \abs{D_v}^s f}_{L^2}^2 \leq \norm{\abs{D_t}^{\frac{l}{2}} \partial_t  f}_{L^2}^{\frac{2(l+1)}{l+2}} \norm{ \abs{D_v}^{(l+2)s} f}_{L^2}^{\frac{2}{l+2}}  \leq \epsilon  \norm{\abs{D_t}^{\frac{l}{2}} \partial_t  f}_{L^2}^2 + C(\epsilon)  \norm{ \abs{D_v}^{(l+2)s} f}_{L^2}^2.
\]

 Thus we conclude
\bal\label{eq:aux0.5-l}
	\Re\int\mathscr M f\cdot \left(\abs{D_v}^{2s} f +\partial_t f \right)\dd z &\gtrsim \norm{\abs{D_v}^{(l+2)s} f}^2_{L^2}   +  \norm{\abs{D_t}^{\frac{l}{2}}\partial_t f}^2_{L^2}  +  \norm{\abs{D_t}^{\frac{l+1}{2}}\abs{D_v}^s f}^2_{L^2} \\
	&\qquad+ \sum_{j = 1}^l \norm{\abs{D_v}^{(j+1)s} \abs{D_x}^{\frac{(l + 1 - j)s}{1+2s}} f}^2_{L^2}  + \norm{\abs{D_v}^{s} \abs{D_x}^{\frac{(l+1)s}{1+2s}} f}^2_{L^2} .
\eal

On the other hand we get with \eqref{eq:step2_l}
\bal\label{eq:aux1_l}
	I_2 &\lesssim \norm{f}_{L^2}\norm{h}_{L^2} +\Norm{\abs{D_v}^{(l+2)s} f}_{L^2}\Norm{\abs{D_v}^{ls} h}_{L^2} + \Norm{\abs{D_x}^{\frac{(l+2)s}{1+2s}} f}_{L^2}\Norm{\abs{D_x}^{\frac{ls}{1+2s}}h}_{L^2} \\
	&\quad+ \Norm{\abs{D_t}^{\frac{ls}{2s}}\partial_t f}_{L^2}\Norm{\abs{D_t}^{\frac{ls}{2s}} h}_{L^2} \\
	&\quad+ \sum_{j =1}^l \Norm{\abs{D_v}^{(j+1)s} \abs{D_x}^{\frac{(l+1-j)s}{1+2s}} f}_{L^2}\Norm{\abs{D_v}^{(j-1)s}\abs{D_x}^{\frac{(l+1- j)s}{1+2s}}h}_{L^2}\\
	&\lesssim \norm{f}_{L^2}\norm{h}_{L^2} +\Norm{\abs{D_v}^{(l+2)s} f}_{L^2}\Norm{\abs{D_v}^{ls} h}_{L^2} +  \Norm{\abs{D_v}^{(l+2)s}f}_{L^2}\Norm{\abs{D_x}^{\frac{ls}{1+2s}}h}_{L^2} \\
	&\quad +  \Norm{\abs{D_v}^{(l+2)s}f}^{\frac{1}{1+2s+sl}}_{L^2}\Norm{\abs{D_x}^{\frac{ls}{1+2s}}h}^{1 + \frac{(l+2)s}{1+2s+sl}}_{L^2} +\Norm{\abs{D_t}^{\frac{ls}{2s}}\partial_t f}_{L^2}\Norm{\abs{D_t}^{\frac{ls}{2s}} h}_{L^2} \\
	&\quad+ \sum_{j =1}^l \Norm{\abs{D_v}^{(j+1)s} \abs{D_x}^{\frac{(l+1-j)s}{1+2s}} f}_{L^2}\Norm{\abs{D_v}^{(j-1)s}\abs{D_x}^{\frac{(l+1- j)s}{1+2s}}h}_{L^2}\\
	&\lesssim \norm{f}_{L^2}\norm{h}_{L^2} +\Norm{\abs{D_v}^{(l+2)s} f}_{L^2}\Norm{\abs{D_v}^{ls} h}_{L^2} +  \Norm{\abs{D_v}^{(l+2)s}f}_{L^2}\Norm{\abs{D_x}^{\frac{ls}{1+2s}}h}_{L^2} \\
	&\quad +  \Norm{\abs{D_v}^{(l+2)s}f}^{\frac{1}{1+2s + ls}}_{L^2}\Norm{\abs{D_x}^{\frac{ls}{1+2s}}h}^{1 + \frac{(l+2)s}{1+2s+sl}}_{L^2} +\Norm{\abs{D_t}^{\frac{ls}{2s}}\partial_t f}_{L^2}\Norm{\abs{D_t}^{\frac{ls}{2s}} h}_{L^2} \\
	&\quad+ \sum_{j =1}^l \Norm{\abs{D_v}^{(j+1)s} \abs{D_x}^{\frac{(l+1-j)s}{1+2s}} f}_{L^2}\Norm{\abs{D_v}^{ls}h}^{\frac{j-1}{l}}_{L^2}\Norm{\abs{D_x}^{\frac{ls}{1+2s}}h}^{\frac{l+1-j}{l}}_{L^2}\\
	&\lesssim \norm{f}_{L^2}\norm{h}_{L^2} +\Norm{\abs{D_v}^{(l+2)s} f}_{L^2}\Norm{\abs{D_v}^{ls} h}_{L^2} +  \Norm{\abs{D_v}^{(l+2)s}f}_{L^2}\Norm{\abs{D_x}^{\frac{ls}{1+2s}}h}_{L^2} \\
	&\quad +  \Norm{\abs{D_v}^{(l+2)s}f}^{\frac{1}{1+2s+ls}}_{L^2}\Norm{\abs{D_x}^{\frac{ls}{1+2s}}h}^{1+\frac{(l+2)s}{1+2s+ls}}_{L^2} + \Norm{\abs{D_t}^{\frac{ls}{2s}}\partial_t f}_{L^2}\Norm{\abs{D_t}^{\frac{ls}{2s}} h}_{L^2} \\
	&\quad+ \Big(\Norm{\abs{D_v}^{ls}h}_{L^2} + \Norm{\abs{D_x}^{\frac{ls}{1+2s}}h}_{L^2}\Big) \sum_{j =1}^l \Norm{\abs{D_v}^{(j+1)s} \abs{D_x}^{\frac{(l+1-j)s}{1+2s}} f}_{L^2},
\eal
where in the second last inequality we again used Hölder in Fourier and for the last line we used Young's inequality. Note that the last sum can be absorbed on the left hand side of \eqref{eq:aux0_l} eventually. 

For $I_1$ in \eqref{eq:aux0_l} we  Fourier-transform $(t, x, v) \to (\eta, k, \xi)$ so that we get
\bals
	I_1 &= -  \left\langle\left\{\left(\delta  + \abs{D_v}^{2(l+1)} + \abs{D_x}^{\frac{2(l+1)}{1+2s}} + \sum_{j =1}^l \abs{D_v}^{2j} \abs{D_x}^{\frac{2l + 2 - 2j}{1+2s}}\right)^s+ \abs{D_t}^l\partial_t\right\}  f, v \cdot \nabla_x f  \right\rangle\\
	&=  - \left\langle\left\{\left(\hat \delta +\abs{\xi}^{2(l+1)} + \abs{k}^{\frac{2(l+1)}{1+2s}} +\sum_{j =1}^l \abs{\xi}^{2j}\abs{k}^{\frac{2l+2-2j}{1+2s}}\right)^s + \abs{\eta}^{l+1}\right\} {\hat f}, k_i \partial_{\xi_i}\hat f  \right\rangle\\
	&= 2s  \Bigg\langle\left(\hat \delta +\abs{\xi}^{2(l+1)} + \abs{k}^{\frac{2(l+1)}{1+2s}} +\sum_{j =1}^l \abs{\xi}^{2j}\abs{k}^{\frac{2l+2-2j}{1+2s}}\right)^{s-1} \\
	&\qquad\qquad \times \left((l+1) \abs{\xi}^{2l}  + \sum_{j = 1}^l  j\abs{k}^{\frac{2l+2-2j}{1+2s}}\abs{\xi}^{2j-2}\right)\xi_i {\hat f}, k_i \hat f  \Bigg\rangle \\
	&\quad+ \left\langle \left\{\left(\hat \delta +\abs{\xi}^{2(l+1)} + \abs{k}^{\frac{2(l+1)}{1+2s}} +\sum_{j =1}^l \abs{\xi}^{2j}\abs{k}^{\frac{2l+2-2j}{1+2s}}\right)^{s} +\abs{\eta}^{l+1}\right\}\partial_{\xi_i}{\hat f}, k_i \hat f  \right\rangle.
\eals
Thus
\bals
	I_1 &= s\Bigg\langle\left(\hat \delta +\abs{\xi}^{2(l+1)} + \abs{k}^{\frac{2(l+1)}{1+2s}} +\sum_{j =1}^l \abs{\xi}^{2j}\abs{k}^{\frac{2l+2-2j}{1+2s}}\right)^{s-1}\\
	&\qquad \times \left((l+1) \abs{\xi}^{2l}  + \sum_{j = 1}^l  j\abs{k}^{\frac{2l+2-2j}{1+2s}}\abs{\xi}^{2j-2}\right)\xi_i {\hat f}, k_i \hat f  \Bigg\rangle\\
	&\lesssim \int \abs{{\hat f}}^2\cdot \left(\abs{k}^{\frac{2(l+1)}{1+2s}} +\sum_{j =1}^{l+1} \abs{\xi}^{2j}\abs{k}^{\frac{2l+2-2j}{1+2s}}\right)^{s-1}  \left(\abs{\xi}^{2l}  + \sum_{j = 1}^l \abs{k}^{\frac{2l+2-2j}{1+2s}}\abs{\xi}^{2j-2}\right) \abs{\xi}\abs{k} \dd z.
\eals
We claim that we can bound
\bal\label{eq:aux_I1_l}
	I_1 \lesssim  \int \abs{{\hat f}}^2\cdot \sum_{j = 1}^l \abs{\xi}^{2(j-1)s + s}\abs{k}^{\frac{2ls+ 3s - 2(j-1)s}{1+2s}} \dd z +  \int \abs{{\hat f}}^2\cdot \abs{\xi}^{2ls + s}\abs{k}^{\frac{3s}{1+2s}} \dd z.
\eal
Indeed, if $\abs{\xi} \approx \abs{k}^{\frac{1}{1+2s}}$ then one can check that the homogeneity is kept. 
Else assume first that $\abs{\xi}\ll \abs{k}^{\frac{1}{1+2s}}$. Then we have
\beqs
	I_1 \lesssim \int  \abs{{\hat f}}^2 \cdot \abs{\xi} \sum_{j = 1}^l \abs{k}^{\frac{2l+2-2j}{1+2s}}\abs{\xi}^{2j-2}\abs{k}^{\frac{2(l+1)(s-1)}{1+2s}+1}\dd z = \int  \abs{{\hat f}}^2\cdot\sum_{j = 1}^l \abs{k}^{\frac{2ls-2j+4s+1}{1+2s}}\abs{\xi}^{2j-1}\dd z.
\eeqs
Comparing the exponents of $\abs{\xi}$ and $\abs{k}$ gives $2l$ conditions that need to be satisfied,
\beqs
	2j -1 \geq 2(j-1)s +s, \qquad 2ls + 4s - 2j + 1 \leq 2ls - 2(j-1)s + 3s, \quad \forall j \in \{1, \dots, l\},
\eeqs
which holds since $s \leq 1$. 
Now assume on the other hand that $\abs{k}^{\frac{1}{1+2s}}\ll \abs{\xi}$. Then we have
\beqs
	I_1 \lesssim \int \abs{{\hat f}}^2\cdot\abs{\xi}^{2l + 1 + 2(l+1)(s-1)} \abs{k}\dd z =  \int  \abs{{\hat f}}^2\cdot\abs{\xi}^{2ls - 1 + 2s} \abs{k}\dd z.
\eeqs
Thus we need
\beqs
	2ls + 2s - 1 \leq 2ls+s, \qquad 1 \geq \frac{3s}{1+2s},
\eeqs
which both clearly holds for $s \leq 1$.

From \eqref{eq:aux_I1_l} we further estimate 
\bal\label{eq:aux2_l}
	I_1 &\lesssim  \int  \abs{{\hat f}}^2\cdot\sum_{j = 1}^{l+1} \abs{\xi}^{2(j-1)s + s}\abs{k}^{\frac{2ls+ 3s - 2(j-1)s}{1+2s}} \dd z \\
	&\lesssim  \sum_{j = 1}^{l+1}  \norm{\abs{D_v}^{js}\abs{D_x}^{\frac{ls + 2s- js}{1+2s}} f}_{L^2}\norm{\abs{D_v}^{(j-1)s} \abs{D_x}^{\frac{ls+3s - js}{1+2s}} f}_{L^2}.
\eal
For each $j \in \{1, \dots, l\}$ we will look for parameters $\theta_j \in (0, 1)$ such that we can express the right hand side of \eqref{eq:aux2_l} in terms of
\beqs
	\norm{ \big(\abs{D_v}^s \abs{D_x}^{\frac{(l+1)s}{1+2s}}\big)^{1-\theta_j}\big(\abs{D_v}^{(l+1)s} \abs{D_x}^{\frac{s}{1+2s}}\big)^{\theta_j} f}_{L^2},
\eeqs
which we bound using Hölder in Fourier:
 \bals
	&\norm{ \big(\abs{D_v}^s \abs{D_x}^{\frac{(l+1)s}{1+2s}}\big)^{1-\theta_j}\big(\abs{D_v}^{(l+1)s} \abs{D_x}^{\frac{s}{1+2s}}\big)^{\theta_j} f}_{L^2} \\
	&\quad\leq\norm{ \abs{D_v}^s \abs{D_x}^{\frac{(l+1)s}{1+2s}}f}_{L^2}^{1-\theta_j}\norm{\abs{D_v}^{(l+1)s} \abs{D_x}^{\frac{s}{1+2s}} f}_{L^2}^{\theta_j}.
\eals
Then we want to use \eqref{eq:step3_l} and \eqref{eq:step2_l} in order to get a right hand side in terms of our source term, 
\bal\label{eq:aux-holder-in-fourier}
	&\norm{\Big(\abs{D_v}^s \abs{D_x}^{\frac{(l+1)s}{1+2s}}\Big)^{1-\theta_j}\Big(\abs{D_v}^{(l+1)s} \abs{D_x}^{\frac{s}{1+2s}}\Big)^{\theta_j} f}_{L^2}\\
	 &\qquad\leq\norm{ \abs{D_v}^s \abs{D_x}^{\frac{(l+1)s}{1+2s}}f}^{1-\theta_j}_{L^2}\norm{\abs{D_v}^{(l+1)s} \abs{D_x}^{\frac{s}{1+2s}} f}^{\theta_j}_{L^2}\\
	&\qquad\lesssim \norm{\abs{D_x}^{\frac{(l+2)s}{1+2s}}f}^{\frac{1-\theta_j}{2}}_{L^2}\norm{\abs{D_x}^{\frac{ls}{1+2s}} h}^{\frac{1-\theta_j}{2}} \norm{\abs{D_v}^{(l+1)s} \abs{D_x}^{\frac{s}{1+2s}} f}_{L^2}^{\theta_j}\\
	&\qquad\lesssim \left(\norm{\abs{D_v}^{(l+2)s}f} + \norm{\abs{D_v}^{(l+2)s}f}_{L^2}^{\frac{1}{1+(l+2)s}}\norm{\abs{D_x}^{\frac{ls}{1+2s}} h}_{L^2}^{\frac{(l+2)s}{1+(l+2)s}}\right)^{\frac{1-\theta_j}{2}}\\
	&\qquad\qquad\times\norm{\abs{D_x}^{\frac{ls}{1+2s}} h}_{L^2}^{\frac{1-\theta_j}{2}} \norm{\abs{D_v}^{(l+1)s} \abs{D_x}^{\frac{s}{1+2s}} f}_{L^2}^{\theta_j}.
\eal

We now apply \eqref{eq:aux-holder-in-fourier} on each term in the right hand side of \eqref{eq:aux2_l}. 
For each $j \in \{2, \dots, l+1\}$ we write
\beqs
	\norm{\abs{D_v}^{(j-1)s} \abs{D_x}^{\frac{ls+3s - js}{1+2s}} f}_{L^2} = \norm{\big(\abs{D_v}^s \abs{D_x}^{\frac{(l+1)s}{1+2s}}\big)^{1-\theta_j}\big(\abs{D_v}^{(l+1)s} \abs{D_x}^{\frac{s}{1+2s}}\big)^{\theta_j} f}_{L^2},
\eeqs
where $\theta_j = \frac{j-2}{l}$. Then using \eqref{eq:aux-holder-in-fourier} and Young's inequality $ab \lesssim_{p, q} a^p + b^q$ with $\frac{1}{p} + \frac{1}{q} = 1$, we bound
\begin{align*}
	\sum_{j = 2}^{l+1}  &\norm{\abs{D_v}^{js}\abs{D_x}^{\frac{ls + 2s- js}{1+2s}} f}_{L^2}\norm{\abs{D_v}^{(j-1)s} \abs{D_x}^{\frac{ls+3s - js}{1+2s}} f}_{L^2} \\
	&\lesssim \sum_{j = 2}^{l+1} \left(\norm{\abs{D_v}^{(l+2)s}f}_{L^2} + \norm{\abs{D_v}^{(l+2)s}f}_{L^2}^{\frac{1}{1+(l+2)s}}\norm{\abs{D_x}^{\frac{ls}{1+2s}} h}_{L^2}^{\frac{(l+2)s}{1+(l+2)s}}\right)^{\frac{1-\theta_j}{2}}\\
	&\qquad\times\Norm{\abs{D_x}^{\frac{ls}{1+2s}} h}_{L^2}^{\frac{1-\theta_j}{2}} \norm{\abs{D_v}^{(l+1)s} \abs{D_x}^{\frac{s}{1+2s}} f}^{\theta_j}_{L^2}\norm{\abs{D_v}^{js}\abs{D_x}^{\frac{ls + 2s- js}{1+2s}} f}_{L^2}\\
	&\lesssim \norm{\abs{D_v}^{(l+2)s}f}_{L^2}^{\frac{4}{3}} \norm{\abs{D_x}^{\frac{ls}{1+2s}} h}_{L^2}^{\frac{2}{3}}  + \norm{\abs{D_v}^{(l+2)s}f}_{L^2}^{\frac{2}{1+(l+2)s}}\norm{\abs{D_x}^{\frac{ls}{1+2s}} h}_{L^2}^{\frac{2(l+2)s}{1+(l+2)s}} \\
	&\quad +\sum_{j = 2}^{l+1} \Bigg[\norm{\abs{D_v}^{(l+1)s} \abs{D_x}^{\frac{s}{1+2s}} f}_{L^2}^{\frac{8\theta_j}{5+3\theta_j}}\norm{\abs{D_x}^{\frac{ls}{1+2s}} h}_{L^2}^{\frac{2(1-\theta_j)}{5+3\theta_j}} \norm{\abs{D_v}^{js}\abs{D_x}^{\frac{ls + 2s- js}{1+2s}} f}_{L^2}^{\frac{8}{5+3\theta_j}}\\
	&\quad +\norm{\abs{D_v}^{(l+1)s} \abs{D_x}^{\frac{s}{1+2s}} f}_{L^2}^{\frac{4\theta_j}{3+\theta_j}}\norm{\abs{D_v}^{js}\abs{D_x}^{\frac{ls + 2s- js}{1+2s}} f}_{L^2}^{\frac{4}{3+\theta_j}}\norm{\abs{D_x}^{\frac{ls}{1+2s}} h}^{\frac{2(1-\theta_j)}{3+\theta_j}}_{L^2}\Bigg] \displaybreak \\
	&\lesssim \norm{\abs{D_v}^{(l+2)s}f}^{\frac{4}{3}}_{L^2}\norm{\abs{D_x}^{\frac{ls}{1+2s}} h}_{L^2}^{\frac{2}{3}}  + \norm{\abs{D_v}^{(l+2)s}f}^{\frac{2}{1+(l+2)s}}_{L^2}\norm{\abs{D_x}^{\frac{ls}{1+2s}} h}_{L^2}^{\frac{2(l+2)s}{1+(l+2)s}} \\
	&\quad +\sum_{j = 2}^{l+1} \Bigg[\norm{\abs{D_v}^{(l+1)s} \abs{D_x}^{\frac{s}{1+2s}} f}_{L^2}^{\frac{16\theta_j}{1+7\theta_j}}\norm{\abs{D_x}^{\frac{ls}{1+2s}} h}^{\frac{2(1-\theta_j)}{1+7\theta_j}}_{L^2} \\
	&\quad + \norm{\abs{D_x}^{\frac{ls}{1+2s}} h}^{\frac{2(1-\theta_j)}{9-\theta_j}}_{L^2} \norm{\abs{D_v}^{js}\abs{D_x}^{\frac{ls + 2s- js}{1+2s}} f}^{\frac{16}{9-\theta_j}}_{L^2} \\
	&\quad+\norm{\abs{D_v}^{(l+1)s} \abs{D_x}^{\frac{s}{1+2s}} f}_{L^2}^{\frac{8\theta_j}{1+3\theta_j}}\norm{\abs{D_x}^{\frac{ls}{1+2s}} h}^{\frac{2(1-\theta_j)}{1+3\theta_j}}_{L^2} \\
	&\quad+ \norm{\abs{D_v}^{js}\abs{D_x}^{\frac{ls + 2s- js}{1+2s}} f}_{L^2}^{\frac{8}{5-\theta_j}}\norm{\abs{D_x}^{\frac{ls}{1+2s}} h}^{\frac{2(1-\theta_j)}{5-\theta_j}}_{L^2}\Bigg] \\
	&\lesssim \varepsilon \norm{\abs{D_v}^{(l+2)s}f}^2_{L^2}  + \varepsilon \sum_{j =2}^{l+1} \norm{\abs{D_v}^{js}\abs{D_x}^{\frac{ls + 2s- js}{1+2s}} f}^2_{L^2}+ C_\varepsilon \norm{\abs{D_x}^{\frac{ls}{1+2s}} h}_{L^2}^2,
\end{align*}
for some $\varepsilon \in (0, 1)$. (Note the second inequality uses Young's inequality twice, once with $p_1 = \frac{8}{3(1-\theta_j)}$ so that $q_1 = \frac{8}{5+3\theta_j}$ and once with $p_2 = \frac{4}{1-\theta_j}$ so that $q_2= \frac{4}{3+\theta_j}$.)


Finally, the only remaining term is when $j = 1$ in \eqref{eq:aux2_l}, which we estimate using \eqref{eq:step3_l} and \eqref{eq:step2_l}
\bals
	 \norm{\abs{D_v}^{s}\abs{D_x}^{\frac{(l+1)s}{1+2s}} f}_{L^2}\norm{\abs{D_x}^{\frac{(l+2)s}{1+2s}} f}_{L^2} &\lesssim \norm{\abs{D_x}^{\frac{ls}{1+2s}} h}_{L^2}^{\frac{1}{2}} \norm{\abs{D_x}^{\frac{(l+2)s}{1+2s}} f}_{L^2}^{\frac{3}{2}} \\
	 &\lesssim \norm{\abs{D_x}^{\frac{ls}{1+2s}} h}_{L^2}^{\frac{1}{2}} \norm{\abs{D_v}^{(l+2)s} f}_{L^2}^{\frac{3}{2}} \\
	 &\quad+  \norm{\abs{D_x}^{\frac{ls}{1+2s}} h}_{L^2}^{\frac{1}{2}+\frac{3(l+2)s}{2(1+(l+2)s)}}  \norm{\abs{D_v}^{(l+2)s} f}_{L^2}^{\frac{3}{2(1+(l+2)s)}}\\
	 &\lesssim  \varepsilon \norm{\abs{D_v}^{(l+2)s}f}_{L^2}^2 + C_\varepsilon  \norm{\abs{D_x}^{\frac{ls}{1+2s}} h}_{L^2}^2.
\eals

Therefore, we have shown
\beq\label{eq:aux3_l}
	I_1 \lesssim \varepsilon \norm{\abs{D_v}^{(l+2)s}f}^2_{L^2} + \varepsilon \sum_{j =2}^{l+1}  \norm{\abs{D_v}^{js}\abs{D_x}^{\frac{ls + 2s- js}{1+2s}} f}_{L^2}^2 + C_\varepsilon  \norm{\abs{D_x}^{\frac{ls}{1+2s}} h}_{L^2}^2.
\eeq
Note that for each $j \in \{1, \dots, l\}$ we can eventually absorb the term $\Norm{\abs{D_v}^{(j+1)s}\abs{D_x}^{\frac{ls + s- js}{1+2s}} f}_{L^2}$ on the left hand side of \eqref{eq:aux0_l}.

We combine \eqref{eq:aux0_l}, \eqref{eq:aux0.5-l}, \eqref{eq:aux1_l} and \eqref{eq:aux3_l} to get
\bals
	 \norm{\abs{D_v}^{(l+2)s}f}^2_{L^2} + \norm{\abs{D_t}^{\frac{ls}{2s}}\partial_t f}^2_{L^2}&+ \sum_{j =0}^l  \norm{\abs{D_v}^{(j+1)s}\abs{D_x}^{\frac{ls + s- js}{1+2s}} f}^2_{L^2} + \norm{\abs{D_t}^{\frac{l+1}{2}}\abs{D_v}^s f}^2_{L^2}\\
	 &\lesssim \norm{f}^2_{L^2} + \norm{h}^2_{L^2} +  \norm{\abs{D_v}^{l s} h}^2_{L^2}+ \Norm{\abs{D_x}^{\frac{ls}{1+2s}}h}_{L^2}^2 +\norm{\abs{D_t}^{\frac{ls}{2s}} h}_{L^2}^2.
\eals
Thus, by \eqref{eq:step2_l} we have
\bals
	\norm{\abs{D_x}^{\frac{(l+2)s}{1+2s}} f}^2_{L^2} \lesssim \norm{f}^2_{L^2} + \norm{h}^2_{L^2} +  \Norm{\abs{D_v}^{l s} h}^2_{L^2}+ \Norm{\abs{D_x}^{\frac{ls}{1+2s}}h}_{L^2}^2 +\Norm{\abs{D_t}^{\frac{ls}{2s}} h}_{L^2}^2.
\eals
We conclude
\bal\label{eq:conclusio-1}
	 \norm{\abs{D_v}^{(l+2)s}f}^2_{L^2}+ &\norm{\abs{D_t}^{\frac{ls}{2s}}\partial_t f}^2_{L^2}+\Norm{\abs{D_x}^{\frac{(l+2)s}{1+2s}} f}^2_{L^2}\\
	 &+ \sum_{j =0}^l  \Norm{\abs{D_v}^{(j+1)s}\abs{D_x}^{\frac{ls + s- js}{1+2s}} f}^2_{L^2}  +\norm{\abs{D_t}^{\frac{l+1}{2}}\abs{D_v}^s f}^2_{L^2}\\
	  &\qquad\lesssim \norm{f}^2_{L^2} + \norm{h}^2_{L^2} +  \Norm{\abs{D_v}^{l s} h}^2_{L^2}+ \Norm{\abs{D_x}^{\frac{ls}{1+2s}}h}_{L^2}^2 +\Norm{\abs{D_t}^{\frac{ls}{2s}} h}_{L^2}^2.
\eal

\textit{Step 2: Local estimates.}
Let $0 < r < R$ and let $\delta = R -r > 0$ be from the statement of the theorem. With no loss in generality set $z_0 = (0, 0, 0)$ and assume $f$ solves \eqref{eq:kolmogorov_energy} in $Q_R(z_0)$. We introduce smooth functions $\theta = \theta(v) \in C_c^\infty(\R^{d})$ and $\eta = \eta(t, x) \in C_c^\infty(\R^{1+d})$ such that $\theta = 1$ in $B_r$ and $\theta = 0$ outside $B_R$, such that $\eta = 1$ in $(-r^{2s}, 0) \times B_{r^{1+2s}}$ and $\eta = 0$ outside $(-R^{2s}, 0) \times B_{R^{1+2s}}$, and so that $\abs{D_v} \theta \lesssim \delta^{-1}, \abs{D_x}^{\frac{1}{1+2s}} \eta \lesssim \delta^{-1}, \abs{D_t}^{\frac{1}{2s}} \eta\lesssim \delta^{-1}$.
Then we let
\beqs
	g = f \theta\eta,
\eeqs
so that $g$ satisfies
\beq\label{eq:local_txv}
	\mathcal T g + \abs{D_v}^{2s} g = f \theta\big(\mathcal T \eta\big) + h \theta\eta + \abs{D_v}^{2s} g - \big(\abs{D_v}^{2s} f\big) \theta\eta   =: H 
\eeq
in $\R^{1+2d}$. The final two terms form a commutator like in Proposition \ref{lem:commutator}.

\textit{Step 3-(i): The base case.}
We start with $l = 0$. 
The global case \eqref{eq:conclusio-1} for $l = 0$ gives
\beqs
	\norm{\abs{D_v}^{2s} g}_{L^2}+\Norm{\partial_t g}_{L^2} +\norm{\abs{D_x}^{\frac{2s}{1+2s}} g}_{L^2} +  \norm{\abs{D_v}^{s}\abs{D_x}^{\frac{s}{1+2s}} g}^2_{L^2} \lesssim \norm{H}_{L^2}.
\eeqs

It remains to bound the right hand side. 
We have by the standard energy estimate (see \cite[Proposition 9]{JGCM} for $s=1$ and \cite[Lemma 6.2]{IS} or \cite[Proposition 3.3]{AL} for the fractional case $s \in (0, 1)$)
\beqs
	\norm{\abs{D_v}^s f}_{L^2(Q_r)} \lesssim \norm{h}_{L^2(Q_R)} + \delta^{-s}\norm{f}_{L^2_{t, x}L^\infty_v(Q_R^v\times \R^d)}. 
\eeqs
Moreover, we see
\beqs
	\norm{f \theta\mathcal T \eta}_{L^2} \lesssim \delta^{-2s}\norm{f}_{L^2(Q_R)}.
\eeqs
The remaining part is a commutator of the form 
\beqs	
	 \abs{D_v}^{2s} g - \big(\abs{D_v}^{2s} f \big) \theta\eta = \eta(t, x) \int_{\R^d} f(w) \frac{\big(\theta(v) - \theta(w)\big)}{\abs{v-w}^{d+2s}} \dd w.
\eeqs	 
By Proposition \ref{lem:commutator}, 
\beqs
	\norm{\abs{D_v}^{2s} g - \big(\abs{D_v}^{2s} f \big) \theta\eta}_{L^2} \lesssim \delta^{-s}\norm{f}_{L^2_{t, x}H^s_v(Q_R)} + \delta^{-s}\norm{f}_{L^2_{t, x, v}(Q_R)} + \delta^{-2s}\norm{f}_{L^2( \mathcal I \times \Omega_x; L^\infty(\R^d))}.
\eeqs
Thus, from the energy estimate
\bals
	\norm{H}_{L^2} &\lesssim \Big(\delta^{-s}\norm{f}_{L^2_{t, x}H^s_v(Q_R)} + \delta^{-2s}\norm{f}_{L^2( \mathcal I \times \Omega_x; L^\infty(\R^d))}\Big) +\norm{h}_{L^2(Q_R)} \\
	&\lesssim \delta^{-2s}\norm{f}_{L^2( \mathcal I \times \Omega_x; L^\infty(\R^d))} + \norm{h}_{L^2(Q_R)}.
\eals

Finally, since $g = f$ in $Q_r$ we conclude
\bals
	\Norm{\partial_t f}_{L^2(Q_r)} + \Norm{\abs{D_v}^{2s} f}_{L^2(Q_r)} + \Norm{\abs{D_x}^{\frac{2s}{1+2s}} f}_{L^2(Q_r)} &+  \Norm{\abs{D_v}^{s}\abs{D_x}^{\frac{s}{1+2s}} f}_{L^2(Q_r)}\\
	&\lesssim \delta^{-2s}\norm{f}_{L^2( \mathcal I \times \Omega_x; L^\infty(\R^d))} + \norm{h}_{L^2(Q_R)}.
\eals

\textit{Step 3-(ii): The general case.}
Now let $l \in \N_0$.
We proceed by induction. Let $l \geq 1$ and assume the conclusion holds for $l -1$, that is we have
\bal\label{eq:ind-hypo}
	&\norm{\abs{D_t}^{\frac{(l+1)s}{2s}} f}_{L^2(Q_r)} + \norm{\abs{D_v}^{(l+1)s} f}_{L^2(Q_r)}+\norm{\abs{D_x}^{\frac{(l+1)s}{1+2s}} f}_{L^2(Q_r)} \\
	&\quad+ \norm{\abs{D_t}^{\frac{ls}{2s}} \abs{D_v}^sf}_{L^2(Q_r)} +\sum_{j =0}^{l-1}  \norm{\abs{D_v}^{(j+1)s}\abs{D_x}^{\frac{ls + s- js}{1+2s}} f}^2_{L^2(Q_r)} \\
	&\qquad\lesssim  \delta^{-(l+1)s}\norm{f}_{L^\infty(\R^{1+2d})} + \delta^{-(l+1)s} \norm{f}_{L^2(Q_{2r})} + \delta^{-(l-1)s} \norm{h}_{L^\infty(\R^{1+2d})}\\
	&\qquad\quad+ \delta^{-{(l-1)s}} \left(\norm{\abs{D_v}^{(l-1)s}h}_{L^2(Q_R)}+  \norm{\abs{D_x}^{\frac{(l-1)s}{1+2s}}h}_{L^2(Q_R)}+  \norm{\abs{D_t}^{\frac{l-1}{2}}h}_{L^2(Q_R)}\right),
\eal
where $\delta = R-r$.

From \eqref{eq:conclusio-1} we have
\bals
	\norm{\abs{D_v}^{(l+2)s} g}_{L^2} +&\norm{\abs{D_t}^{\frac{l}{2}}\partial_t g}_{L^2}  + \norm{\abs{D_x}^{\frac{(l+2)s}{1+2s}} g}_{L^2} \\
	&+ \sum_{j =0}^l  \norm{\abs{D_v}^{(j+1)s}\abs{D_x}^{\frac{ls + s- js}{1+2s}} g}^2_{L^2} +\norm{\abs{D_t}^{\frac{(l+1)s}{2s}} \abs{D_v}^sg}_{L^2} \\
	&\qquad\lesssim  \norm{\abs{D_v}^{ls}H}_{L^2}+  \norm{\abs{D_x}^{\frac{ls}{1+2s}}H}_{L^2} +  \norm{\abs{D_t}^{\frac{l}{2}}H}_{L^2}.
\eals

To estimate the right hand side of \eqref{eq:local_txv}, we recall $\eta = \eta(t, x)$ and $\theta = \theta(v)$, so that
\bal\label{eq:aux-D-ls-H}
	\abs{D}^{ls} H &= \abs{D}^{ls}\big(f\theta\mathcal T \eta\big) +\abs{D}^{ls} (h \eta \theta)+ \abs{D_v}^{ls} \big(\big[\abs{D_v}^{2s}, \theta\big] f \big) \eta\\
	&\qquad +\abs{D_x}^{\frac{ls}{1+2s}} \big(\big[\abs{D_v}^{2s}, \theta\big] f  \eta\big) +\abs{D_t}^{\frac{ls}{2s}} \big(\big[\abs{D_v}^{2s}, \theta\big] f  \eta\big).
\eal
(We abbreviate $\abs{D}$ to denote a pseudo-differential operator in either $t$, $x$ or $v$.)
The first two terms in \eqref{eq:aux-D-ls-H} are bounded using the fractional Leibniz rule, Lemma \ref{lem:product-cutoff}: applied with $\beta = ls$, and then using the induction hypothesis \eqref{eq:ind-hypo}, we find 
\bals
	 &\norm{\abs{D_v}^{ls}(f \theta\mathcal T\eta)}_{L^2}+  \Norm{\abs{D_x}^{\frac{ls}{1+2s}}(f \theta\mathcal T\eta)}_{L^2} +  \norm{\abs{D_t}^{\frac{l}{2}}(f \theta\mathcal T\eta)}_{L^2} \\
	 &\lesssim \delta^{-{2s}} \Big(\Norm{\abs{D_v}^{ls}f}_{L^2(Q_{2r})} + \Norm{\abs{D_x}^{\frac{ls}{1+2s}}f}_{L^2(Q_{2r})} +\Norm{\abs{D_t}^{\frac{l}{2}}f}_{L^2(Q_{2r})}\Big) \\
	 &\quad + \delta^{-(l+2)s} \norm{f}_{L^\infty(\R^{1+2d})} + \delta^{2-2s} \norm{\mathcal T \eta }_{C^2} \norm{f}_{L^2(Q_{2r})} + \delta^{1-s} \norm{\mathcal T \eta }_{C^1} \norm{f}_{L^2_{t, v}H_x^{(l-1)s}(Q_{2r})}  \\
	  &\quad+ \delta^{2-2s} \norm{\theta }_{C^2} \norm{f}_{L^2(Q_{2r})} + \delta^{1-s} \norm{\mathcal T \eta }_{C^1} \norm{f}_{L^2_{x, v}H_t^{(l-1)s}(Q_{2r})}  + \delta^{1-s} \norm{\theta }_{C^1} \norm{f}_{L^2_{t, x}H_v^{(l-1)s}(Q_{2r})} \\
	 &\lesssim \delta^{-{(l+2)s}} \norm{f}_{L^\infty(\R^{1+2d})} + \delta^{-(l+2)s} \norm{f}_{L^2(Q_{2r})}\\
	 &\quad+\delta^{-{(l+1)s}} \left(\norm{\abs{D_v}^{(l-2)s}h}_{L^2(Q_R)}+  \norm{\abs{D_x}^{\frac{(l-2)s}{1+2s}}h}_{L^2(Q_R)}+  \norm{\abs{D_t}^{\frac{l-2}{2}}h}_{L^2(Q_R)}+ \norm{h}_{L^\infty(\R^{1+2d})}\right).
\eals
The third term in \eqref{eq:aux-D-ls-H} consists of a commutator of order $(l+2)s$, and a commutator of order $ls$ applied to $\abs{D_v}^{2s} f$, which we bound via Corollary \ref{cor:comm-localisation}, and again the induction hypothesis:
\bal\label{eq:aux-comms}
	&\norm{  \abs{D_v}^{ls} \big(\big[\abs{D_v}^{2s}, \theta\big] f \big) \eta}_{L^2(\R^{1+2d})} \\
	&\quad\lesssim  \left(\delta^{-2s} \norm{f}_{L^2(Q_R)} + \delta^{-(l+2)s} \norm{f}_{L^2( \mathcal I \times \Omega_x; L^\infty(\R^d))} + \delta^{-s} \norm{\abs{D_v}^{(l+1)s} f}_{L^2(Q_{2r})}\right)\\
	&\quad\lesssim \delta^{-(l+2)s}\norm{f}_{L^\infty(\R^{1+2d})} + \delta^{-(l+2)s} \norm{f}_{L^2(Q_{2r})}+\delta^{-{(l+1)s}}\norm{h}_{L^\infty(\R^{1+2d})}\\
	&\qquad+   \delta^{-{(l+1)s}}\left(\norm{\abs{D_v}^{(l-1)s}h}_{L^2(Q_R)}+  \norm{\abs{D_x}^{\frac{(l-1)s}{1+2s}}h}_{L^2(Q_R)}+  \Norm{\abs{D_t}^{\frac{l-1}{2}}h}_{L^2(Q_R)}\right).
\eal
The final two terms in \eqref{eq:aux-D-ls-H} are bounded using first the fractional Leibniz rule in Lemma \ref{lem:product-cutoff}, second the commutator estimate of Proposition \ref{lem:commutator} of order $k = 2s$, and third the induction hypothesis \eqref{eq:ind-hypo}
\bals
	&\norm{\abs{D_x}^{\frac{ls}{1+2s}} \big(\big[\abs{D_v}^{2s}, \theta\big] f  \eta\big)}_{L^2} \\
	&\quad\leq C \norm{ \big(\big[\abs{D_v}^{2s}, \theta\big] \abs{D_x}^{\frac{ls}{1+2s}} f  \big)}_{L^2(\mathcal I \times \Omega_x \times \R^d)} + C \delta^{-ls} \norm{ \big(\big[\abs{D_v}^{2s}, \theta\big] f  \big)}_{L^2(\mathcal I \times \Omega_x \times \R^d)} \\
	&\quad\lesssim  \norm{  \abs{D_v}^{s}\abs{D_x}^{\frac{ls}{1+2s}} f  }_{L^2(Q_R)} + \delta^{-2s} \norm{  \abs{D_x}^{\frac{ls}{1+2s}} f  }_{L^2(Q_R)} + \delta^{-2s} \norm{\abs{D_x}^{\frac{ls}{1+2s}}f}_{L^2( \mathcal I \times \Omega_x; L^\infty(\R^d))}   \\
	&\quad \quad +  \delta^{-(l+2)s} \norm{  f }_{L^2(Q_R)} +  \delta^{-(l+2)s} \norm{  f }_{L^\infty(\mathcal I\times \Omega_x \times\R^d)}  + \delta^{-(l+2)s} \norm{f}_{L^2_{t}(\mathcal I)L^\infty_{x, v}(\R^{2d})} +\delta^{-s} \norm{f}_{L^2_{t, v}H^{(l+1)s}_x(Q_{2r})}\\ 
	&\quad\lesssim \delta^{-(l+2)s}\norm{f}_{L^\infty(\R^{1+2d})} + \delta^{-(l+2)s} \norm{f}_{L^2(Q_{2r})}\\
	&\quad \quad+\delta^{-{(l+1)s}}\norm{h}_{L^\infty(\R^{1+2d})}+   \delta^{-{(l+1)s}}\norm{\abs{D}^{(l-1)s}h}_{L^2(Q_R)}.
\eals
The same argument for the temporal derivative shows
\bals
	&\norm{\abs{D_t}^{\frac{ls}{2s}} \big(\big[\abs{D_v}^{2s}, \theta\big] f  \eta\big)}_{L^2} \\
	&\qquad\leq C \norm{ \big(\big[\abs{D_v}^{2s}, \theta\big] \abs{D_t}^{\frac{ls}{2s}} f  \big)}_{L^2(\mathcal I \times \Omega_x \times \R^d)} + C \delta^{-ls} \norm{ \big(\big[\abs{D_v}^{2s}, \theta\big] f  \big)}_{L^2(\mathcal I \times \Omega_x \times \R^d)} \\
	&\qquad\lesssim \delta^{-(l+2)s}\norm{f}_{L^\infty(\R^{1+2d})} + \delta^{-(l+2)s} \norm{f}_{L^2(Q_{2r})}\\
	&\qquad \quad+\delta^{-{(l+1)s}}\norm{h}_{L^\infty(\R^{1+2d})}+   \delta^{-{(l+1)s}}\norm{\abs{D}^{(l-1)s}h}_{L^2(Q_R)}.
\eals
We combine all these estimates for the right hand side and use that $f = g$ in $B_r$ so that we conclude the proof for smooth functions $f$.

\textit{Step 4: A density argument.}


Let $\mathcal T = \partial_t + v\cdot \nabla_x$ be the transport operator.
Define $\mathcal K_0 := \mathcal T + \abs{D_v}^{2s}$ on the domain of $\mathcal K_0$, namely $\mathcal D(\mathcal K_0) = C_0^\infty(\R^{1+2d})$. Then $\mathcal K_0$ is accretive and closable, since for $f \in C_0^\infty(\R^{1+2d})$, due to the fact that $\mathcal T$ is skew-adjoint on $L^2(\R^{1+2d})$, we find
\[
	\big\langle \mathcal K_0 f, f\rangle = \langle \abs{D_v}^{2s} f, f\big\rangle = \norm{\abs{D_v}^s f}_{L^2}^2 \geq 0.
\]
To conclude maximal accretivity, it suffices to show that
\[
	\textrm{ker } \big(I + \mathcal K_0^\ast\big) = \{0\}.
\]
Since $\mathcal T^\ast = -\mathcal T$, we find $\mathcal K_0^\ast = -\mathcal T + \abs{D_v}^{2s}$ in the sense of distributions. Thus, assume 
\beq\label{eq:ker-eq}
	f \in L^2(\R^{1+2d}), \qquad \big(I - \mathcal T + \abs{D_v}^{2s}\big) f = 0, \qquad \textrm{in } \mathcal D'. 
\eeq
By hypoellipticity of $\big(I - \mathcal T + \abs{D_v}^2s\big)$ we get $f \in C^\infty$. 

Let $R > 1$ and let $\chi_R \in C_\infty^\infty(\R^{1+2d})$ be a cutoff so that $\chi_R = 1$ on $Q_R$, and $\chi_R = 0$ outside $Q_{2R}$. We test \eqref{eq:ker-eq} with $f \chi_R^2$ to find
\bals
	0 = \big\langle I + \mathcal K_0^\ast f, \chi_R^2f\rangle = \norm{\chi_R f }_{L^2}^2 + \norm{\abs{D_v}^s (\chi_R f)}_{L^2}^2 + \big\langle \big[\abs{D_v}^{2s} , \chi_R\big] f, \chi_R f\big\rangle - \big\langle \mathcal T f, \chi_R^2 f\big\rangle. 
\eals
The transport term is bounded by
\[
	\abs{ \big\langle \mathcal T f, \chi_R^2 f\big\rangle} \leq C R^{-2s} \norm{f}_{L^2(Q_{2R})}^2. 
\]
The commutator term is bounded using Proposition \ref{lem:commutator} with $k =2$
\[
	\norm{\big[\abs{D_v}^{2s} , \chi_R\big] f}_{L^2} \leq C R^{-s}\norm{f}_{L^2_{t, x}H^s_v(Q_{2R})} + C R^{-2s} \norm{f}_{L^2(Q_{2R})} +  CR^{-2s} \norm{f}_{L^2(\mathcal I\times \Omega_x\times\R^d)}^2. 
\]
Thus
\bals
	\norm{\chi_R f }_{L^2}^2 + \norm{\abs{D_v}^s (\chi_R f)}_{L^2}^2 \leq  &C R^{-2s} \norm{f}_{L^2(Q_{2R})}^2 + \epsilon \norm{f}_{L^2_{t, x}H^s_v(Q_{2R})}^2 \\
	&+ C(\epsilon)R^{-2s} \norm{f}_{L^2(Q_{2R})}^2+ CR^{-2s} \norm{f}_{L^2(\mathcal I\times \Omega_x\times \R^d)}^2.
\eals
This implies
\[
	\norm{\chi_R f }_{L^2}^2 + \norm{\abs{D_v}^s (\chi_R f)}_{L^2}^2 \leq  C R^{-2s} \norm{f}_{L^2(Q_{2R})}^2+ CR^{-2s} \norm{f}_{L^2(\mathcal I\times \Omega_x\times\R^d)}^2.
\]
Since $f \in L^2$ and since $\supp \nabla \chi_R \subset Q_{2R} \setminus Q_R$, we have $R^{-2s} \norm{f}_{L^2(Q_{2R})} \to 0$ as $R \to \infty$. Thus
\[
	\lim_{R \to \infty} \norm{\chi_R f }_{L^2}^2  = \norm{f}_{L^2}^2 = 0,
\]
from which we deduce $f = 0$. Therefore, $\textrm{ker }(I + \mathcal K_0^\ast) = \{0\}$, and thus $\textrm{Im }(I+ \mathcal K_0)$ is dense in $L^2(\R^{1+2d})$. Since $\mathcal K_0$ is accretive, this means that $\bar{\mathcal K}_0$ is maximal accretive. In particular, for every $f \in \mathcal D(\bar{\mathcal K}_0)$ there exists $f_n \in C_0^\infty(\R^{1+2d})$ so that 
\[
	f_n \to f, \textrm{ in } L^2(\R^{1+2d}), \qquad \mathcal K_0 f_n \to \bar{\mathcal K}_0 f, \textrm{ in } L^2(\R^{1+2d}).
\]

We use this to justify the result of Proposition \ref{prop:fractional_energy}. We approximate the localised equation \eqref{eq:local_txv}: define the shifted equation
\[
	\big(I + \mathcal T + \abs{D_v}^{2s}\big) g = F, \qquad F = g + H \in L^2.
\]
We approximate $F$ by picking smooth $F_n \in C_0^\infty(\R^{1+2d})$ such that
\[
	F_n \to F,  \textrm{ in } L^2(\R^{1+2d}), \qquad  \abs{D}^{ls} F_n \to \abs{D}^{ls} F,  \textrm{ in } L^2(\R^{1+2d}). 
\]	
This can be achieved using a standard mollifier, since $F$ is supported in a local cylinder. Moreover, \textit{Step 3-(ii)} above establishes that $\abs{D}^{ls} H \in L^2$ independently of the density argument due to the commutator estimates that are purely analytical and don't require $f$ to be smooth.

Let $g_n$ be defined by
\[
	g_n := \big(I + \bar{\mathcal K}_0\big)^{-1} F_n.
\]
Then $g_n$ is in the domain of $\bar{\mathcal K}_0$ and satisfies, 
\[
	\big(I + \mathcal T + \abs{D_v}^{2s}\big) g_n = F_n. 
\]
By hypoellipticity, if $F_n \in C_0^\infty$, then $g_n \in C^\infty$. Therefore our argument above applies to $g_n$ without truncation. Moreover, resolvent contractivity gives $g_n \to g$ in $L^2$, and the equation gives
\[
	\big(\mathcal T + \abs{D_v}^{2s}\big) g_n = F_n - g_n \to F - g = H, \qquad \textrm{in } L^2.
\]	

If we apply our estimates not only to $g_n$ but also to $g_n - g_m$, which solves
\[
\big(I + \mathcal T + \abs{D_v}^{2s}\big)(g_n-g_m) = F_n - F_m,
\] 
then we obtain
\[
	\norm{\abs{D}^{(l+2)s} (g_n-g_m)}_{L^2} \leq C \Big(\norm{F_n - F_m}_{L^2} + \norm{\abs{D}^{ls} (F_n - F_m)}_{L^2}\Big).
\]
Since $F_n$ is Cauchy in these Sobolev norms, it follows that $\abs{D}^{(l+2)s} g_n$ is Cauchy in $L^2$; and therefore $\abs{D}^{(l+2)s} g_n$ strongly converges in $L^2$. We can therefore take the limit $n \to \infty$, and conclude the result for general $g \in L^2$. 

Finally we use that $g = f$ in $Q_r$ and our commutator lemma to conclude the result for $f$.

The same argument applies for the operator $\mathcal L_0$ since the corresponding kernel $K_0$ is symmetric and pointwise bounded by the kernel of the fractional Laplacian.

\end{proof}

\subsection{Kolmogorov equation: Fundamental solution}
Lastly, for the lower order perturbation arising with the freezing of coefficients, we will make use of the fundamental solution for the (fractional) Kolmogorov equation
\beq
	\mathcal T f = -(-\Delta_v)^s f + h, \quad (t, x, v) \in \R^{1+2d}
\label{eq:fracKolm}
\eeq
for some source term $h \in L^\infty$. In the non-fractional case set $s = 1$. This equation preserves the same Lie group structure as \eqref{eq:1.1loc_div}, \eqref{eq:1.1loc_nondiv} and \eqref{eq:1.1nonloc} and it admits the following fundamental solution \cite{kolmogorov} in case that $s = 1$:
\bals
	J(t, x, v) = \Bigg(\frac{\sqrt 3}{2\pi t^2}\Bigg)^d\exp\Bigg(\frac{-3\abs{x + \frac{tv}{2}}}{t^3}- \frac{\abs{v}^2}{4t}\Bigg), \quad t > 0,
\eals
and $J = 0$ for $t \leq 0$. 
In case that $s \in (0, 1)$ the Green's function is given by
\bals
	J(t, x, v) = c t^{-d(1+\frac{1}{s})} \mathcal J\Bigg(\frac{x}{t^{1+\frac{1}{2s}}}, \frac{v}{t^{\frac{1}{2s}}}\Bigg),
\eals
where $\mathcal J$ is given in Fourier variables as
\bals
	\hat{\mathcal J}(\varphi, \xi) = \exp\Bigg(-\int_0^1 \abs{\xi  + \tau \varphi}^{2s} \dd t\Bigg).
\eals
This can be verified using the characteristic method in the Fourier variables $(\varphi, \xi)$. Since the Green's function is smooth away from the origin, it provides the hypoellipticity of $\mathcal T + (-\Delta_v)^s$.
Similarly to Proposition 2.1 of \cite{IM} we have
\begin{lemma}
Given $h \in L^\infty(\R\times \R^{2d})$ with compact support in time, the function 
\bals
	f(t, x, v) = \int_{\R\times\R^{2d}} h(\tilde t,\tilde x,\tilde v) J(t - \tilde t, x - \tilde x -(t-\tilde t)v, v - \tilde v) \dd \tilde t\dd \tilde x\dd \tilde v =: J\ast_{\textrm{kin}} h(z) 
\eals
solves \eqref{eq:fracKolm} in $\R\times \R^{2d}$. Moreover, for all $z_0 \in \R\times \R^{2d}$ and $r > 0$ there holds
\beqs
	\norm{J \ast_{\textrm{kin}}\mathbbm{1}_{Q_r(z_0)}}_{L^\infty(Q_r(z_0))} \leq C r^{2s},
\eeqs
for some universal constant $C$ depending on $d$.
\label{prop:2.1}
\end{lemma}
\begin{proof}
Given $z = (t, x, v) \in Q_r(z_0)$ we compute the scaling of the fundamental solution stemming from the parabolicity of the equation
\bals
	J \ast_{\textrm{kin}}\mathbbm{1}_{Q_r(z_0)}(t, x, v) &= \int_{Q_r(z_0)} J(t - \tilde t, x - \tilde x -(t-\tilde t)v, v - \tilde v) \dd \tilde z\\
	&= r^{2s}\int_{Q_1(z_0)} J\Bigg(\frac{t}{r^{2s}} - \bar t, \frac{x}{r^{1+2s}} - \bar x -\Big(\frac{t}{r^{2s}}-\bar t\Big)v, \frac{v}{r} - \bar v\Bigg) \dd \bar z\\
	&= r^{2s}J\ast_{\textrm{kin}}\mathbbm{1}_{Q_1(z_0)}\Bigg(\frac{t}{r^{2s}}, \frac{x}{r^{1+2s}}, \frac{v}{r}\Bigg),
\eals
and conclude the proof of Lemma \ref{prop:2.1}.
\end{proof}

\section{Campanato's inequality}\label{sec:campanato-inequality}

\subsection{Local (non-fractional) Campanato's inequality}\label{subsec:0_loc}
Let $0 < r < R$ and $z_0 \in \R^{1+2d}$. Assume $f$ solves \eqref{eq:kolmogorov_energy_loc} in $Q_R(z_0)$ for some constant diffusion coefficient $A$ satisfying \eqref{eq:unifellip} and zero source term $h = 0$. As the coefficients $A$ are constant, there is no distinction between the non-divergence and divergence form. Moreover, note that in this case we can assume $f \in C^\infty$ as we can approximate $f$ with a smooth solution by mollification respecting the Lie group structure. We want to combine \eqref{eq:campnonloc1} with the regularity estimates in Proposition \ref{prop:nonfractional_energy} to infer Campanato's inequality. Together with Poincaré's inequality this constitutes Campanato's approach to Schauder estimates.

We know from \eqref{eq:giaq1} that for any $f \in W^{m, p}$ there is a unique polynomial of degree $m-1$ such that the average of $f-p_{m-1}$ and all derivatives up to order $m-1$ vanishes. Thus, we can apply Poincaré's inequality in $L^2$ \cite[Proposition 3.12]{GiaquintaMartinazzi} to $f-{p_{m-1}}$ by subtracting off zero in form of the average of $f-p_{m-1}$ to bound it by the $L^2$ norm of $D(f-p_{m-1})$. Since this integrand is again averaging to zero, we apply Poincaré's inequality again. Repeating this process $m$-times, and then a fractional Poincaré inequality in the final step, see for example \cite[Equation 1.2]{HSV} or \cite[page 241]{ponce}, we find
\bal
	\int_{Q_r(z_0)} \abs{f - p_{m-1}}^2\dd z&\leq C r^{2m}\int_{Q_r(z_0)} \abs{D_v^m f}^2 \dd z+ Cr^{6\lfloor \frac{m}{3}\rfloor} \int_{Q_r(z_0)} \abs{D_x^{\lfloor\frac{m}{3}\rfloor} (f-p_{m-1})}^2 \dd z  \\
	&\quad\qquad+Cr^{4\lfloor \frac{m}{2}\rfloor}\int_{Q_r(z_0)} \abs{D_t^{\lfloor\frac{m}{2}\rfloor} (f-p_{m-1})}^2 \dd z \\
	&\quad\qquad+ C\sum_{\substack{i, j, k \geq 0\\i + j +k = m}}r^{2 \big(2\lfloor \frac{i}{2}\rfloor + 3 \lfloor \frac{j}{3}\rfloor + k\big)}\int_{Q_r(z_0)} \abs{D_t^{\lfloor\frac{i}{2}\rfloor}D_x^{\lfloor\frac{j}{3}\rfloor}D_v^{k}(f-p_{m-1})}^2 \dd z\\
	&\leq  C r^{2m} \int_{Q_{2r}(z_0)} \abs{D^m f}^2 \dd z,
\label{eq:aux0_loc}
\eal
where $D^m$ is a derivative in time, space or velocity of order $m$. 
To control the right hand side, we use \eqref{eq:campnonloc1}, Sobolev's embedding for some $k$ sufficiently large depending on $n$, and the regularity estimates of Proposition \ref{prop:nonfractional_energy} to get
\bal\label{eq:camp-first}
	\int_{Q_{2r}(z_0)} \abs{D^{m} f}^2 \dd z &\leq C r^n \norm{D^{m} f}_{L^\infty(Q_{2r}(z_0))}^2 \leq C r^n\norm{f}_{H^k(Q_{R/2}(z_0))}^2 \\
	&\leq C(n, k) \frac{r^n}{R^{n+2m}} \norm{f}^2_{L^2(Q_R(z_0))}.
\eal
Thus we deduce 
\bals
	\Norm{f - p_{m-1}}^2_{L^2(Q_r(z_0))} \leq C \left(\frac{r}{R}\right)^{n+2m}\norm{f}^2_{L^2(Q_R(z_0))},
\eals
where $C = C(n)$. This inequality is Campanato's (second) inequality. Now, dividing by $r^{n+2m}$ yields the Campanato norm on the left hand side: 
\bals
	[f]_{\mathcal L_{m-1}^{2, \lambda}(Q_r(z_0))}^2 = r^{-\lambda}\Norm{f - p_{m-1}}^2_{L^2(Q_r(z_0))} \leq C R^{-n-2m} \norm{f}^2_{L^2(Q_R(z_0))},
\eals
where 
\beqs
	\lambda = n + 2m.
\eeqs
\begin{remark}
As a consequence of \eqref{eq:camp-first}, we deduce that the only smooth solutions of \eqref{eq:kolmogorov_energy_loc} with constant coefficients that grow at most polynomially at infinity are kinetic polynomials: if we assume that a solution $f$ of \eqref{eq:kolmogorov_energy_loc} in $\R^{1+2d}$ satisfies 
\beqs
	\sup_{Q_R} f(z) \leq M R^{m-1}, \qquad \forall R \geq 1,
\eeqs
for some constant $M >0$ and $m \geq 1$, then as before we get with Poincaré's inequality, Sobolev embedding, and the regularity estimates for any $r >0$
\bals
	\int_{Q_r} \abs{f - p_{m-1}}^2\dd z&\leq C r^{2m} \int_{Q_{2r}} \abs{D^m f}^2 \dd z\leq C r^{n+2m} \Norm{D^m f}_{H^k(Q_{2r})}^2 \leq C \Big(\frac{r}{R}\Big)^{n+2m} \Norm{f}_{L^2(Q_{R})}^2,
\eals
where $p_{m-1}$ is some kinetic polynomial of degree $m-1$.
Due to the growth assumption on $f$, we thus find
\beqs
	\int_{Q_r} \abs{f - p_{m-1}}^2\dd z\leq C(r, n) R^{-n-2m}R^{2m -2+n},
\eeqs
which tends to $0$ as $R \to \infty$. Thus $f = p_{m-1}$ in $Q_r$. Since $r >0$ was arbitrary, we deduce $f$ is a polynomial of degree at most $m-1$ in $\R^{1+2d}$. In other words, a generalisation of Liouville's theorem holds. Note that a Liouville-type theorem has been used to derive Schauder estimates in the elliptic case by \cite[Lemma 1]{simon} and in the hypoelliptic case by \cite[Theorem 4.1]{ISschauder}.
\end{remark}

\subsection{Non-local (fractional) Campanato's inequality}\label{subsec:cc_nonloc}
As before, we want to combine the observation in \eqref{eq:campnonloc1} with the energy estimates derived in the last subsection to infer Campanto’s inequality. 
Let $0 < r < R$ and $z_0 \in \R^{1+2d}$. We consider the constant coefficient equation \eqref{eq:kolmogorov_energy} with zero source term in $Q_R(z_0)$. 
We have by combining \eqref{eq:giaq1} and the fractional Poincaré inequality, see \cite[page 241]{ponce}, \cite[equation (1.2)]{HSV}, or \cite[Section 1]{mouhot-russ-sire},
\bal
	&\int_{Q_r(z_0)} \abs{f - p_{2s}}^2\dd z\\
	&\leq C r^{2s} \Bigg(\int_{Q_r(z_0)} \Abs{D_t^{\frac{s}{2s}} (f-p_{2s})}^2 \dd z +\int_{Q_r(z_0)} \Abs{D_x^{\frac{s}{1+2s}} (f-p_{2s})}^2 \dd z +\int_{Q_r(z_0)} \Abs{D_v^s (f-p_{2s})}^2 \dd z\Bigg)\\
	&\leq C r^{4s} \Bigg(\int_{Q_r(z_0)} \Abs{D_t^{\frac{2s}{2s}} (f-p_{2s})}^2 \dd z +\int_{Q_r(z_0)} \Abs{D_x^{\frac{2s}{1+2s}} (f-p_{2s})}^2 \dd z +\int_{Q_r(z_0)} \Abs{D_v^{2s} (f-p_{2s})}^2 \dd z\\
	&\qquad + \int_{Q_r(z_0)} \Abs{D_t^{\frac{s}{2s}}D_x^{\frac{s}{1+2s}} (f-p_{2s})}^2 \dd z + \int_{Q_r(z_0)} \Abs{D_t^{\frac{s}{2s}}D_v^{s} (f-p_{2s})}^2 \dd z\\
	&\qquad + \int_{Q_r(z_0)} \Abs{D_v^{s}D_x^{\frac{s}{1+2s}} (f-p_{2s})}^2 \dd z\Bigg)\\
	&\leq C r^{6s} \int_{Q_r(z_0)} \Abs{D^{3s} f}^2 \dd z,
\label{eq:aux0}
\eal
where $D^{3s}$ is a differential of order $3s$ in time, space, or velocity. 
We use \eqref{eq:campnonloc1}, Sobolev's embedding for some $k$ sufficiently large depending on $s$ and $n$, and the energy estimates of Proposition \ref{prop:fractional_energy} to get
\bals
	\int_{Q_r(z_0)} \abs{D^{3s} f}^2 \dd z \leq r^n \Norm{D^{3s} f}_{L^\infty(Q_r(z_0))}^2 \leq C r^n\Norm{f}_{H^k(Q_{R/2}(z_0))}^2 \leq C(n, s, k) \frac{r^n}{R^{n+6s}} \norm{f}^2_{L^\infty(\R^{1+2d})}.
\eals
This can be seen as a non-local analogue of Campanto's inequality. Thus we deduce
\bals
	\Norm{f - p_{2s}}^2_{L^2(Q_r(z_0))} \leq C \left(\frac{r}{R}\right)^{n+6s}\norm{f}^2_{L^\infty(\R^{1+2d})},
\eals
with $C = C(n, s)$. Therefore, dividing by $r^{n+6s}$ yields the Campanato norm on the left hand side: 
\bals
	[f]_{\mathcal L_{2s}^{2, \lambda}(Q_r(z_0))}^2 = r^{-\lambda}\Norm{f - p_{2s}}^2_{L^2(Q_r(z_0))} \leq CR^{-n-6s}\norm{f}^2_{L^\infty(\R^{1+2d})},
\eals
where 
\bals	
	\lambda = n + 6s.
\eals

\section{Campanato's approach: the local (non-fractional) case}
\label{sec:schauder_nonfrac}

We freeze coefficients (also known as Korn's trick) to derive Schauder estimates for the general case. Let $f$ classically solve \eqref{eq:1.1loc_div} or \eqref{eq:1.1loc_nondiv}. Suppose $A = A(t, x,v)$ satisfies \eqref{eq:unifellip} and $h \in C_\ell^{m-3+\alpha}(Q_1)$. Assume that the diffusion coefficient satisfies $A  \in C_\ell^{m-3+\alpha}(Q_1)$. For the divergence form equation \eqref{eq:1.1loc_div} we additionally require $\nabla_v A  \in C_\ell^{m-3+\alpha}(Q_1)$.
 
Similarly to \cite{IM} we consider $0 < \rho \leq \frac{1}{2}$ to be determined and we pick $z_0, z_1 \in Q_1$ and $0 < r \leq 1$ such that $z_1 \in Q_r(z_0)$ and
\bals
	[f]_{C_\ell^{m-1+\alpha}(Q_{\frac{1}{4}})} \leq 2\frac{\abs{f(z_1) - p_{m-1}^{z_0}[f](z_1)}}{r^{m-1+\alpha}}.
\eals
We recall that the Taylor expansion of $f$ at $z_0$ of kinetic degree $m-1$ is given by
\bals
	p_{m-1}^{z_0}[f](z) = \sum_{j} \frac{a_j(z_0)}{j!}\big(t - t_0\big)^{j_0}&\big(x_1-(x_0)_1-(t-t_0)v_1\big)^{j_1} \cdots\big(x_d-(x_0)_d-(t-t_0)v_d\big)^{j_d}\\
	&\times\big(v_{1}-(v_0)_1\big)^{j_{d+1}}\cdots \big(v_d-(v_0)_d\big)^{j_{2d}},
\eals
where we require $j_0 \leq \lfloor\frac{m-1}{2}\rfloor$, $j_1 + \dots + j_d \leq \lfloor\frac{m-1}{3}\rfloor$ and $j_{d+1} + \dots + j_{2d} \leq m-1$. The coefficients can be computed and are given by
\beqs
	a_j(z_0) = (\partial_t + v\cdot \nabla_x)^{j_0}\partial_{x_1}^{j_1}\cdots\partial_{x_d}^{j_d}\partial_{v_1}^{j_{d+1}}\cdots \partial_{v_{d}}^{j_{2d}}f(z_0).
\eeqs

If $r \geq \rho$, we have, using Lemma \ref{lem:2.7},
\bals
	[f]_{C_\ell^{m-1+\alpha}(Q_{1/4})} &\leq 2\rho^{-(m-1+\alpha)}\Bigg\{2\norm{f}_{L^\infty(Q_r(z_0))} + \sum_{j}\Bigg[\rho^{2j_0}\norm{(\partial_t + v\cdot \nabla_x)^{j_0}f}_{L^\infty} \\	&\qquad\qquad\qquad\qquad+ \rho^{3(j_1 + \dots +j_d)}\norm{\partial_{x_1}^{j_1}\cdots\partial_{x_d}^{j_d}f}_{L^\infty}  +\rho^{(j_{d+1} + \dots +j_{2d})} \norm{\partial_{v_1}^{j_{d+1}}\cdots\partial_{v_d}^{j_{2d}}f}_{L^\infty}\Bigg]\Bigg\}\\
	&\leq \frac{1}{4} [f]_{C_\ell^{m-1+ \alpha}(Q_1)} + C(\rho) \norm{f}_{L^\infty(Q_1)}.
\eals

\subsection{Non-divergence form}\label{sec:nondiv-perturb}
Now we consider $r \leq \rho$ and a solution $f$ of \eqref{eq:1.1loc_nondiv}. Let $\eta \in C_c^\infty(\R^{1+2d})$ be a cut-off with $0 \leq \eta \leq 1$, such that $\eta = 1$ in $Q_\rho(z_0)$ and $\eta = 0$ outside $Q_{2\rho}(z_0)$. Let $\tilde f = f \eta$. With no loss of generality we set $z_0 = (0, 0, 0)$. We denote with $p_{2}^{(z_0)}[f]$ the Taylor polynomial of $f$ at $z_0$ with kinetic degree less or equal to $2$. To approximate the general case by the constant coefficient case we split
\bals
	\tilde f - p_{m-1}^{(0)}[\tilde f] = g_1 + g_2,
\eals
where $g_1$ solves
\beqs
	\partial_t g_1 + v\cdot \nabla_x g_1 - \sum_{i, j} a^{i, j}_{(0)}\partial^2_{v_iv_j} g_1 = 0,
\eeqs
for $a^{i, j}_{(0)} = a^{i, j}(z_1)$. Then $g_2$ solves
\beqs
	\partial_t g_2 + v\cdot \nabla_x g_2 -\sum_{i, j} a^{i, j}_{(0)}  \partial^2_{v_iv_j} g_2 = \tilde h -\left(\partial_t + v\cdot \nabla_x - \sum_{i, j} a^{i, j}_{(0)}  \partial^2_{v_iv_j}\right)  p_{m-1}^{(0)}[\tilde f],
\eeqs
where 
\beq\label{eq:tilde-h}
	\tilde h := \Bigg[\sum_{i, j} \Big(a^{i, j} - a^{i, j}_{(0)}\Big)\partial^2_{v_iv_j} f \Bigg] \eta + \sum_{i} \big(b^i \eta - 2 a^{i, j}_{(0)}\partial_{v_i}\eta\big)\partial_{v_i} f + \sum_{i, j} \big(c \eta + \mathcal T\eta - a^{i, j}_{(0)}\partial_{v_iv_j}^2 \eta\big)f + h \eta.
\eeq
Note that for $m =3$ we find 
\[
	\left(\partial_t + v\cdot \nabla_x - \sum_{i, j} a^{i, j}_{(0)}  \partial^2_{v_iv_j}\right)  p_{2}^{(0)}[\tilde f] = \tilde h(0, 0, 0),
\]
coinciding with the zeroth order Taylor expansion of $\tilde h$ around $z_0$. This remains true for larger $m$: this expression is the Taylor polynomial for $\tilde h$ of order $m-3$ around $z_0 = (0, 0, 0)$;
\[
\left(\partial_t + v\cdot \nabla_x - \sum_{i, j} a^{i, j}_{(0)}  \partial^2_{v_iv_j}\right)  p_{2}^{(0)}[\tilde f] = p_{m-3}^{(0)}[\tilde h].
\]

For $g_1$ we have by Subsection \ref{subsec:0_loc} 
\bals
	\int_{Q_r} \Abs{g_1 - p_{m-1}^{(0)}[g_1]}^2 \dd z &\leq C \Big(\frac{r}{R}\Big)^{n + 2m} \int_{Q_R} \abs{g_1}^2 \dd z \\
	&\leq C \Big(\frac{r}{R}\Big)^{n + 2m} \int_{Q_R} \abs{\tilde f-p_{m-1}^{(0)}[\tilde f]}^2 \dd z +  C \Big(\frac{r}{R}\Big)^{n + 2m} \int_{Q_R} \abs{g_2}^2 \dd z.
\eals
For $g_2$ we first perform a change of variables $g_{2, (0)}(t, x, v) := g_2\left(t, (A_{(0)})^{-\frac{1}{2}} x, (A_{(0)})^{-\frac{1}{2}} v\right)$ where $A_{(0)}$ is the constant diffusion coefficient $A_{(0)} =  \big(a^{i, j}_{(0)}\big)_{i, j}$. Then $g_{2, (0)}$ solves 
\bal\label{eq:change-var-A}
	\Bigg(\partial_t + v\cdot \nabla_x  -\sum_{i, j}  \partial^2_{v_iv_j}  \Bigg)g_{2, (0)} (t, x, v)  &= \Bigg(\partial_t   + v\cdot \nabla_x   -\sum_{i, j} a^{i, j}_{(0)} \partial^2_{v_iv_j}   \Bigg)g_2\Big(t, (A_{(0)})^{-\frac{1}{2}} x, (A_{(0)})^{-\frac{1}{2}} v\Big) \\
	&= \Big(\tilde h -\big( p_{m-3}^{(0)}[\tilde h]\big)\Big)_{(0)}\Big(t, (A_{(0)})^{-\frac{1}{2}} x, (A_{(0)})^{-\frac{1}{2}} v\Big)\\
	&=:  \Big(\tilde h_{(0)} - \big(p_{m-3}^{(0)}[\tilde h]\big)_{(0)}\Big)(t,  x,  v).
\eal
Thus, using the scaling of the fundamental solution as stated in Lemma \ref{prop:2.1},
\bals
	\int_{Q_r} \abs{g_{2, (0)}}^2 \dd z \leq C r^n \norm{g_{2, (0)}}_{L^\infty(Q_r)}^2 \leq C r^{n+ 2m + 2\alpha - 2}\left[{\tilde h}_{(0)}\right]^2_{C_\ell^{m-3 +\alpha}(Q_{r})}.
\eals
Since $\norm{g_{2, (0)}}_{L^2} \approx \norm{g_2}_{L^2}$ and $\left[{\tilde h}_{(0)}\right]^2_{C_\ell^{m-3 +\alpha}} \approx [\tilde h]^2_{C_\ell^{m-3 +\alpha}}$ up to a constant depending on $A^{(0)}$, we thus find for $R = c_0 r$ with $c_0 > 1$ to be determined
\bals
	\inf_{p \in \mathcal P_{m-1}} \int_{Q_r} \abs{\tilde f - p}^2 \dd z &\leq  \int_{Q_r}\abs{\tilde f - p_{m-1}^{(0)}[\tilde f] - p_{m-1}^{(0)}[g_1]}^2 \dd z\\
	&\leq \int_{Q_r} \Abs{g_1 - p_{m-1}^{(0)}[g_1]}^2 \dd z  + \int_{Q_r} \abs{g_2}^2 \dd z \\
	&\leq C \Big(\frac{r}{R}\Big)^{n + 2m} \int_{Q_R} \abs{\tilde f-p_{m-1}^{(0)}[\tilde f]}^2 \dd z +  C   \int_{Q_{c_0r}} \abs{g_2}^2 \dd z \\
	&\leq C \Big(\frac{r}{R}\Big)^{n + 2m +2\alpha -2}\Big(\frac{r}{R}\Big)^{2-2\alpha} \int_{Q_R} \abs{\tilde f-p_{m-1}^{(0)}[\tilde f]}^2 \dd z \\
	&\quad+  C (c_0r)^{n+ 2m + 2\alpha-2}[\tilde h]^2_{C_\ell^{m-3 + \alpha}(Q_{c_0r})}.
\eals
Equivalently, 
\bals
	\big[\tilde f\big]_{\mathcal L^{2, n +2m+ 2\alpha-2}_{m-1}(Q_r)} \leq C \Big(\frac{1}{c_0}\Big)^{2-2\alpha}\big[\tilde f\big]_{\mathcal L^{2, n +2m+ 2\alpha-2}_{m-1}(Q_R)} + C c_0^{n+2m + 2\alpha-2}[\tilde h]_{C_\ell^{m-3 + \alpha}(Q_{R})}.
\eals
Thus by the characterisation of Campanato-norms with Hölder-norms in Theorem \ref{thm:equicampholderhigh} we have
\bals
	\big[\tilde f\big]_{C_\ell^{m-1+ \alpha}(Q_r)} \leq C\Big(\frac{1}{c_0}\Big)^{2-2\alpha} \big[\tilde f\big]_{C_\ell^{m-1+ \alpha}(Q_{c_0r})} + Cc_0^{n+2m + 2\alpha-2}[\tilde h]_{C_\ell^{m-3+\alpha}(Q_{c_0r})}.
\eals
Since $A, B, c, h \in C_\ell^{m-3+\alpha}(Q_1)$ we therefore obtain
\bal\label{eq:aux12}
	&[f]_{C_\ell^{m-1+ \alpha}(Q_{1/4})} \\
	&\leq [\tilde f]_{C_\ell^{m-1+ \alpha}(Q_r)} \leq C\Big(\frac{1}{c_0}\Big)^{2-2\alpha} \big[\tilde f\big]_{C_\ell^{m-1+ \alpha}(Q_{c_0r})} + Cc_0^{n+2m + 2\alpha-2}[\tilde h]_{C_\ell^{m-3+\alpha}(Q_{c_0r})}\\
	&\leq  C\Big(\frac{1}{c_0}\Big)^{2-2\alpha}  \big[ f\big]_{C_\ell^{m-1+ \alpha}((Q_{2\rho}(z_0))}  + C(c_0)\Big[\sum_{i, j} \big(a^{i, j}_{(0)} - a^{i, j}\big)\partial^2_{v_iv_j}\tilde f\Big]_{C_\ell^{m-3+\alpha}(Q_{2\rho}(z_0))}  \\
	&\quad +C(c_0) [b^i \partial_{v_i} f]_{C_\ell^{m-3+\alpha}(Q_{2\rho}(z_0))} +C(c_0, \rho) [\partial_{v_i} f]_{C_\ell^{m-3+\alpha}(Q_{2\rho}(z_0))} \\
	&\quad+C(c_0)[cf]_{C_\ell^{m-3+\alpha}(Q_{2\rho}(z_0))} +C(c_0, \rho)[f]_{C_\ell^{m-3+\alpha}(Q_{2\rho}(z_0))} + C(c_0)[h]_{C_\ell^{m-3+\alpha}(Q_{2\rho}(z_0))} \\
	&\leq C\Big(\frac{1}{c_0}\Big)^{2-2\alpha}  \big[ f\big]_{C_\ell^{m-1+ \alpha}((Q_{2\rho}(z_0))} +C(c_0) \rho^{m-3+\alpha}[D^2_v f]_{C_\ell^{m-3+\alpha}(Q_{2\rho}(z_0))} \\
	&\quad+ C(c_0)\rho^{m-3+\alpha}[D_v f]_{C_\ell^{m-3+\alpha}(Q_{2\rho}(z_0))}  + C(c_0)\rho^{m-3+\alpha} [f]_{C_\ell^{m-3+\alpha}(Q_{2\rho}(z_0))}\\
	&\quad+ C(\rho, c_0)[f]_{C_\ell^{m-3+\alpha}(Q_{2\rho}(z_0))}+ C(c_0, \rho)[D_vf]_{C_\ell^{m-3+\alpha}(Q_{2\rho}(z_0))}+ C(c_0)[h]_{C_\ell^{m-3+\alpha}(Q_{2\rho})}\\
	&\leq C_0 \left(\frac{1}{c_0}\right)^{2-2\alpha} \big[ f\big]_{C_\ell^{m-1+ \alpha}((Q_{2\rho}(z_0))}+ \frac{1}{4}[f]_{C_\ell^{m-1+\alpha}(Q_{2\rho}(z_0))} + C(\rho) \norm{f}_{L^\infty(Q_{2\rho}(z_0))} \\
	&\quad + C_1(c_0) \rho^{m-1+\alpha}[f]_{C_\ell^{m-1+\alpha}(Q_{2\rho}(z_0))} + C [h]_{C_\ell^{m-3+\alpha}(Q_\rho)},
\eal
where we have used Lemma \ref{lem:2.7} and Proposition \ref{prop:interpolation}.
Choosing first $c_0$ such that $C_0 \left(\frac{1}{c_0}\right)^{2-2\alpha} \leq \frac{1}{16}$ and then  $\rho = \rho(c_0)  >0$ such that $\frac{1}{16}+ C_1(c_0) \rho^{m-1+\alpha} + \frac{1}{4} \leq \frac{1}{2}$, we find for some $\beta > 0$
\bal\label{eq:covering-aux}
	[f]_{C_\ell^{m-1+\alpha}(Q_{\rho/4})} &\leq C \rho^{-\beta} \norm{f}_{L^\infty(Q_{2\rho})} + C [h]_{C_\ell^{m-3+\alpha}(Q_{2\rho})} + \frac{1}{2} [f]_{C_\ell^{m-1+\alpha}(Q_{2\rho}(z_0))}.
\eal
A standard iteration argument implies
\bal\label{eq:covering-conclusio}
	[f]_{C_\ell^{m-1+\alpha}(Q_{\rho/4})} &\leq C \norm{f}_{L^\infty(Q_{2\rho})} + C [h]_{C_\ell^{m-3+\alpha}(Q_{2\rho})},
\eal
where $C$ depends on $n, m, \alpha, \lambda_0,$ and the Hölder norms of all coefficients: if we define $\Psi(r) := [f]_{C_\ell^{m-1+\alpha}(Q_{r}(z_0))}$, then \eqref{eq:covering-aux} yields
\beqs
	\Psi\Big(\frac{\rho}{4}\Big) \leq C_2  \left(\frac{7\rho}{4}\right)^{-\beta}\left(\norm{f}_{L^\infty(Q_{2\rho})} + [h]_{C_\ell^{m-3+\alpha}(Q_{2\rho})}\right)  + \varepsilon \Psi(2\rho)
\eeqs
for $0\leq \varepsilon < 1$, $C_2 >0$ and $\beta > 0$. For some $0 < \tau < 1$ we then introduce
\bals
	\begin{cases}
		r_0 :=\frac{\rho}{4}, \\
		r_{i+1} := r_i + (1-\tau) \tau^i \frac{7\rho}{4}, \quad i \geq 0.
	\end{cases}
\eals
Since 
\beqs
	\sum_{i=1}^\infty \tau^i = \frac{\tau}{1-\tau},
\eeqs
we have that $r_i < 2\rho$ and inductively we prove that
\beqs
	\Psi(r_0) \leq \varepsilon^k \Psi(r_k) + C_0 \big(\norm{f}_{L^\infty(Q_\rho)} + [h]_{C_\ell^{m-3+\alpha}(Q_\rho)}\big) (1-\tau)^{-\beta}  \Big(\frac{7\rho}{4}\Big)^{-\beta}\sum_{i=0}^{k-1} \varepsilon^i \tau^{-i\beta}.
\eeqs
We choose $\tau$ such that $\varepsilon \tau^{-\beta} < 1$ so that letting $k \to \infty$ we deduce \eqref{eq:covering-conclusio}.

\subsection{Divergence form}\label{sec:div-form-loc}
The case of divergence form equations follows is similar by modifying the $\tilde h$ in \eqref{eq:tilde-h} as follows
\beqs
	\tilde h := \Bigg[\sum_{i, j}\partial_{v_i} \Big(a^{i, j} - a^{i, j}_{(0)}\Big)\partial_{v_j} f \Bigg] \eta + \sum_{i} \big(b^i \eta - 2 a^{i, j}_{(0)}\partial_i\eta\big)\partial_{v_i} f + \sum_{i, j} \big(c \eta + \mathcal T\eta - a^{i, j}_{(0)}\partial_{v_iv_j}^2 \eta\big) f + h \eta.
\eeqs
Note that for \eqref{eq:aux12} we will require $\nabla_v A \in C^{m-3+\alpha}(Q_1)$.

\section{Campanato's approach: the non-local (fractional) case}
\label{sec:schauder_frac}

We consider a solution $f$ to \eqref{eq:1.1nonloc} in $Q_1$ of class $C_\ell^\gamma([-1, 0]\times B_1\times \R^d)$ and assume that the non-negative kernel satisfies the ellipticity assumptions \eqref{eq:upperbound}, \eqref{eq:coercivity} and the Hölder condition \eqref{eq:holder}. Moreover, we further assume that it either satisfies the non-divergence form symmetry \eqref{eq:symmetry}, or that it verifies the divergence form symmetry \eqref{eq:cancellation1}, \eqref{eq:cancellation2}, and the additional Hölder condition \eqref{eq:holder_div}.

Let $\eta \in C_c^\infty((-1, 0] \times B_1 \times \R^d)$ so that $\eta = 1$ in $Q_{\frac{3}{4}}$ and $\eta = 0$ outside $Q_1$. Let $\tilde f = f\eta$. We freeze coefficients and write $K_0(w) = K(0, 0, 0, w)$ for the constant coefficient kernel; its corresponding operator $\mathcal L_0$ satisfies \eqref{eq:cc_kernel}. We compute for any $z \in \R^d$
\bals
	\mathcal T \tilde f - \mathcal L_0 \tilde f = h \eta + A \cdot \eta + B + f \mathcal T \eta
\eals
with 
\bals
	A(z) := \int_{\R^d} \big(f(w) - f(v)\big)\big[K(t, x, v, w) - K_0(w)\big]\dd w
\eals
and
\bals
	B(z) := \int_{\R^d} \big(\eta(v) - \eta(w)\big)f(w)K_0(w)\dd w.
\eals

We write
\beqs
	\tilde f - p[\tilde f] = g_1 + g_2,
\eeqs
where $g_1$ solves
\beqs
	\mathcal T g_1 - \mathcal L_0 g_1 = 0,
\eeqs
and with $$p[f] := f(z_0) + (t - t_0)\big(\mathcal Tf(z_0) - \mathcal L_0 f(z_0)\big)$$ for some $z_0 \in \R^{1+2d}$. With no loss of generality set $z_0 = (0, 0, 0)$. In particular, $g_2$ solves
\bals
	\mathcal T g_2 - \mathcal L_0 g_2 &= \mathcal T (\tilde f - p[\tilde f] - g_1) - \mathcal L_0 (\tilde f - p[\tilde f] - g_1) \\
	&=   h\cdot \eta + A \cdot \eta + B + f \mathcal T \eta-  \mathcal T\tilde f(z_0) + \mathcal L_0 \tilde f(z_0) \\
	&= h\cdot \eta + A \cdot \eta + B + f\mathcal T\eta- \big(h\cdot \eta + A \cdot \eta + B + f \mathcal T\eta\big)(z_0)\\
	&= \tilde h - \tilde h(z_0),
\eals
where $\tilde h:= h \eta + A  \eta + B + f\mathcal T \eta$.
For $g_1$ we find with Subsection \ref{subsec:cc_nonloc} for $0 < r < 1< R$
\bals
	\int_{Q_r} \Abs{g_1 - p_{2s}^{(0)}[g_1]}^2 \dd z &\leq C \left(\frac{r}{R}\right)^{n + 6s} \norm{g_1}^2_{L^\infty(\R^{1+2d})} \\
	&\leq C \left(\frac{r}{R}\right)^{n + 6s} \norm{\tilde f(\cdot) - p[\tilde f]}^2_{L^\infty(\R^{1+2d})} +C \left(\frac{r}{R}\right)^{n + 6s} \norm{g_2}^2_{L^\infty(\R^{1+2d})},
\eals
where $r > 0$ is such that $Q_r \subset Q_{1/2}$.
For $g_2$ we first perform a change of variables $g_{2, (0)}(t, x, v) := g_2\left(t, \kappa_0^{-\frac{1}{2s}} x, \kappa_0^{-\frac{1}{2s}} v\right)$ where $\kappa_0$ is such that $K_0(w) = \frac{\kappa_0}{\abs{w}^{d+2s}}$. Then $g_{2, (0)}$ solves 
\bals
	\Big(\partial_t + v\cdot \nabla_x +(-\Delta_v)^s \Big)g_{2, (0)}(t, x, v)  &= \Big(\partial_t   + v\cdot \nabla_x   + \mathcal L_0 \Big)g_2\Big(t, \kappa_0^{-\frac{1}{2s}} x, \kappa_0^{-\frac{1}{2s}} v\Big)\\
	&= \Big(\tilde h - \tilde h(0, 0, 0)\Big)\Big(t, \kappa_0^{-\frac{1}{2}} x, \kappa_0^{-\frac{1}{2}} v\Big)\\
	&=:  \Big(\tilde h_{(0)} - \tilde h_{(0)}(0, 0, 0)\Big)(t,  x,  v).
\eals
Thus by Lemma \ref{prop:2.1}
\bals
	\int_{Q_r} \abs{g_{2, (0)}}^2 \dd z \leq C r^n \Norm{g_{2, (0)}}_{L^\infty(Q_r)}^2 \leq C r^{n+ 4s + 2\alpha}\big[\tilde h_{(0)}\big]^2_{C_\ell^{\alpha}(Q_{r})}.
\eals
Since $\Norm{g_{2, (0)}}_{L^2} \approx \norm{g_2}_{L^2}$ and $\big[\tilde h_{(0)}\big]^2_{C_\ell^{\alpha}} \approx [\tilde h]^2_{C_\ell^{\alpha}}$ up to a constant depending on $\kappa_0$, since $\tilde f$ vanishes outside $Q_1$, and using that $\tilde h$ is compactly supported in time and space, we thus find
\bals
	\inf_{p \in \mathcal P_{m-1}} &\int_{Q_r} \abs{\tilde f - p}^2 \dd z \\
	&\leq  \int_{Q_r}\Abs{\tilde f - p[\tilde f] - p_{2s}^{(0)}[g_1]}^2 \dd z\\
	&\leq C  \left(\frac{r}{R}\right)^{n + 6s}\norm{\tilde f(\cdot) - p[\tilde f]}^2_{L^\infty(\R^{1+2d})} + C\frac{r^{n+ 10s+2\alpha } }{R^{n+6s}}[\tilde h]_{C^{\alpha}_\ell(Q_{2r}^v \times \R^d)}^2+ Cr^{n+ 4s+ 2\alpha}[\tilde h]^2_{C_\ell^{\alpha}(Q_{r})}\\
	&\leq C \frac{r^{n + 6s}}{R^{n+4s-2\alpha}}[\tilde f]_{C^{2s+\alpha}(Q_{R})}^2 + Cr^{n+ 10s+2\alpha } [\tilde h]_{C^{\alpha}_\ell(Q_{2r}^v \times \R^d)}^2+ Cr^{n+ 4s+ 2\alpha}[\tilde h]^2_{C_\ell^{\alpha}(Q_{r})}\\
	&\leq Cr^{n +4s + 2\alpha} \Bigg(\left(\frac{r}{R}\right)^{2(s-\alpha)}[\tilde f]_{C^{2s+\alpha}(Q_{R})}^2+ [\tilde h]^2_{C_\ell^{\alpha}(Q_{R}^v \times \R^d)}\Bigg).
\eals
In the last inequality we used $\alpha < s$ since $\alpha = \frac{2s}{1+2s}\gamma < \frac{2s}{1+2s} \min(1, 2s)$. 
Equivalently, 
\bals
	[\tilde f]_{\mathcal L^{2, n +4s + 2\alpha}_{2s}(Q_r)} \leq C\left(\frac{r}{R}\right)^{2(s-\alpha)}[\tilde f]_{C^{2s+\alpha}(Q_{R})} + C [\tilde h]_{C_\ell^{\alpha}(Q_{R}^v \times \R^d)}.
\eals
Thus by the characterisation of Campanato norms with Hölder norms in Theorem \ref{thm:equicampholderhigh} we have for all $0< r < 1  < R$
\bals
	[\tilde f]_{C_\ell^{2s+ \alpha}(Q_r)} \leq C\left(\frac{r}{R}\right)^{2(s-\alpha)}[\tilde f]_{C^{2s+\alpha}(Q_{R})}+ C[\tilde h]_{C_\ell^{\alpha}(Q_{R}^v \times \R^d)}.
\eals

It remains to bound the $C^\alpha_\ell$-norm of $\tilde h = h \eta + A \eta + B + f \mathcal T \eta$. 
We claim 
\bal
	[A]_{C_\ell^\alpha(Q_{\frac{1}{2}})} \lesssim A_0\big(\norm{f}_{C_\ell^{2s+\alpha}(Q_1)} + \norm{f}_{C_\ell^{\gamma}((-1, 0] \times B_1 \times \R^d)}\big).
\label{eq:claim1}
\eal
To justify our claim, we write $A(z_1) - A(z_2) = I_1 + I_2$ with
\bals
	&I_1 = \int \big(f(z_2 \circ (0, 0, w)) - f(z_2)\big)\big[K_{z_1}(w) - K_{z_2}(w)\big] \dd w, \\
	&I_2 = \int \big(f(z_1 \circ (0, 0, w)) - f(z_1) - f(z_2 \circ (0, 0, w)) + f(z_2)\big)\big[K_{z_1}(w) - K_0(w)\big] \dd w.
\eals
For $I_1$ we distinguish the far and the close part and write $I_{11}$ and $I_{12}$ respectively. Then for the far part there holds with \eqref{eq:3.10holder}
\bals
	\abs{I_{11}} \leq \norm{f}_{L^\infty((-1, 0] \times B_1 \times \R^d)} \int_{\abs{w}\geq 1} \Abs{K_{z_1}(w) - K_{z_2}(w)} \dd w \lesssim A_0\norm{f}_{L^\infty((-1, 0] \times B_1 \times \R^d)}d_\ell(z_1, z_2)^\alpha.
\eals

For the close part we have in case of the \textit{non-divergence form symmetry} \eqref{eq:symmetry} and Lemma \ref{lem:2.7}
\bals
	\abs{I_{12}} &\leq \int_{\abs{w}\leq 1}\Abs{f(z_2 \circ (0, 0, w)) - p_{2s}^{z_2\circ(0, 0, w)}[f]} \Abs{K_{z_1}(w)- K_{z_2}(w)} \dd w \\
	&\quad+ \frac{1}{2} \int_{\abs{w}\leq 1}\Abs{p_{2s}^{z_2\circ(0, 0, w)}[f] +p_{2s}^{z_2\circ(0, 0, -w)}[f] - f(z_2)} \Abs{K_{z_1}(w)- K_{z_2}(w)} \dd w\\
	&\lesssim [f]_{C_\ell^{2s + \alpha}}\int_{\abs{w}\leq 1}\abs{w}^{2s+\alpha} \Abs{K_{z_1}(w)- K_{z_2}(w)} \dd w + \Abs{D_v^2 f} \int_{\abs{w}\leq 1}\abs{w}^{2}\Abs{K_{z_1}(w)- K_{z_2}(w)} \dd w\\
	&\lesssim A_0 \norm{f}_{C_\ell^{2s + \alpha}}d_\ell(z_1, z_2)^\alpha.
\eals

If instead we assume the \textit{divergence form symmetry} \eqref{eq:cancellation1} and \eqref{eq:cancellation2} we get
\bals
	\abs{I_{12}} &\leq  \abs{\int_{\abs{w}\leq 1}\left(f(z_2 \circ (0, 0, w)) - p_{2s}^{z_2\circ(0, 0, w)}[f]\right)\big(K_{z_1}(w)- K_{z_2}(w)\big) \dd w}\\
	&\quad+ \abs{ \int_{\abs{w}\leq 1}\left(p_{2s}^{z_2\circ(0, 0, w)}[f]  - f(z_2)\right)\big(K_{z_1}(w)- K_{z_2}(w)) \dd w}\\
	&\lesssim [f]_{C_\ell^{2s + \alpha}}\int_{\abs{w}\leq 1}\abs{w}^{2s+\alpha} \Abs{K_{z_1}(w)- K_{z_2}(w)} \dd w \\
	&\quad + \Abs{D_v f}\Big\vert\ \textrm{PV}\int_{\abs{w}\leq 1} w \big(K_{z_1}(w)- K_{z_2}(w)\big) \dd w\Big\vert\\\
	&\quad + \Abs{D_v^2 f} \int_{\abs{w}\leq 1}\abs{w}^{2}\Abs{K_{z_1}(w)- K_{z_2}(w)} \dd w\\
	&\lesssim A_0 \norm{f}_{C_\ell^{2s + \alpha}}d_\ell (z_1, z_2)^\alpha,
\eals
by assumption \eqref{eq:holder} and \eqref{eq:holder_div}.

To estimate $I_2$ we can use Lemma \ref{lem:3.7}. This proves the claim.

We further claim 
\bal
	[B]_{C_\ell^\alpha(Q_{2r}^v \times \R^d)} \lesssim \norm{f}_{C_\ell^{\gamma}((-1, 0] \times B_1 \times \R^d)}.
\label{eq:claim2}
\eal
For $z_2 \in Q_r$ we compute $B(z_2) - B(z_1) = J_1 + J_2$ with
\bals
	&J_1 = \int \big[\eta(z_1 \circ (0, 0, w)) - \eta(z_1) - \eta(z_2 \circ (0, 0, w))+ \eta(z_2)\big]f\big(z_1 \circ (0, 0, w)\big)K_0(w)\dd w, \\
	&J_2 = \int_{\abs{w} > r/4}\big[\eta(z_2 \circ (0, 0, w)) - \eta(z_2)\big]\big[f(z_1 \circ (0, 0, w)) - f(z_2 \circ (0, 0, w))\big]K_0(w)\dd w
\eals
Since $\eta$ is smooth we can apply Lemma \ref{lem:3.7} and get 
\bals
	\abs{J_1} \leq C\norm{f}_{L^\infty((-1, 0]\times B_1 \times \R^d)} d_\ell(z_1, z_2)^\alpha.
\eals
For $J_2$ we have
\bals
	\abs{J_2} \leq 2 \norm{\eta}_{L^\infty} [f]_{C_\ell^{\gamma}}\int_{\abs{w} > r/4} d_\ell(z_1\circ(0, 0, w), z_2\circ (0, 0, w))^\gamma K_0(w) \dd w.
\eals
Since $\alpha = \frac{2s\gamma}{1+2s}$ we have
\bals
	\abs{J_2} &\lesssim  [f]_{C_\ell^{\gamma}}\int_{\abs{w} > r/4} d_\ell(z_1\circ(0, 0, w), z_2\circ (0, 0, w))^\gamma K_0(w) \dd w \\
	&\lesssim [f]_{C_\ell^{\gamma}}\int_{\abs{w} > r/4} \big(d_\ell(z_1, z_2) + \abs{t_1 -t_2}^{\frac{1}{1+2s}}\abs{w}^{\frac{1}{1+2s}}\big)^\gamma K_0(w) \dd w \\
	&\lesssim [f]_{C_\ell^{\gamma}}\int_{\abs{w} > r/4} \big(1 + \abs{w}^{\frac{\gamma}{1+2s}}\big) d_\ell(z_1, z_2)^{\frac{2s\gamma}{1+2s}} K_0(w) \dd w \\
	&\lesssim_{\Lambda} [f]_{C_\ell^{\gamma}} d_\ell(z_1, z_2)^{\frac{2s\gamma}{1+2s}} =C  [f]_{C_\ell^{\gamma}} d_\ell(z_1, z_2)^\alpha,
\eals
where we used the upper bound on $K_0$ \eqref{eq:upperbound}. This proves the second claim \eqref{eq:claim2}.

By combining \eqref{eq:claim1} with \eqref{eq:claim2} and by choosing $R = c_0 r$ for some $c_0 > 1$ for any $0 < r$, we deduce for some $C_0 > 0$
\bal
	\norm{f}_{C_\ell^{2s + \alpha}(Q_{\frac{1}{4}})} \leq C(1+ A_0)\norm{f}_{C_\ell^{\gamma}((-1, 0] \times B_1\times \R^d)} + C_0\left(A_0+c_0^{-(s-\alpha)}\right)\norm{f}_{C_\ell^{2s + \alpha}(Q_1)} +C\norm{h}_{C_\ell^{\alpha}(Q_1)}.
\label{eq:prop5.4}
\eal
Without loss in generality we can assume that $A_0 < 1$, otherwise we scale the equation initially. Then we pick $c_0$ such that $C_0 \left(A_0+c_0^{-(s-\alpha)}\right)\leq \frac{1}{2}$. With the same iteration argument that was outlined in Subsection \ref{sec:nondiv-perturb} (which is a standard iteration argument), we conclude
\[
	\norm{f}_{C_\ell^{2s + \alpha}(Q_{\frac{1}{4}})} \leq C\Big(\norm{h}_{C_\ell^\alpha(Q_1)} + \norm{f}_{C_\ell^{\gamma}((-1, 0] \times B_1\times \R^d)}\Big),
\]
where $C$ depends on $s, d, \lambda_0, \Lambda_0, A_0$.


\appendix

\section{Hypoelliptic Operators}\label{app:hypo-sec}
\subsection{Toolbox}\label{app:hypo}
In this section, we briefly outline that our approach is robust enough to deal with general second order Kolmogorov equations of the form
\bal\label{eq:hypo}
	\mathscr L f(t, x) := \sum_{N-d \leq i, j \leq N} a_{i, j}(t, x) \partial_{x_ix_j} f(t, x) &+ \sum_{1 \leq i, j \leq N} \tilde b_{i, j} x_j  \partial_{x_i} f(t, x) - \partial_t  f(t, x) \\
	& + \sum_{N-d \leq i \leq N} b_i(t, x)  \partial_i f(t, x) + c(t, x) f(t, x) = h,
\eal
where $z = (t, x) = (t, x_0, x_1, \dots, x_\kappa) \in \R^{1+N}$, $\kappa \geq 1$ is the number of commutators, and $1 \leq d \leq N$. The velocity variable corresponds to the last entry $x_\kappa \in \R^{d}$ . The diffusion matrix $A(z) = \big(a_{i, j}(z)\big)_{N-d\leq i, j\leq N}$ is symmetric with real measurable entries, and uniformly elliptic \eqref{eq:unifellip}. The matrix $\tilde B = \big(\tilde b_{i, j}\big)_{1\leq i, j \leq N}$ has constant entries and satisfies suitable assumptions such that the \textit{principal part operator} $\mathscr K$ of $\mathscr L$ with respect to the kinetic degree, given by
\beq\label{eq:principal}
	\mathscr K f(t, x) = \sum_{N-d \leq i, j \leq N} \partial_{x_ix_j} f(t, x) + \sum_{1 \leq i, j \leq N} \tilde b_{i, j} x_j  \partial_{x_i} f(t, x) - \partial_t  f(t, x),
\eeq
is hypoelliptic, i.e. any distributional solution of $\mathscr K f = h$ is smooth whenever $h \in C^\infty$. In particular, this assumption coincides with $\tilde B$ having constant real entries and taking the form
\bal\label{eq:tildeB}
	\tilde B = 
	\begin{pmatrix}
		\ast & \tilde B_1 & 0 & \dots & 0 \\
		\ast & \ast & \tilde B_2 & \dots & 0 \\
		\vdots & \vdots & \vdots & \ddots & \vdots \\
		\ast & \ast & \ast &\dots & \tilde B_\kappa \\
		\ast & \ast & \ast &\dots & \ast
	\end{pmatrix},
\eal
where each $\tilde B_i$ is a $d_{i-1} \times d_{i}$ matrix of rank $d_i$ with $d := d_\kappa \geq d_{\kappa-1} \geq \dots \geq d_0 \geq 1$ and $\sum_{i = 0}^\kappa d_i = N$.
For further discussion on this operator, we refer the reader to \cite[Section 1 and 2]{lanconelli-polidoro}. We remark that the principal part operator $\mathscr K$ is still invariant under Galilean transformation \eqref{eq:galilean}. Moreover, $\mathscr K$ is invariant under the scaling given by
\beq\label{eq:scaling-hypo}
	(t, x_0, \dots, x_{\kappa}) \to (r^2t, r^3 x_0, \dots, r^{2\kappa +1}x_{\kappa-1}, r x_\kappa) =: z_r,
\eeq
for $r > 0$, where $\kappa\geq 1$ is the number of commutators, \textit{if and only if} all the $\ast$-blocks in $\tilde B$ are zero \cite[Proposition 2.2]{lanconelli-polidoro}. We denote the scaling invariant principal part by $\mathscr K_0$, and emphasise that it is of the form \eqref{eq:principal} with the matrix $\tilde B$ as in \eqref{eq:tildeB} where all the $\ast$-entries are zero. The cylinders will be defined respecting the scaling invariance, similar as above \eqref{eq:cylinder}. 

We briefly sketch how to obtain Schauder estimates for a solution $f$ of \eqref{eq:hypo} in $Q_1$. Note that the kinetic distance and the corresponding Hölder norms have to be defined more generally taking into account the scaling \eqref{eq:scaling-hypo}. 

First, the regularity estimates will be replaced by an argument of Hörmander \cite[Theorem 3.7]{hormander} as follows. Any solution $f$ of $\mathscr Kf = 0$ satisfies for $l \geq 1$
\beq\label{eq:hormander}
	\norm{D^l f}_{L^\infty(Q_{r}(z_0))} \leq C(l, N) \norm{f}_{L^2(Q_R(z_0))},
\eeq
where $D^l$ is a differential of order $l$. To see this, let $\delta$ be a multi-index such that $\abs\delta = l \geq 1$.
Let $G \subset L^2(Q_R(z_0))$ be defined as
\beqs
	G := \big\{ g \in L^2(Q_R(z_0)) \cap C^\infty(Q_R(z_0)) : \mathscr K g  = 0 \textrm{ in } Q_R(z_0)\big\}.
\eeqs	
Due to the hypoellipticity of $\mathscr K$  the subspace $G$ is closed in $L^2(Q_R(z_0))$. Define $\mathcal B : G \to C^0(Q_r(z_0))$ by $\mathcal B g = D^\delta g \vert_{Q_r(z_0)}$ for $\delta$ such that $\abs{\delta} = l \geq 0$. Then $\mathcal B$ has closed graph in $G \times C^0(Q_r(z_0))$, and thus, by virtue of the closed graph theorem we conclude \eqref{eq:hormander}. Then we derive Campanato's inequality \eqref{eq:camp-first} just as above in Subsection \ref{subsec:0_loc}. 

Second, the principal part operator $\mathscr K_0$ admits an explicit fundamental solution given in \cite[Equation (2.7)]{AR}. In particular, it satisfies for $r > 0$
\beq\label{eq:kolmo-scaling}
	\mathscr K_0 f_r = r^2 \big(\mathscr K_0f\big)_r, 
\eeq
where $f_r$ denotes the rescaled function $f_r(z) := f(z_r)$. Note that we do not require the scaling invariance to deduce the Schauder estimates. To see this, we denote the fundamental solution of $\mathscr K$ by $\Gamma$ and the fundamental solution of $\mathscr K_0$ by $\Gamma_0$, respectively. Then we can use the upper bound on $\Gamma$ by $\Gamma_0$, stated in \cite[Theorem 3.1]{lanconelli-polidoro},
\beq\label{eq:upperbound-fund-sol}
	\Gamma(z) \leq a \Gamma_0(z),
\eeq
for some $a > 0$. Due to \eqref{eq:kolmo-scaling} we then have the good scaling for $g_2$, where $g_2$ comes from the splitting of our solution $f - p_2^{(0)}[f] = g_1 + g_2$ as done in Section \ref{sec:schauder_nonfrac} above, with the polynomial $p_2^{(0)}[f]$ given in \eqref{eq:p2-hypo} below.
Alternatively, we can directly consider the scaling of the full matrix $\tilde B$ in \eqref{eq:tildeB}. According to \cite[Remark 3.2]{lanconelli-polidoro} and \cite[Remark 2.4]{lanconelli-pascucci-polidoro}, the $\ast$-blocks in \eqref{eq:tildeB} scale to some higher power of $r$ than the superdiagonal blocks. Thus, using $\tilde B = \tilde B_0 + \tilde B - \tilde B_0$, where $\tilde B_0$ corresponds to $\tilde B$ with all $\ast$-blocks equal to zero, we rewrite 
\beqs
	\mathscr K = \mathscr K_0  +  \sum_{1 \leq i, j \leq N} \big(\tilde b_{i, j} - \tilde b^0_{i, j} \big)x_j  \partial_{x_i} f,
\eeqs
so that 
\beqs
	\mathscr K^a_0 g_2 = \tilde h + \sum_{1 \leq i, j \leq N} \big(\tilde b^0_{i, j} - \tilde b_{i, j} \big)x_j  \partial_{x_i} f
\eeqs
with $\mathscr K^a_0$ as in \eqref{eq:Ka} but where $\tilde B_0$ replaces $\tilde B$, and with $\tilde h$ given in \eqref{eq:tilde-h-hypo}. The right hand side can be bounded as in Section \ref{sec:schauder_nonfrac} above, since the term $\sum_{1 \leq i, j \leq N} \big(\tilde b_{i, j} - \tilde b^0_{i, j} \big)x_j  \partial_{x_i} f$ scales like a lower order term due to \cite[Remark 3.2]{lanconelli-polidoro} and \cite[Remark 2.4]{lanconelli-pascucci-polidoro}. The details of the splitting are done for Dini-continuous coefficients in Subsection \ref{sec:dini} below.

\subsection{Hölder coefficients}
We have assembled the toolbox required for the Schauder estimates, and the argument of Section \ref{sec:schauder_nonfrac} goes through (with suitable modifications as outlined above in Subsection \ref{app:hypo}), so that we derive
\begin{theorem}[Schauder estimate for Kolmogorov operators]\label{thm:hypo}
Let $\alpha \in (0, 1)$ be given. Let $m \geq 3$ be some integer. Let $f$ solve \eqref{eq:hypo} in $Q_1$. Suppose $A \in C_\ell^{m-3+ \alpha}(Q_1)$ satisfies \eqref{eq:unifellip} for some $\lambda_0 > 0$, where $d = d_0$, and assume $B, c, h \in C_\ell^{m-3+\alpha}(Q_1)$. We further assume that the principal part operator $\mathscr K$ defined in \eqref{eq:principal} is hypoelliptic, i.e. $\tilde B$ is of the form \eqref{eq:tildeB}.
Then there holds
\beqs
	\norm{f}_{C_\ell^{m-1 + \alpha}(Q_{1/4})} \leq C \Big(\norm{f}_{L^\infty(Q_1)}  + \norm{h}_{C_\ell^{m-3+\alpha}(Q_1)}\Big),
\eeqs
for some $C$ depending on $N, \lambda_0, \alpha, \norm{A}_{C_\ell^{m-3+\alpha}},  \norm{B}_{C_\ell^{m-3+\alpha}},  \norm{c}_{C_\ell^{m-3+\alpha}}$.
\end{theorem}
Similar to Subsection \ref{sec:div-form-loc} the divergence form case just follows by realising that any divergence form equation can be written in non-divergence form plus an additional lower order term, provided that $\nabla_{x_\kappa} A \in C_\ell^{m-3+ \alpha}(Q_1)$. Finally, we can derive a Schauder-type estimate under less stringent assumptions assuming Dini-regularity instead of Hölder regularity, inspired from \cite{PSR}.

\subsection{Dini Coefficients}
\label{sec:dini}
We point out a structural peculiarity when we consider more generally Dini-regular coefficients $A, B, c$ and source $h$. We denote by $\omega_g$ the modulus of continuity of a function $g$ on a subset $Q \subset \R^{1+N}$, given by
\beqs
	\omega_g(\ln r) := \sup_{\substack{z_1, z_2 \in Q\\ d_\ell(z_1, z_2) < r}} \big\vert g(z_1) - g(z_2) \big\vert.
\eeqs
A function $g$ is said to be Dini-continuous in $Q$ if 
\beqs
	\int_0^1 \frac{\omega_g(\ln r)}{r} \dd r = \int_{-\infty}^0 \omega_g(\rho) \dd \rho < + \infty.
\eeqs
We aim to show:
\begin{theorem}\label{thm:dini}
Let $f$ solve \eqref{eq:hypo} in $Q_1$ such that $A$ is a symmetric, uniformly elliptic matrix with real measurable entries, and suppose $\tilde B$ has constant entries. Assume that the principal part operator $\mathscr K$ \eqref{eq:principal} is hypoelliptic, i.e. $\tilde B$ is of the form \eqref{eq:tildeB}. Suppose that the coefficients $A, B, c$ and the source $h$ are Dini-regular. Then, for any $z, z_0 \in \R^{1+N}$ such that $d_\ell(z, z_0) < 1/2$, $f$ satisfies
\bal\label{eq:dini}
	\big\vert D^2 f(z) &- D^2 f(z_0)\big\vert \\
	&\leq C  \Bigg(\int^{\ln d_\ell(z, z_0)}_{-\infty} \omega_{A}(\xi) \dd \xi + d_\ell(z, z_0)\int_{\ln d_\ell(z, z_0)}^0 \omega_{A}(\xi) e^{-\xi} \dd \xi\Bigg)\sum_{i, j} \sup_{Q_1} \big\vert \partial_{v_iv_j}^2 f\big\vert\\
	&\quad+ C\Bigg(d_\ell(z, z_0)+ \int^{\ln d_\ell(z, z_0)}_{-\infty} \omega_{c}(\xi) \dd \xi + d_\ell(z, z_0)\int_{\ln d_\ell(z, z_0)}^0 \omega_{c}(\xi) e^{-\xi} \dd \xi\Bigg)\sup_{Q_1} \big\vert f\big\vert \\
	&\quad+C\Bigg(\int^{\ln d_\ell(z, z_0)}_{-\infty} \omega_{B}(\xi)\dd \xi +d_\ell(z, z_0) \int_{\ln d_\ell(z, z_0)}^0 \omega_{B}(\xi)e^{-\xi} \dd \xi\Bigg)\sum_{i} \sup_{Q_1} \big\vert \partial_{v_i} f\big\vert \\
	&\quad+C\int^{\ln d_\ell(z, z_0)}_{-\infty} \omega_{h}(\xi)\dd \xi +C d_\ell(z, z_0) \int_{\ln d_\ell(z, z_0)}^0 \omega_{h}(\xi)e^{-\xi} \dd \xi+ Cd_\ell(z, z_0)\sup_{Q_1} \abs{h}.
\eal
Here $D^2$ is a differential of order $2$, and $C = C(N, \lambda_0)$. 
\end{theorem}
In particular we recover Theorem 1.6 of \cite{PSR}. 
\begin{remark}
Theorem \ref{thm:dini} suggests that Dini continuity is the suitable notion of regularity for Schauder estimates. In particular, in Theorem \ref{thm:hypo}, we see that Hölder regular solutions $f$ are \textit{fixed points} of the Schauder estimates.
\end{remark}

For this purpose, we consider $0 <  \rho \leq 1$ to be determined and a solution $f$ of \eqref{eq:1.1loc_nondiv} in $Q_1$. 
Let $\eta \in C_c^\infty(\R^{1+2d})$ be a cut-off with $0 \leq \eta \leq 1$, such that $\eta = 1$ in $Q_\rho$ and $\eta = 0$ outside $Q_{2\rho}$. Let $\tilde f = f \cdot \eta$. With no loss in generality we set $z_0 = (0, 0, 0)$. We denote with $p_{2}^{(z_0)}[f]$ the Taylor polynomial of $f$ at $z_0$ with kinetic degree less or equal to $2$, given by
\bal\label{eq:p2-hypo}
	p_{2}^{z_0}[f](z) = f(z_0) &+ \sum_{N-d \leq i \leq N}  \partial_{x_i} f(z_0) \big(z^{(i)} - z_0^{(i)}\big) + \frac{1}{2} \sum_{N-d \leq i, j \leq N}  \partial^2_{x_ix_j} f(z_0) \big(z^{(i)} - z_0^{(i)}) (z^{(j)} - z_0^{(j)}\big) \\
	&+ \Bigg[\sum_{1 \leq i, j \leq N} \tilde b_{i, j} x_j  \partial_{x_i} f(z_0) - \partial_t  f(z_0)\Bigg] (t - t_0),
\eal
where $z^{(i)}$ denotes the element at index $i$.
We then write
\bal\label{eq:lim}
	\tilde f - p_{2}^{(0)}[\tilde f] = \tilde f - \tilde f_{k} + \tilde f_k - p_{2}^{(0)}[\tilde f],
\eal
where each $f_k$ solves 
\bal\label{eq:f-k}
	\mathscr K^a f_k = \tilde h(0, 0, 0),
\eal
in $\mathcal Q_k := Q_{\rho^{k}}$, with the constant coefficient operator $\mathscr K^a$ given by 
\beq\label{eq:Ka}
\mathscr K^a := \sum_{N-d\leq i, j\leq N} a^{(0)}_{i, j} \partial_{x_i x_j}^2 +  \sum_{1\leq i, j\leq N} \tilde b_{i, j} x_j\partial_{x_i} - \partial_t
\eeq
for $a_{(0)}^{i, j}  =a^{i, j}(z_0)$, and the right hand side $\tilde h$ given by
\beq\label{eq:tilde-h-hypo}
	\tilde h := \sum_{N-d \leq i, j\leq N} \Big(a_{i, j}^{(0)}-a_{i, j} \Big)\partial^2_{x_ix_j} f \cdot \eta  + \sum_{N-d \leq i\leq N} \big(2a_{i, j}^{(0)} \partial_{x_i}\eta -b_i \eta \big)\partial_{x_i} f +  (-c \eta +\mathscr K_a\eta) f + h \cdot\eta.
\eeq
In particular, there holds
\beq\label{eq:g-1}
	\mathscr K^a \big(\tilde f_k -\tilde f_{k+1}\big) = 0, \qquad \textrm{ in } \mathcal Q_{k+1},
\eeq
and 
\beq\label{eq:g-2}
	\mathscr K^a \big(\tilde f - \tilde f_{k}\big) = \tilde h - \tilde h(0, 0, 0), \qquad \textrm{ in } \mathcal Q_k.
\eeq
On the one hand, we first perform a constant change of variables to rewrite $\mathscr K^a$ in terms of $\mathscr K$, as was done in \eqref{eq:change-var-A}. Then, due to \eqref{eq:g-2}, the upper bound of the fundamental solution \eqref{eq:upperbound-fund-sol} and the scaling \eqref{eq:kolmo-scaling}, which extends Lemma \ref{prop:2.1}, we bound for any $k \geq 1$
\beqs
	\int_{\mathcal Q_{k+1}} \Abs{\tilde f - \tilde f_k}^2 \dd z \leq C \rho^{(n+4)(k+1)} \sup_{\mathcal Q_{k+1}} \big\vert \tilde h - \tilde h(0, 0, 0)\big\vert^2 \leq C  \rho^{(n+4)(k+1)}\omega^2_{\tilde h}(\rho^{k+1}).
\eeqs
Since $\tilde f_k = \tilde f_0 + \sum_{l=0}^{k-1} \tilde f_{l+1} - \tilde f_l$ we thus find
\bal\label{eq:lim-aux-2}
	\Bigg(\rho^{-(n+6)(k+1)}\int_{\mathcal Q_{k+1}} \Abs{\tilde f - \tilde f_k}^2 \dd z\Bigg)^{\frac{1}{2}} &\leq \Bigg(\sum_{l=0}^{k-1}\rho^{-(n+6)(l+1)}\int_{\mathcal Q_{l+1}} \Abs{\tilde f_{l+1} - \tilde f_l}^2 \dd z\Bigg)^{\frac{1}{2}} \\
	&\leq \Bigg\{\sum_{l=0}^{k-1}\rho^{-(n+6)(l+1)}\Bigg(\int_{\mathcal Q_{l+1}} \Abs{\tilde f_{l+1} - \tilde f}^2 + \Abs{\tilde f - \tilde f_l}^2 \dd z\Bigg)\Bigg\}^{\frac{1}{2}} \\
	&\leq C \sum_{l=0}^{k-1}\frac{\omega_{\tilde h}(\ln \rho^{l+1})}{\rho^{l+1}}\\
	&\leq C \int_{\ln \rho}^0 \omega_{\tilde h}(\xi) e^{-\xi} \dd \xi.
\eal
On the other hand, we note that $p_{2}^{(0)}[\tilde f] = \lim_{k\to \infty} \tilde f_k$. This is because $p_{2}^{(0)}[\tilde f]$ is the Taylor polynomial of $\tilde f$, so that 
\beqs
	\sup_{\mathcal Q_{k}} \big(\tilde f - p_2^{(0)}[\tilde f]\big) =  o(\rho^{2k}),
\eeqs
and we also refer to \cite[Equation (5.16)]{PSR}.
Moreover, due to \eqref{eq:g-2} we can use \eqref{eq:kolmo-scaling} so that overall we find
\bals
	\big\vert \tilde f_{k}(z) - p_2^{(0)}[\tilde f](z) \big\vert &\leq \sup_{\mathcal Q_{k}} \big\vert \tilde f_{k} - \tilde f\big\vert + \sup_{\mathcal Q_{k}} \big\vert \tilde f - p_2^{(0)}[\tilde f]\big\vert \\
	&\leq C\rho^{2k}  \sup_{\mathcal Q_{k}} \big\vert \tilde h - \tilde h(0, 0, 0)\big\vert + o(\rho^{2k})\\
	&\leq C\rho^{2k} \omega_{\tilde h}(\ln \rho^{k}) + o(\rho^{2k})\\
	&\leq o(\rho^{2k}).
\eals
Therefore, we may write
\beq\label{eq:sum}
	\tilde f_k - p_{2}^{(0)}[\tilde f] = \sum_{l = k}^\infty \tilde f_l - \tilde f_{l+1}.
\eeq
Due to \eqref{eq:f-k}, Subsection \ref{subsec:0_loc} (suitably making the replacements for the more general equation as outlined in Subsection \ref{app:hypo}), \eqref{eq:sum}, \eqref{eq:g-2} and \eqref{eq:kolmo-scaling}, we then find for $\tilde f_k - p_2^{(0)}[\tilde f]$
\bals
	\int_{\mathcal Q_{k+1}}\Abs{\tilde f_k - p_2^{(0)}[\tilde f] - p^{(0)}_2\big[\tilde f_k -p_2^{(0)}[\tilde f]\big]}^2 \dd z &\leq C \Big(\frac{\rho^{k+1}}{\rho^{k}}\Big)^{n + 6} \int_{\mathcal Q_{k}} \Abs{\tilde f_k - p_2^{(0)}[\tilde f]}^2 \dd z\\
	&= C \Big(\frac{\rho^{k+1}}{\rho^{k}}\Big)^{n + 6}  \sum_{l=k}^\infty \int_{\mathcal Q_{k}} \Abs{\tilde f_l - \tilde f_{l+1}}^2 \dd z\\
	&\leq C \Big(\frac{\rho^{k+1}}{\rho^{k}}\Big)^{n + 6}  \sum_{l=k}^\infty \Bigg( \int_{\mathcal Q_{l}}\Abs{\tilde f_l - \tilde f}^2\dd z +\int_{\mathcal Q_{l}}\Abs{\tilde f -  \tilde f_{l+1}}^2\dd z \Bigg)\\
	&\leq C \Big(\frac{\rho^{k+1}}{\rho^{k}}\Big)^{n + 6} \sum_{l=k}^\infty  \rho^{l(n + 4)}\omega^2_{\tilde h}(\ln\rho^{l})\\
	&\leq C\rho^{(k+1)(n + 6)} \rho^{-2k}\sum_{l=k}^\infty \omega^2_{\tilde h}(\ln\rho^{l}), 
\eals
or equivalently
\bal\label{eq:lim-aux-1}
	\Bigg(\rho^{-(n+6)(k+1)}\int_{\mathcal Q_{k+1}} \Abs{\tilde f_k - p_2^{(0)}[\tilde f] - p^{(0)}_2\big[\tilde f_k -p_2^{(0)}[\tilde f]\big]}^2 \dd z\Bigg)^{\frac{1}{2}}&\leq C\rho^{-k}\sum_{l=k}^\infty \omega_{\tilde h}(\ln \rho^{l}) \\
	&\leq C \rho^{-k} \int^{\ln \rho}_{-\infty} \omega_{\tilde h}(\xi)\dd \xi.
\eal

Thus due to \eqref{eq:lim}, \eqref{eq:lim-aux-2} and \eqref{eq:lim-aux-1} we conclude
\beqs
	\Bigg(\rho^{-(n+6)(k+1)}\int_{\mathcal Q_{k+1}} \Abs{\tilde f - p_2^{(0)}[\tilde f]}^2 \dd z\Bigg)^{\frac{1}{2}} \leq C \rho^{-(k+1)} \int^{\ln\rho}_{-\infty}\omega_{\tilde h}(\xi) \dd \xi + C \int_{\ln\rho}^0 \omega_{\tilde h}(\xi)e^{-\xi} \dd \xi.
\eeqs
The right hand side will further be bounded using the explicit form of $\tilde h$ in \eqref{eq:tilde-h-hypo}:
\bals
	\rho^{-(k+1)}\int^{\ln\rho}_{-\infty} &\omega_{\tilde h}(\xi) \dd \xi + \int_{\ln\rho}^0 \omega_{\tilde h}(\xi)e^{-\xi} \dd \xi \\
	&\lesssim  \Bigg(\rho^{-(k+1)}\int^{\ln\rho}_{-\infty} \omega_{A}(\xi) \dd \xi + \int_{\ln\rho}^0 \omega_{A}(\xi)e^\xi\dd \xi\Bigg)\sum_{1\leq i, j\leq d_0} \sup_{Q_1} \big\vert \partial_{x_ix_j}^2 f\big\vert \\
	&\quad+ \Bigg(1+\rho^{-(k+1)}\int^{\ln\rho}_{-\infty} \omega_{c}(\xi) \dd \xi + \int_{\ln\rho}^0\omega_{c}(\xi) e^{-\xi}\dd \xi\Bigg)\sup_{Q_1} \big\vert f\big\vert\\
	&\quad+\Bigg(\rho^{-(k+1)}\int^{\ln\rho}_{-\infty} \omega_{B}(\xi) \dd \xi + \int_{\ln\rho}^0 \omega_{B}(\xi)e^{-\xi} \dd \xi\Bigg)\sum_{1\leq i\leq d_0} \sup_{Q_1} \big\vert \partial_{x_i} f\big\vert \\
	&\quad+\rho^{-(k+1)}\int^{\ln\rho}_{-\infty}\omega_{h}(\xi) \dd \xi + \int_{\ln\rho}^0\omega_{h}(\xi)e^{-\xi} \dd \xi  + \sup_{Q_1} \abs{h}.
\eals
For the left hand side we find for $z, z_0$ such that $d_\ell(z, z_0) \leq 1/2$ upon choosing $\rho = d_\ell(z, z_0)$
\bals
	\frac{\big\vert D^2 f(z) - D^2 f(z_0)\big\vert^2}{d_\ell(z, z_0)^2} &\leq C [f]^2_{C^{2+1-}_\ell(Q_{\rho})}  \leq C \inf_{p \in \mathcal P_2} \rho^{-(n+6)}\int_{Q_{\rho}} \Abs{\tilde f - p}^2 \dd z\leq C \rho^{-(n+6)}\int_{Q_{\rho}} \Abs{\tilde f - p_2^{(0)}[\tilde f]}^2 \dd z,
\eals
where we used Lemma \ref{lem:2.7} and the characterisation of Campanato norms in Theorem \ref{thm:equicampholderhigh}. This concludes the proof of \eqref{eq:dini}.

\section{Relation between Hölder and Campanato spaces}\label{app:campanato}

This section is devoted to the proof of the equivalence between kinetic Campanato and Hölder spaces, as stated in Theorem \ref{thm:equicampholderhigh}. We follow Campanato's arguments from \cite{Camp}. We recall the notation $\Omega(z_0, r) := \Omega \cap Q_r(z_0)$ for any subset $\Omega \subset \R^n$ with $n = 2s + 2d(1+2s)$. Throughout this section we will denote $\Omega = Q_R(\tilde z_0)$ as in the statement of Theorem \ref{thm:equicampholderhigh}.

\subsection{Auxiliary Result}

We start with a preliminary lemma, which in the elliptic case has first been derived by De Giorgi \cite[Lemma 2.1]{Camp}.  
\begin{lemma}
For a polynomial $P \in \mathcal P_k$, a real number $q \geq 1$, $z_0 \in \R^{1+2d}$, and $\rho > 0$ there exists a constant $c$ such that 
\beqs
	\Big\vert(\partial_t + v\cdot \nabla_x)^{j_0}\partial_{x_1}^{j_1}\cdots\partial_{x_d}^{j_d}\partial_{v_1}^{j_{d+1}}\cdots \partial_{v_{d}}^{j_{2d}}P(z)\big\vert_{z = z_0}\Big\vert^q \leq \frac{c}{\rho^{n+\abs{J}q}}\int_{Q_\rho(z_0)} \abs{P(z)}^q \dd z
\eeqs
where $\abs{J} = 2s\cdot j_0 + (1+2s) \abs{(j_1, \dots, j_d)} + \abs{(j_{d+1}, \dots, j_{2d})}$.
\label{lem:2.1}
\end{lemma}
\begin{proof}
Let $\mathcal T_k \subset \mathcal P_k$ be the subset of $k$-degree polynomials such that 
\beq\label{eq:normalisation}
	\sum_{\abs{J} \leq k} \abs{a_j}^2 = 1,
\eeq
where we recall that $a_j$ are the coefficients of an element $p \in \mathcal P_k$, which can be written as in \eqref{eq:poly}.
Let $\mathcal F$ denote the class of measurable functions $f: \R^n \to [0, 1]$ compactly supported on $Q_1$ such that $\int_{\R^n} f(z) \dd z \geq A$, where $A = \abs{Q_\rho(z_0)} \rho^{-n}$. Let $\gamma(A) = \inf_{P \in \mathcal T_k, f \in \mathcal F} \int_{Q_1} \abs{P(z)}^q f(z) \dd z$.
We want to show that 
\beq
	\gamma(A) =\min_{P \in \mathcal T_k, f \in \mathcal F} \int_{Q_1} \abs{P(z)}^q f(z) \dd z.
\label{eq:claim}
\eeq
For any integer $m$ there exists $P_m \in \mathcal T_k$ and $f_m \in \mathcal F$ such that
\beq\label{eq:subseq}
	\gamma(A) \leq \int_{Q_1} \abs{P_m(z)}^q f_m(z) \dd z < \gamma(A) + \frac{1}{m}.
\eeq
Due to the normalisation \eqref{eq:normalisation} we can extract a subsequence $\{P_\nu\}$ of $\{P_m\}$ converging uniformly on compact subsets of $\R^n$ to $P^* \in \mathcal T_k$. Similarly, since $0 \leq f \leq 1$ we can extract another subsequence $\{f_\mu\}$ of $\{f_\nu\}$ converging weakly in $L^2(Q_1)$ to some $f^* \in \mathcal F$. The subsequence will still satisfy \eqref{eq:subseq}, so that taking the limit yields
\beqs
	\gamma(A) = \int_{Q_1} \abs{P^*(z)}^q f^*(z) \dd z.
\eeqs
This proves the claim \eqref{eq:claim}. It follows that $\gamma(A) > 0$. Moreover, for $z_0$ and $\rho$ such that $Q_\rho(z_0) \subset Q_1$, and for $P \in \mathcal T_k$ there holds
\beqs
	 \gamma(A) \leq \int_{Q_\rho(z_0)} P(z) \dd z. 
\eeqs
since $\abs{Q_\rho(z_0)} \geq A\rho^n$.
If $P \in \mathcal P_k$ then $P(z) \cdot \Big\{\sum_{\abs{J} \leq k} \abs{a_j}^2\Big\}^{-\frac{1}{2}} \in \mathcal T_k$ and thus $\Bigg\{\sum_{\abs{J} \leq k} \abs{a_j}^2\Bigg\}^{\frac{q}{2}}  \leq \frac{1}{\gamma(A)}\int_{Q_\rho(z_0)} \abs{P(z)}^q \dd z$,
or also
\beq
	 \abs{a_j}^q  \leq \frac{1}{\gamma(A)}\int_{Q_\rho(z_0)} \abs{P(z)}^q \dd z, \quad \forall \abs{J} \leq k.
\label{eq:2.11}
\eeq
Now let $P \in \mathcal P_k$. Denote with $(s, y, w) = T(t, x, v)$ the transformation respecting the Lie group structure
\beqs
	\tilde z := (s, y, w) = \Bigg(\frac{t-t_0}{\rho^{2s}},  \frac{x-x_0- (t-t_0)v_0}{\rho^{1+2s}},  \frac{v-v_0}{\rho}\Bigg) = \big(z_0^{-1} \circ z\big)_{\frac{1}{\rho}}.
\eeqs
Then
\bal
	\int_{Q_\rho(z_0)} \abs{P(z)}^q \dd z &= \rho^n \int_{T(Q_\rho(z_0))} \Abs{P(\rho^{2s}s + t_0, \rho^{1+2s}y + x_0 + (t-t_0)v_0, \rho w + v_0)}^q \dd \tilde z\\
	&= \rho^n \int_{T(Q_\rho(z_0))} \Abs{P\big(z_0 \circ \tilde z_\rho\big)}^q \dd \tilde z.
\label{eq:2.12}
\eal
We note that $T(Q_\rho(z_0)) \subset Q_1$, $\abs{T(Q_\rho(z_0))} = \rho^{-n}\int_{Q_\rho(z_0)}\dd z \geq A$ and for $J_1 := (j_1, \dots, j_d)$, $J_2 := (j_{d+1}, \dots, j_{2d})$
\beqs
	P\big(z_0 \circ \tilde z_\rho\big) = \sum_{\abs{J} \leq k}\frac{(\partial_t + v\cdot \nabla_x)^{j_0}\partial_{x_1}^{j_1}\cdots\partial_{x_d}^{j_d}\partial_{v_1}^{j_{d+1}}\cdots \partial_{v_{d}}^{j_{2d}}P(z)\vert_{z = z_0}}{j!} \rho^{2s\cdot j_0}\rho^{(1+2s)\cdot \abs{J_1}}\rho^{\abs{J_2}} \tilde z^j.
\eeqs
Equations \eqref{eq:2.11} and \eqref{eq:2.12} then imply
\begin{equation*}
	\Big\vert(\partial_t + v\cdot \nabla_x)^{j_0}\partial_{x_1}^{j_1}\cdots\partial_{x_d}^{j_d}\partial_{v_1}^{j_{d+1}}\cdots \partial_{v_{d}}^{j_{2d}}P(z)\vert_{z = z_0}\Big\vert^q \leq \frac{(j!)^q}{\rho^{n +q[2s j_0 + (1+2s)\abs{J_1} +\abs{J_2}]}\gamma(A)}\int_{Q_\rho(z_0)} \abs{P(z)}^q \dd z \quad \forall j.\qedhere
\end{equation*}
\end{proof}

\subsection{Expansion of $f$}

We let $f \in \mathcal{L}^{q, \lambda}_k(\Omega)$. For all $z_0 \in \bar\Omega$ and for all $\rho \in [0, \textrm{diam }\Omega]$ we show the existence of a unique polynomial $P_k(z, z_0, \rho, f)$ such that
\beq	
	\inf_{p \in \mathcal P_k} \int_{\Omega(z_0, \rho)} \abs{f(z) - p(z)}^q \dd z = \int_{\Omega(z_0, \rho)} \abs{f(z) - P_k(z, z_0, \rho, f)}^q \dd z. 
\label{eq:3.1}
\eeq
In fact, $P_k(z, z_0, \rho, f)$ is the kinetic Taylor expansion of $f$ at $z_0$. Let $P \in \mathcal P_k$ and write
\beqs
	P(z) = \sum_{j \in \N^{1+2d}, \abs J\leq k} \frac{a_j(z_0)}{j!}(z-z_0)^j.
\eeqs
We denote 
\bals
	h(\{a_j\}) = \norm{f - P}_{L^q(\Omega(z_0, \rho))},
\eals
where $\Omega(z_0, \rho) = Q_R(\tilde z_0) \cap Q_\rho(z_0)$ with $Q_R(\tilde z_0)$ as in the statement of Theorem \ref{thm:equicampholderhigh}.
Note that $h$ is a non-negative continuous real function of the coefficients of $P$. The infimum of $h$ will be attained in a compact set containing the origin, so that the existence of $P_k$ follows standardly. The uniqueness of $P_k$ follows by uniform convexity of the Lebesgue spaces $L^q$. We will denote the coefficients of $P_k(z, z_0, \rho, f)$ with $a_j(z_0, \rho)$. Note that they are given by
\beq
	a_j(z_0, \rho, f) = (\partial_t + v\cdot \nabla_x)^{j_0}\partial_{x_1}^{j_1}\cdots\partial_{x_d}^{j_d}\partial_{v_1}^{j_{d+1}}\cdots \partial_{v_{d}}^{j_{2d}}P_k(z, z_0, \rho, f)\big\vert_{z = z_0}.
\label{eq:3.2}
\eeq
\begin{lemma}
For $f \in \mathcal L^{q, \lambda}_k(\Omega)$ there exists a constant $c(q, \lambda) > 0$ such that for any $z_0 \in \Omega$ and $0 < \rho \leq \textrm{diam }\Omega$ and $l \in \mathbb N_0$ there holds
\beqs
	\int_{\Omega(z_0, \rho 2^{-(l-1)})} \Abs{P_k(z, z_0, \rho2^{-l}, f) - P_k(z, z_0, \rho2^{-l-1}, f)}^q \dd z \leq c 2^{-l\lambda} \rho^\lambda [f]_{\mathcal L^{q, \lambda}_k}^q 
\eeqs
\label{lem:3.1}
\end{lemma}
\begin{proof}
For all $z \in \Omega\big(z_0, \rho 2^{-(l-1)}\big)$ there holds
\bals
	&\abs{P_k\big(z, z_0, \rho2^{-l}, f\big) - P_k\big(z, z_0, \rho2^{-l-1}, f\big)}^q \\
	&\quad\leq 2^q \abs{P_k\big(z, z_0, \rho2^{-l}, f\big) - f(z)}^q + 2^q\abs{P_k\big(z, z_0, \rho2^{-(l-1)}, f\big) - f(z)}^q.
\eals
Thus
\bals
	\int_{\Omega\big(z_0, \rho 2^{-(l-1)}\big)} \abs{P_k\big(z, z_0, \rho2^{-l}, f\big) - P_k\big(z, z_0, \rho2^{-l-1}, f\big)}^q \dd z \leq 2^{q+1}2^{(-l-1)\lambda} \rho^\lambda [f]_{\mathcal L^{q, \lambda}_k}^q. 
\eals
\end{proof}
\begin{lemma}
Suppose $f \in \mathcal L^{q, \lambda}_k(\Omega)$. Then for any $z_0, z_1  \in \bar \Omega$ and for any multi-index $l$ such that $\abs L = k$ with $\abs L = 2s \cdot l_0 + (1+2s)\abs{L_1} + \abs{L_2}$ there holds
\bal
	\Abs{a_l(z_0, 2d_\ell(z_0, z_1), f) - a_l(z_1, 2d_\ell(z_0, z_1), f)}^q \leq c 2^{q + 1 + \lambda}[f]_{\mathcal L^{q, \lambda}_k}^q d_\ell(z_0, z_1)^{\lambda - n - kq},
\label{eq:3.4}
\eal
where $d_\ell$ is the kinetic distance defined in \ref{def:kin_distance}.
\end{lemma}
\begin{proof}
Let $z_0, z_1\in \bar \Omega$. We write $\rho = d_\ell(z_0, z_1)$ and $I_\rho = \Omega(z_0, 2\rho)\cap \Omega(z_1, 2\rho)$. Then we have
\bals
	\abs{P_k(z, z_0, 2\rho, f) - P_k(z, z_1, 2\rho, f)}^q \leq 2^q\abs{P_k(z, z_0, 2\rho, f) - f(z)}^q + 2^q\abs{P_k(z, z_1, 2\rho, f) - f(z)}^q.
\eals
Integrating over $\Omega(z_0, \rho) \subset I_\rho$ we obtain
\bal
	\int_{\Omega(z_0, \rho)} &\abs{P_k(z, z_0, 2\rho, f) - P_k(z, z_1, 2\rho, f)}^q \dd z \\
	&\leq 2^q\int_{\Omega(z_0, \rho)}\abs{P_k(z, z_0, 2\rho, f) - f(z)}^q \dd z + 2^q \int_{\Omega(z_0, \rho)} \abs{P_k(z, z_1, 2\rho, f) - f(z)}^q \dd z\\
	&\leq 2^{q + \lambda +1}\rho^\lambda[f]_{\mathcal L^{q, \lambda}_k}^q.
\label{eq:3.5}
\eal
On the other hand, by \eqref{eq:3.2}, and Lemma \ref{lem:2.1} applied to $P(z) = P_k(z, z_0, 2\rho, f) - P_k(z, z_1, 2\rho, f)$ and since the $k$-th derivative of a polynomial of degree $k$ is constant, we have
\bal
	\Abs{a_l\big(z_0, 2d_\ell(z_0, z_1), f\big) &- a_l\big(z_1, 2d_\ell(z_0, z_1),f\big)}^q \\
	&\leq c \rho^{-(n+kq)}\int_{\Omega(z_0, \rho)} \Abs{P_k(z, z_0, 2\rho, f) - P_k(z, z_1, 2\rho,f)}^q \dd z.
\label{eq:3.6}	
\eal
Finally, the combination of \eqref{eq:3.5} and \eqref{eq:3.6} implies \eqref{eq:3.4} and concludes the proof.
\end{proof}
\begin{lemma}
Let $f \in \mathcal L_k^{q, \lambda}(\Omega)$. Then there exists a constant $c = c(\lambda, q)$ such that for all $z_0 \in \bar \Omega, 0 < \rho \leq \textrm{diam }\Omega$, $i \in \N$ and multi-index $l \in \N^{1+2d}$ with $\abs{L} \leq k$ there holds
\bals
	\Abs{a_l(z_0, \rho, f) - a_l(z_0, \rho 2^{-i}, f)} \leq c [f]_{\mathcal L_k^{q, \lambda}}\sum_{m = 0}^{i-1} 2^{m\big(\frac{n + \abs{L}q - \lambda}{q}\big)}\rho^{\frac{\lambda - n - \abs{L}q}{q}}.
\eals
\label{lem:3.3}
\end{lemma}
\begin{proof}
We have
\bals	 
	\Abs{a_l(z_0, \rho, f) - a_l(z_0, \rho 2^{-i}, f)} \leq \sum_{m=0}^{i-1}\Abs{a_l(z_0, \rho2^{-m}, f) - a_l(z_0, \rho 2^{-m-1}, f)}.
\eals
Using the relation \eqref{eq:3.2} and applying Lemma \ref{lem:2.1} to $P_k(z, z_0, \rho2^{-m}, f) - P_k(z, z_0, \rho2^{-m-1}, f)$ we get
\bals
	&\Abs{a_l(z_0, \rho, f) - a_l(z_0, \rho 2^{-i}, f)} \\
	&\quad\leq c \rho^{-\frac{n}{q} - \abs{L}}\sum_{m=0}^{i-1}2^{(m+1)\big(\frac{n}{q} +\abs L\big)}\Bigg[\int_{\Omega(z_0, \rho2^{-m-1})} \abs{P_k(z, z_0, \rho2^{-m}, f) - P_k(z, z_0, \rho2^{-m-1}, f)}^q\dd z\Bigg]^{\frac{1}{q}}.
\eals
We conclude using Lemma \ref{lem:3.1}.
\end{proof}
Now we can prove the following useful lemma. In the sequel, we let $\textrm{diam }\bar \Omega$ denote the supremum over the length of all paths that stay in $\Omega$.
\begin{lemma}
Let $f \in \mathcal L_k^{q, \lambda}(\Omega)$ such that $n + \tilde kq < \lambda \leq n+ (\tilde k+1)q$ where $0 \leq\tilde k\leq k$. Then there exists functions $\{g_j(z_0)\}$ for $j \in \N^{1+2d}$ with $\abs J \leq \tilde k$ such that for all $0 < \rho \leq \textrm{diam }\bar \Omega, ~ z_0 \in \bar \Omega$ there holds
\beq
	\Abs{a_j(z_0, \rho, f) - g_j(z_0)} \leq c(\lambda, q, k, n) \rho^{\frac{\lambda - n - \abs J q}{q}} [f]_{\mathcal L^{q, \lambda}_k}.
\label{eq:3.10}
\eeq
As a consequence, there holds
\beq
	\lim_{\rho \to 0}a_j(z_0, \rho, f) = g_j(z_0),
\label{eq:3.11}
\eeq
uniformly with respect to $z_0$.
\end{lemma}
\begin{proof}
We show that the sequence $\{a_j(z_0, \rho 2^{-i}, f)\}$ converges in the limit $i \to \infty$. Let $i_1, i_2$ be two non-negative integers and assume without loss in generality that $i_2 > i_1$. With Lemma \ref{lem:3.3} we obtain
\bals
	\Abs{a_j(z_0, \rho 2^{-i_2}, f) - a_j(z_0, \rho 2^{-i_1},f)} \leq c [f]_{\mathcal L_k^{q, \lambda}}\sum_{m = i_1}^{i_2-1} 2^{m\big(\frac{n + \abs{J}q - \lambda}{q}\big)}\rho^{\frac{\lambda - n - \abs{J}q}{q}}.
\eals
Since $\abs{J} \leq p=\tilde k$ and $\lambda > n +\tilde kq$ the series $\sum_{m = 0}^{\infty} 2^{m\big(\frac{n + \abs{J}q - \lambda}{q}\big)}$ converges. Thus $\{a_j(z_0, \rho 2^{-i}, f)\}$ is a Cauchy sequence and hence converges as $i \to \infty$. 

We now show that the limit is uniform in $\rho$. Let $\rho_1$ and $\rho_2$ be such that $0 < \rho_1 \leq \rho_2 \leq\textrm{diam }\Omega$. With Lemma \ref{lem:2.1} we get
\bals
	\Abs{a_j(z_0, \rho_1 2^{-i}, f) - a_j(z_0, \rho_2 2^{-i}, f)}^q &\leq c \frac{2^{i(n+\abs J q)}}{\rho_1^{n+\abs J q}}\int_{\Omega(z_0, \rho_12^{-i})} \Abs{P_k(z, z_0, \rho_12^{-i}, f) - P_k(z, z_0, \rho_2 2^{-i}, f)}^q \dd z\\
	&\leq c \frac{2^{i(n+\abs J q)}}{\rho_1^{n+\abs J q}}\Bigg[\int_{\Omega(z_0, \rho_12^{-i})} \Abs{P_k(z, z_0, \rho_12^{-i}, f) - f(z)}^q \dd z\\
	&\qquad\qquad \qquad+ \int_{\Omega(z_0, \rho_2 2^{-i})} \Abs{P_k(z, z_0, \rho_22^{-i}, f) - f(z)}^q \dd z\Bigg]\\
	&\leq c 2^q \frac{\rho_1^\lambda + \rho_2^\lambda}{\rho_1^{n + \abs J q}} 2^{-i(\lambda - n - \abs J q)} [f]_{\mathcal L^{q, \lambda}_k} \rightarrow 0,
\eals
as $i \to \infty$ since $\lambda - n - \abs J q > 0$. 

Thus for $z_0 \in \bar \Omega, 0 < \rho \leq \textrm{diam }(\Omega)$ and $\abs J \leq \tilde k$ we can take
\beq
	g_j(z_0) = \lim_{i \to \infty} a_j(z_0, \rho 2^{-i}, f).
\label{eq:3.14}
\eeq
The sequence $g_j(z_0)$ is well-defined in $\bar \Omega$. Since the series $\sum_{m=0}^\infty 2^{m\big(\frac{n+\abs J q - \lambda}{q}\big)}$ converges, we deduce from Lemma \ref{lem:3.3}
\bal
	\Abs{a_j(z_0, \rho, f) - a_j(z_0, \rho 2^{-i}, f)} \leq c [f]_{\mathcal L_k^{q, \lambda}}\rho^{\frac{\lambda - n - \abs{J}q}{q}}.
\label{eq:3.15}
\eal
Combining \eqref{eq:3.14} and \eqref{eq:3.15} yields the result.
\end{proof}

\subsection{The function $g_j(z_0)$}
We have the following theorem.
\begin{theorem}
Let $f \in \mathcal L^{q, \lambda}_k(\Omega)$ with $n + kq < \lambda$. Then the functions $g_j(z_0)$ with $\abs J = k$ are Hölder continuous in $\bar \Omega$ and for any $z_1, z_2 \in \bar \Omega$ there holds
\beq
	\abs{g_j(z_1) -g_j(z_2)} \leq c [f]_{\mathcal L^{q, \lambda}_k}d_\ell(z_1, z_2)^{\frac{\lambda - n - kq}{q}}.
\label{eq:4.1}
\eeq
\label{thm:4.1}
\end{theorem}
\begin{proof}
Take $z_1, z_2 \in \bar \Omega$ such that $\rho = d_\ell(z_1, z_2) \leq \frac{\textrm{diam }\Omega}{2}$. Then 
\bals
	\abs{g_j(z_1) -g_j(z_2)} \leq \abs{g_j(z_1) -a_j(z_1, 2\rho)} + \abs{g_j(z_2) -a_j(z_2, 2\rho)} + \abs{a_j(z_1, 2\rho) -a_j(z_2, 2\rho)}.
\eals
On the one hand, by \eqref{eq:3.10} we have
\bals
	\abs{g_j(z_1) -a_j(z_1, 2\rho)} \leq c 2^{\frac{\lambda - n - kq}{q}} \rho^{\frac{\lambda - n - kq}{q}}[f]_{\mathcal L^{q, \lambda}_k},
\eals
and
\bals
	\abs{g_j(z_2) -a_j(z_2, 2\rho)} \leq c 2^{\frac{\lambda - n - kq}{q}} \rho^{\frac{\lambda - n - kq}{q}}[f]_{\mathcal L^{q, \lambda}_k}.
\eals
On the other hand \eqref{eq:3.4} implies
\bals
	\abs{a_j(z_1, 2\rho) -a_j(z_2, 2\rho)} \leq c 2^{\frac{q+1+\lambda}{q}} \rho^{\frac{\lambda - n - kq}{q}} [f]_{\mathcal L^{q, \lambda}_k}.
\eals
This yields the result in case that $d_\ell(z_1, z_2) \leq \frac{\textrm{diam }\Omega}{2}$. 

In case that $d_\ell(z_1, z_2) > \frac{\textrm{diam }\Omega}{2}$ we can construct a polygon contained in $\bar \Omega$ with extremal points $z_1$ and $z_2$ and with sides of length smaller or equal to $\frac{\textrm{diam }\Omega}{2}$, see Figure \ref{fig:polygon}. The length of the sides can be bounded by $\textrm{diam }\Omega$ uniformly with respect to $z_1$ and $z_2$. Thus to conclude it suffices to apply \eqref{eq:4.1} to all points at the end of the sides of such a polygonal.
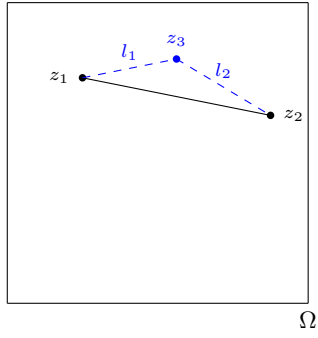
\begin{figure}\label{fig:polygon}
\begin{tikzpicture}[scale =0.5]
  \draw  (-4,0) -- (4,0); 
   \draw (-4,-8) -- (4,-8);
  \draw (-4,0) -- (-4,-8);
  \draw (4,0) -- (4,-8) node[anchor=north, scale=1] {$\Omega$};
  \node[fill,circle,inner sep=1pt,label={left:\footnotesize{$z_1$}}] at (-2, -2) {};
   \node[fill,circle,inner sep=1pt,label={right:\footnotesize{$z_2$}}] at (3, -3) {};
   \node[fill,circle,blue, inner sep=1pt,label={above:\textcolor{blue}{\footnotesize{$z_3$}}}] at (0.5, -1.5) {};
   \draw (-2, -2) --(3, -3);
   \draw[dashed, blue] (-2, -2) -- node[above]{\footnotesize{$l_1$}} (0.5, -1.5);
   \draw[dashed, blue] (0.5, -1.5) -- node[above]{\footnotesize{$l_2$}}(3, -3);
\end{tikzpicture}
\caption{In case that $d_\ell(z_1, z_2) > \frac{\textrm{diam }\Omega}{2}$ we construct a polygon with side lengths $l_1, l_2$ such that $l_1, l_2 \leq \frac{\textrm{diam }\Omega}{2}$.}
\end{figure}
\end{proof}
For the sequel, we denote by $(0)$ the $d$-tuple $(0, \dots, 0)$ and by $e_i$ the vector in $\R^{d}$ with the $i$-th coordinate equal to $1$ and else $0$. We also note that any polynomial degree $k \in \N + 2s\N$ can be written as $k = 2s \cdot k_0 +(1+2s)\cdot k_1+ k_2$ with $k_0, k_1, k_2 \in \N$.

\begin{theorem}
Let $f \in \mathcal L^{q, \lambda}_k(\Omega)$ with $k_0, k_1, k_2 \geq 1$ and $n + kq < \lambda$. Then for any multi-index $j\in \N^{1+2d}$ such that $\abs J \leq k$ the function $g_j$ has a first partial derivative in $\Omega$, and for any $z \in \Omega$ and $i = 1, \dots, d$ there holds
\bal
	&\mathcal T g_{j}(z) = g_{(j_0 + 1, J_1, J_2)}(z),\qquad \qquad j_0 \leq k_0 - 1, \abs{J_1} \leq k_1, \abs{J_2} \leq k_2 \\
	&\frac{\partial g_{j}(z)}{\partial x_i} = g_{(j_0, J_1 + e_i, J_2)}(z),\qquad  \qquad j_0 \leq k_0, \abs{J_1} \leq k_1 - 1,\abs{J_2} \leq k_2 \\
	&\frac{\partial g_{j}(z)}{\partial v_i} = g_{(0, J_1, J_2 + e_i)}(z), \qquad\qquad  j_0 = 0, \abs{J_1} \leq k_1, \abs{J_2} \leq k_2-1. 
\label{eq:4.6}
\eal
\label{thm:4.2}
\end{theorem}
\begin{proof}
For this proof we omit the dependency on $f$ in the coefficients $a_j(z_0, \rho, f)$ and $P_k(z, z_0, \rho, f)$ and simply write $a_j(z_0, \rho)$ and $P_k(z, z_0, \rho)$, respectively.

\textit{Step 1.} We will start proving the first line. We consider $j = (j_0, J_1, J_2)$ for $j_0 \leq k_0 - 1, \abs{J_1} = k_1, \abs{J_2} = k_2$. 
Theorem \ref{thm:4.1} proves that $g_{(k_0, J_1, J_2)}$ is Hölder continuous in a classical sense for $\abs{J_1} = k_1, \abs{J_2} = k_2$ and in particular continuous. Thus we may assume that $g_{(j_0 + \delta, J_1, J_2)}$ is continuous in $\bar \Omega$ for $\delta = 1, \dots, k_0 - j_0$. Let $z_0 \in \Omega$ and $\rho$ be such that $B_{\abs\rho}(z_0) \subset \Omega$. By \eqref{eq:3.2} we have
\bal
	\frac{a_{j}\big(z_0 + (\rho, (0), (0)), 2\abs\rho\big) - a_{j}(z_0, 2\abs \rho)}{\rho} &= \frac{D^{j}\big[P_{k}\big(z, z_0 +(\rho, (0), (0)), 2\abs\rho\big) -P_{k}(z, z_0, 2\abs\rho)\big]}{\rho}\\
	&\qquad- \sum_{\delta = 1}^{k_0 - j_0} \frac{(-1)^{\delta}}{\delta!}\rho^{\delta -1}a_{j}\big(z_0 + (\rho, (0), (0)), 2\abs \rho\big).
\label{eq:4.7}
\eal
With Lemma \ref{lem:2.1} and \eqref{eq:3.5} we obtain
\bal
	&\Bigg\vert \frac{D^{j}\big[P_k\big(z, z_0 +(\rho, (0), (0)), 2\abs\rho\big) -P_k(z, z_0, 2\abs\rho)\big]}{\rho}\Bigg\vert^q \\
	&\qquad\qquad\leq c {\abs\rho}^{-n - \abs{J} q}\int_{\Omega(z_0, \abs \rho)} \Big\vert P_k\big(z, z_0 +(\rho, (0), (0)), 2\abs\rho\big) -P_k(z, z_0, 2\abs\rho)\Big\vert^q \dd z\\
	&\qquad\qquad\leq c 2^{q+\lambda+1}\abs\rho^{\lambda - n - \abs{J} q}[f]_{\mathcal L^{q, \lambda}_k}.
\label{eq:4.8}
\eal
Moreover, for $1 \leq \delta \leq k_0 -j_0$ there holds
\bal
	&\Big\vert a_{(j_0 + \delta, J_1, J_2)}\big(z_0 + (\rho, (0), (0)), 2\abs\rho\big) - g_{(j_0 + \delta, J_1, J_2)}(z_0)\Big\vert\\
	&\qquad\qquad\leq \Abs{a_{(j_0 + \delta, J_1, J_2)}\big(z_0 + (\rho, (0), (0)), 2\abs\rho\big) - g_{(j_0 + \delta, J_1, J_2)}\big(z_0+(\rho, (0), (0))\big)} \\
	&\qquad\qquad\quad+ \Abs{g_{(j_0 + \delta, J_1, J_2)}\big(z_0 + (\rho, (0), (0))\big) - g_{(j_0+ \delta, J_1, J_2)}(z_0)}.
\label{eq:4.9}
\eal
Using \eqref{eq:3.10} we can estimate the first term on the right hand side of \eqref{eq:4.9} by
\bal
	\Big\vert a_{(j_0 + \delta, J_1, J_2)}\big(z_0 + (\rho, (0), (0)), 2\abs\rho\big) - g_{(j_0 + \delta, J_1, J_2)}&\big(z_0+(\rho, (0), (0))\big)\Big\vert  \\
	&\leq c 2^{\frac{\lambda - n -(\abs{J} + 2s\delta)}{q}}\abs{\rho}^{\frac{\lambda - n - (\abs{J} + 2s \delta)}{q}} [f]_{\mathcal L_k^{q, \lambda}}.
\label{eq:4.10}
\eal
From \eqref{eq:4.9} and \eqref{eq:4.10} and since by induction hypothesis $g_{(j_0 + \delta, J_1, J_2)}$ are continuous for $\delta = 1, \dots, k_0 - j_0$ we have
\beq
	\lim_{\rho \to 0} a_{(j + \delta, J_1, J_2)}\big(z_0+ (\rho, (0), (0)), 2\abs \rho\big) = g_{(j_0+ \delta, J_1, J_2)}(z_0) \qquad \delta = 1, \dots, k_0 - j_0.
\label{eq:4.11}
\eeq
Thus from \eqref{eq:4.7}, \eqref{eq:4.8} and \eqref{eq:4.11} we deduce that
\beqs
	\lim_{\rho \to 0} \frac{a_j\big(z_0 + (\rho, (0), (0)), 2\abs\rho\big) - a_{j}(z_0, 2\abs \rho)}{\rho} = g_{(j_0 + 1, J_1, J_2)}(z_0),
\eeqs
uniformly in $z_0$. Thus if we can show that
\beq
	\lim_{\rho \to 0} \frac{g_{j}\big(z_0 + (\rho, (0), (0))\big) - g_{j}(z_0)}{\rho} = \lim_{\rho \to 0} \frac{a_{j}\big(z_0 + (\rho, (0), (0)), 2\abs\rho\big) - a_{j}(z_0, 2\rho)}{\rho},
\label{eq:4.13}
\eeq 
then we can conclude the proof of the first line of \eqref{eq:4.6}.
We first notice that by \eqref{eq:3.10} 
\bal
	\Bigg\vert\frac{g_{j}\big(z_0 + (\rho, (0), (0))\big) - a_{j}\big(z_0+(\rho, (0), (0)), 2\abs\rho\big)}{\rho} \Bigg\vert\leq c2^{\frac{\lambda - n - \abs{J}q}{q}}\abs{\rho}^{\frac{\lambda - n}{q} -\abs{J} - 1}[f]_{\mathcal L^{q, \lambda}_k},
\label{eq:4.14}
\eal
and
\bal
	\Big\vert \frac{g_{j}(z_0) - a_{j}(z_0, 2\abs\rho)}{\rho}\Big\vert\leq c 2^{\frac{\lambda - n - \abs{J}q}{q}}\abs{\rho}^{\frac{\lambda - n}{q} - \abs{J} - 1}[f]_{\mathcal L^{q, \lambda}_k}.
\label{eq:4.15}
\eal
Thus with the triangle inequality \eqref{eq:4.14} and \eqref{eq:4.15} imply \eqref{eq:4.13}, which in turn implies the first line of\eqref{eq:4.6}.

\textit{Step 2.} To prove the second statement in \eqref{eq:4.6} we proceed as in Step 1. Now we consider $j_0 = 1, \dots, k_0, \abs{J_1} \leq k_1 - 1$ and $\abs{J_2} = k_2$. We have shown that $g_j$ is continuous for $j_0 = 1, \dots, k_0, \abs{J_1} = k_1, \abs{J_2} = k_2$. Assume then that $g_{(j_0, J_1 + \delta e_i, J_2)}$ is continuous in $\bar \Omega$ for $\delta = 1, \dots, k_1 - \abs{J_1}$. We again have by \eqref{eq:3.2}
\bal
	\frac{a_{j}\big(z_0 +\rho (0, e_i, (0)), 2\abs\rho\big) - a_{j}(z_0, 2\abs \rho)}{\rho} &= \frac{D^{j}\big[P_{k}\big(z, z_0 +\rho(0, e_i, (0)), 2\abs\rho\big) -P_{k}(z, z_0, 2\abs\rho)\big]}{\rho}\\
	&\quad- \sum_{\delta = 1}^{k_1 - \abs{J_1}} \frac{(-1)^{\delta}}{\delta!}\rho^{\delta -1}a_{j}\big(z_0 + \rho(0, e_i, (0)), 2\abs \rho\big).
\label{eq:4.7i}
\eal
The proof is exactly the same if we replace $(\rho, (0), (0))$ with $\rho(0, e_i, (0))$, $k_0- j_0$ with $k_1-\abs{J_1}$ and instead of $2s\delta$ in the exponent of \eqref{eq:4.10} we get $(1+2s)\delta$. 

\textit{Step 3.} To deduce the final statement in \eqref{eq:4.6} the ideas are the same but the statement only holds for $j_0 = 0$ since $\mathcal T$ and $D_v$ do not commute. Therefore it was important to prove the first statement first, since now we know that $g_j$ is continuous for $j_0 = 0, \abs{J_1} \leq k_1$ and $\abs{J_2} = k_2$. We now assume that $g_{(j_0, J_1, J_2 + \delta e_i)}$ is continuous in $\bar \Omega$ for $\delta = 1, \dots, k_2 - \abs{J_2}$. Replacing $(\rho, (0), (0))$ with $\rho(0, (0), e_i)$, $k_0- j_0$ with $k_2-\abs{J_2}$ and $2s\delta$ in the exponent of \eqref{eq:4.10} with $\delta$, but otherwise proceeding as above, we conclude.

Finally, combining the argument for the continuity of $g_j$ in all three steps yields the improvement in ranges of $\abs{J_1}$ and $\abs{J_2}$ as stated in the theorem. 
\end{proof}
As a corollary of Theorem \ref{thm:4.1} and \ref{thm:4.2} we get
\begin{theorem}\label{thm:4.3}
Let $f \in \mathcal L^{q, \lambda}_k(\Omega)$ with $n+ kq< \lambda$. Then the function $g_{(0)} \in C_\ell^{\beta}(\bar \Omega)$ where $\beta = \frac{\lambda - n}{q}$ and there holds
\beqs
	\mathcal T^{j_0}D_x^{J_1}D_v^{J_2} g_{(0)}(z) = g_j(z) \qquad \forall z \in \Omega, ~ \forall \abs J \leq k.
\eeqs
Recall $j = (j_0, J_1, J_2) \in \N^{1+2d}$ and $\abs{J} = 2s \cdot j_0 + (1+2s)\cdot \abs{J_1} + \abs{J_2}$. 
\end{theorem}
\begin{remark}
For $f \in \mathcal L^{q, \lambda}_k(\Omega)$ with $n+ (k+1)q< \lambda$ we deduce from \eqref{eq:4.1} that $g_j$ with $\abs J = k$ are constant and thus by Theorem \ref{thm:4.3}, $g_{(0)}$ is a polynomial of kinetic degree at most $k$.
\label{rmk:4.2}
\end{remark}

\subsection{Comparing the Hölder norm and the Campanato norm}
\begin{theorem}
Let $f \in \mathcal L^{q, \lambda}_k(\Omega)$ with $n+ kq< \lambda \leq n + (k+1)q$. Then $f \in C_\ell^{\beta}(\Omega)$ where $\beta = \frac{\lambda - n}{q}$ and there holds
\beq
	[f]_{C_\ell^{\beta}} \leq c [f]_{\mathcal L^{q, \lambda}_k}.
\label{eq:5.1}
\eeq
If $\lambda > n + (k+1)q$ then $f$ is a polynomial of kinetic degree at most $k$. 
\label{thm:5.1}
\end{theorem}
\begin{proof}
Due to Theorem \ref{thm:4.3} and Remark \ref{rmk:4.2} it suffices to show that $f(z) = g_{(0)}(z) = \lim_{\rho \to 0} a_{(0)}(z, \rho)$ for almost every $z \in \Omega$. Then \eqref{eq:5.1} follows from \eqref{eq:4.1} in Theorem \ref{thm:4.1} and Taylor's formula. 

Since $f \in L^q(\Omega)$ there holds for almost every $z_0 \in \Omega$ 
\beq
	\lim_{\rho \to 0} \frac{1}{\abs{\Omega(z_0, \rho)}} \int_{\Omega(z_0, \rho)} \abs{f(z) - f(z_0)}^q\dd z= 0.
\label{eq:5.2}
\eeq
Now let $z_0 \in \Omega$ be such that \eqref{eq:5.2} holds. Then for almost every $z \in \Omega$ we have
\bals
	\Abs{a_{(0)}(z_0, \rho) - f(z_0)}^q \leq c \Big( \Abs{P_k(z, z_0, \rho) - a_{(0)}(z_0, \rho)}^q + \Abs{P_k(z, z_0, \rho) - f(z)}^q + \abs{f(z) - f(z_0)}^q\Big).
\eals
Integrating this inequality over $\Omega(z_0, \rho)$ yields
\bal
	\Abs{a_{(0)}(z_0, \rho) - f(z_0)}^q &\leq \frac{c}{A_1\rho^n}\int_{\Omega(z_0, \rho)}\Abs{P_k(z, z_0, \rho) - a_{(0)}(z_0, \rho)}^q \dd z \\
	&\quad +  \frac{c}{A_1\rho^n}\int_{\Omega(z_0, \rho)} \Abs{P_k(z, z_0, \rho) - f(z)}^q \dd z+  \frac{c}{A_1\rho^n}\int_{\Omega(z_0, \rho)} \abs{f(z) - f(z_0)}^q\dd z.
\label{eq:5.3}
\eal
By definition of $\mathcal L^{q, \lambda}_k$ we have
\beqs
	\frac{c}{R^{-n}\abs{Q_R(\tilde z_0)}\rho^n}\int_{\Omega(z_0, \rho)}\Abs{P_k(z, z_0, \rho) - f(z)}^q \dd z \leq c \frac{\rho^{\lambda - n}}{R^{-n}\abs{Q_R(\tilde z_0)}}[f]_{\mathcal L^{q, \lambda}_k} \xrightarrow[\rho \to 0]{} 0.
\eeqs
Due to \eqref{eq:5.2} the last integral in \eqref{eq:5.3} vanishes as well in the limit $\rho \to 0$. Finally there holds
\bals
	 \frac{c}{R^{-n}\abs{Q_R(\tilde z_0)}\rho^n}\int_{\Omega(z_0, \rho)}\Abs{P_k(z, z_0, \rho) - a_{(0)}(z_0, \rho)}^q \dd z \leq c(n, q, k) \sum_{\substack{j \in \N^{1+2d}, \\ \abs{J} \leq k}}\Abs{a_j(z_0, \rho)}^q \rho^{\abs J q}. 
\eals
Due to \eqref{eq:3.11} this integral vanishes in the limit $\rho \to 0$, so that \eqref{eq:5.3} gives for almost every $z_0 \in \Omega$
\beqs
	\lim_{\rho \to 0} a_{(0)}(z_0, \rho) = f(z_0).
\eeqs
Equivalently, there holds $f(z) = g_{(0)}(z)$ almost everywhere in $\Omega$.
\end{proof}

\begin{proof}[Proof of Theorem \ref{thm:equicampholderhigh}]
If $f \in \mathcal L_k^{p, \lambda}(\Omega)$, then Theorem \ref{thm:5.1} yields $f \in C_\ell^{\beta}(\Omega)$ and the Hölder semi-norm is bounded above by the Campanato semi-norm \eqref{eq:5.1}.

Conversely, let $f \in C_\ell^{\beta}(\bar\Omega)$ and $P \in \mathcal P_k$ where $k = \rm{deg}_{\rm{kin}} P < \beta$. For $z \in Q_r(z_0) \cap \Omega$ we have 
\beqs
	\abs{f(z) - P(z)} \leq [f]_{C_\ell^{\beta}} r^{\beta}.
\eeqs
Thus for $\beta = \frac{\lambda - n}{p}$ there holds
\bals
	\frac{1}{r^\lambda}\int_{Q_r(z_0) \cap \Omega} \abs{f(z) - P(z)}^p \dd z\leq C  [f]^p_{C_\ell^{\beta}} r^{p\beta - \lambda +n} = C  [f]^p_{C_\ell^{\beta}}.
\eals
\end{proof}

\section{Interpolation Inequality for Hölder spaces}\label{app:interpolation}
For the sake of completeness, we prove Lemma \ref{prop:interpolation} following the arguments of Imbert-Silvestre \cite[Proposition 2.10]{ISschauder}.
\begin{proof}[Proof of Lemma \ref{prop:interpolation}]

It suffices to prove the statement for $\beta_3$ sufficiently close to $\beta_1$. Thus we assume that there exists only one element $\bar \beta \in \N +2s\N$ such that $\bar \beta \in [\beta_1, \beta_3)$.
We know that if $p^i_z \in \mathcal P_{\beta_i}$ is the polynomial expansion of $f$ at $z$ of order less than $\beta_i$ for all $i \in \{1, \dots, 3\}$, then for all $z \circ \xi \in Q_1$
\beq
	\Abs{f(z \circ \xi)- p^i_z(\xi)} \leq [f]_{C_\ell^{\beta_i}} \norm{\xi}^{\beta_i}, \quad i = 1, 2, 3.
\label{eq:aux2.6}
\eeq
The polynomials $p_z^i$ are of increasingly higher order. We assume that the difference of degree of homogeneity of $p_z^1$ and $p_z^3$ is at most one, so that $p_z^2$ coincides with either $p_z^1$ or $p_z^3$, depending on whether $\bar \beta \geq \beta_2$ or $\bar \beta< \beta_2$. If there is no $\bar \beta$ then all three polynomials coincide. Let us first assume therefore that there is exactly one $\bar \beta$. We have by subtracting \eqref{eq:aux2.6} for $i = 1, 3$ from each other
\beq
	\abs{p_z^3(\xi) - p_z^1(\xi)} \leq [f]_{C_\ell^{\beta_1}}\norm{\xi}^{\beta_1} + [f]_{C_\ell^{\beta_3}}\norm{\xi}^{\beta_3}.
\label{eq:aux2.7}
\eeq
For any $R \in (0, 1]$ and $z \in Q_1$ we pick $\xi_1 \in Q_1$ such that $\norm{\xi_1} \leq R$ and whenever $d_\ell(\xi_1, \xi) < cR$, then $\norm{\xi} \leq R$ and $z \circ \xi \in Q_1$ with some universal constant $c$. From \eqref{eq:aux2.7} we then have
\beqs
	 \sup_{\xi: d_\ell(\xi_1, \xi) \leq cR} \abs{p_z^3(\xi) - p_z^1(\xi)} \leq [f]_{C_\ell^{\beta_1}}R^{\beta_1} +  [f]_{C_\ell^{\beta_3}} R^{\beta_3}.
\eeqs
Since $p_z^3 - p_z^1$ is homogeneous of degree $\bar \beta$ we get by scaling
\beqs
	 \sup_{\xi: d_\ell((\xi_1)_{R^{-1}}, \xi) \leq c} \abs{p_z^3(\xi) - p_z^1(\xi)} \leq [f]_{C_\ell^{\beta_1}}R^{\beta_1 -\bar \beta} +  [f]_{C_\ell^{\beta_3}} R^{\beta_3 - \bar \beta}.
\eeqs
Using the triangle inequality from \cite[Prop. 2.2]{ISschauder} we can assure that whenever $\abs{\xi}\leq 1$ then $d_\ell\big((\xi_1)_{R^{-1}}, \xi\big) \leq C$ for some universal constant $C$. Since all norms on the space of polynomials are equivalent, we have
\bals
	\norm{p_z^3 - p_z^1} &= \sup_{\xi: \norm{\xi}\leq 1} \abs{p_z^3(\xi) - p_z^1(\xi)} \leq C\sup_{\xi: d_\ell\big((\xi_1)_{R^{-1}}, \xi\big) \leq c} \abs{p_z^3(\xi) - p_z^1(\xi)} \\
	&\leq C [f]_{C_\ell^{\beta_1}}R^{\beta_1-\bar\beta} +  C[f]_{C_\ell^{\beta_3}} R^{\beta_3 - \bar \beta}.
\eals
For $$R = \Bigg(\frac{[f]_{C_\ell^{\beta_1}}}{[f]_{C_\ell^{\beta_3}}}\Bigg)^{\frac{1}{\beta_3 - \beta_1}}$$ we obtain
\beqs
	\norm{p_z^3 - p_z^1}\leq C [f]_{C_\ell^{\beta_1}}^{\bar \theta}[f]_{C_\ell^{\beta_3}}^{1-\bar \theta} + [f]_{C_\ell^{\beta_1}},
\eeqs
where $\bar \beta =\bar \theta \beta_1 + (1-\bar\theta)\beta_3$.

Therefore we can estimate $f - p_z^2$. Assume first $\beta_2 \leq \bar \beta$. Then $p_z^2 = p_z^1$ and 
\beqs
	\abs{f(z \circ \xi) - p_z^2(\xi)} \leq \begin{cases} [f]_{C_\ell^{\beta_1}}\norm{\xi}^{\beta_1},\\ [f]_{C_\ell^{\beta_3}}\norm{\xi}^{\beta_3} + \left([f]_{C_\ell^{\beta_1}}^{\bar \theta}[f]_{C_\ell^{\beta_3}}^{1-\bar \theta} + [f]_{C_\ell^{\beta_1}}\right)\norm{\xi}^{\bar \beta}. \end{cases}
\eeqs
Now if $\norm{\xi} \geq R$ then
\beqs
	[f]_{C_\ell^{\beta_1}}\norm{\xi}^{\beta_1} \leq [f]_{C_\ell^{\beta_1}}^{ \theta}[f]_{C_\ell^{\beta_3}}^{1- \theta}\norm{\xi}^{\beta_2}.
\eeqs
Else if $\norm{\xi} < R$
\beqs
	[f]_{C_\ell^{\beta_3}}\norm{\xi}^{\beta_3} + \Big([f]_{C_\ell^{\beta_1}}^{\bar \theta}[f]_{C_\ell^{\beta_3}}^{1-\bar \theta} + [f]_{C_\ell^{\beta_1}}\Big)\norm{\xi}^{\bar \beta} \leq [f]_{C_\ell^{\beta_1}}^{ \theta}[f]_{C_\ell^{\beta_3}}^{1- \theta}\norm{\xi}^{\beta_2} + [f]_{C_\ell^{\beta_1}}\norm{\xi}^{\bar \beta}.
\eeqs
Thus we conclude $\abs{f(z \circ \xi) - p_z^2(\xi)} \leq [f]_{C_\ell^{\beta_1}}^{ \theta}[f]_{C_\ell^{\beta_3}}^{1- \theta}\norm{\xi}^{\beta_2} + [f]_{C_\ell^{\beta_1}}\norm{\xi}^{\beta_2}$. 

In case that $\bar \beta < \beta_2$ 
\beqs
	\abs{f(z \circ \xi) - p_z^2(\xi)} \leq \begin{cases} [f]_{C_\ell^{\beta_1}}\norm{\xi}^{\beta_1}+ \Big([f]_{C_\ell^{\beta_1}}^{\bar \theta}[f]_{C_\ell^{\beta_3}}^{1-\bar \theta} + [f]_{C_\ell^{\beta_1}}\Big)\norm{\xi}^{\bar\beta},\\ [f]_{C_\ell^{\beta_3}}\norm{\xi}^{\beta_3}. \end{cases}
\eeqs
and we conclude as above. 

In case that no $\bar \beta$ exists, then all polynomials coincide and we get
\beqs
	\abs{f(z \circ \xi) - p_z^2(\xi)} \leq [f]_{C_\ell^{\beta_1}}^{ \theta}[f]_{C_\ell^{\beta_3}}^{1- \theta}\norm{\xi}^{\beta_2}.
\eeqs	
\end{proof}

\section{Proof of Bouchut's Proposition}\label{app:bouchut}
For the sake of self-containment, we recall the proof of Proposition \ref{prop:bouchut} from \cite[Proposition 1.1]{Bouchut}.
\begin{proof}[Proof of Proposition \ref{prop:bouchut}]
Let $\psi_\delta \in C_c^\infty(\R^{1+d})$ be a standard mollifier in $(t, x)$. We carry out the whole argument first on $f_\delta := f \ast_{t, x} \psi_\delta$ which is smooth in $(t, x)$ and solves 
\[
	\partial_t f_\delta + v\cdot \nabla_x f_\delta = S \ast_{t, x} \psi_\delta =: S_\delta.
\]
The estimate for general $f \in L^2(\R^{1+2d})$ then follows by Fatou's lemma as $\delta \to 0$.

We denote by $\hat f(\eta, k, v)$ the Fourier-transform of a solution $f$ of \eqref{eq:transport} in time $t$ and space $x$. Then $\hat f$ solves
\beqs
	i (\eta + v \cdot k) \hat f = \hat S. 
\eeqs
We introduce a smoothing sequence $\rho_1 \in C_c^\infty(\R^d)$ in velocity such that 
\beq\label{eq:properties-rho}
	\rho_\varepsilon(v) = \frac{1}{\varepsilon^d}\rho_1\Big(\frac{v}{\varepsilon}\Big), \qquad \int \rho_1 \dd v = 1, \qquad \int v^\alpha \rho_1 = 0 \textrm{ for } 1 \leq \abs\alpha < \abs\beta.
\eeq
For fixed $(\eta, k)$ we decompose 
\beq\label{eq:decomposition}
	\hat f(\eta, k, v) = \left(\rho_\varepsilon \ast_v \hat f\right)(\eta, k, v) + \left( \hat f -  \left(\rho_\varepsilon \ast_v \hat f\right)\right)(\eta, k, v),
\eeq
where $\ast_v$ denotes the convolution in velocity $v$. Then by the properties of $\rho$ \eqref{eq:properties-rho} we can bound $\abs{1 - \hat \rho_\varepsilon} \leq C_{d, \beta}\abs{\varepsilon v}^\beta$ so that
\beq\label{eq:bouchut-aux1}
	\Bigg\Vert\left( \hat f -  \left(\rho_\varepsilon \ast_v \hat f\right)\right)(\eta, k, \cdot)\Bigg\Vert_{L^2(\R^d)} \leq C_{d, \beta} \varepsilon^\beta\Norm{\abs{D_v}^\beta \hat f(\eta, k, \cdot)}_{L^2(\R^d)}.
\eeq	
For the first term in \eqref{eq:decomposition} we introduce $\lambda > 0$ such that 
\beqs
	\big(\lambda + i(\eta + v \cdot k)\big) \hat f(\eta, k, v) = \lambda \hat f(\eta, k, v) + \hat S(\eta, k v).
\eeqs
Equivalently, 
\beqs
	\hat f(\eta, k, v) = \frac{\lambda \hat f(\eta, k, v) + \hat S(\eta, k, v)}{\lambda + i(\eta + v \cdot k)}, 
\eeqs
which yields
\beqs
	\left(\rho_\varepsilon \ast_v \hat f\right)(\eta, k, v) = \int \frac{\lambda \hat f(\eta, k, \xi) + \hat S(\eta, k, \xi)}{\lambda + i(\eta + \xi \cdot k)} \rho_\varepsilon(v - \xi) \dd \xi.
\eeqs
Then we bound
\bals
	\Big\vert &\left(\rho_\varepsilon \ast_v \hat f\right)(\eta, k, v) \Big\vert \\
	&\leq \Big(\Norm{\hat f(\eta, k, \cdot) \abs{\rho_\varepsilon(v- \cdot)}^{\frac{1}{2}}}_{L^2(\R^d)}+ \lambda^{-1}\Norm{\hat S(\eta, k, \cdot) \abs{\rho_\varepsilon(v- \cdot)}^{\frac{1}{2}}}_{L^2(\R^d)}\Big)\Bigg(\int \frac{\abs{\rho_\varepsilon(v - \xi)}}{\abs{1 + i(\eta + \xi \cdot k)\lambda^{-1}}^2}  \dd \xi\Bigg)^{\frac{1}{2}}.
\eals
The last integral is estimated using $\abs{\rho_\varepsilon(v)} \leq C_{d,\beta} \varepsilon^{-d}\chi_{\abs v \leq \varepsilon}$, and decomposing $\xi = \tilde \xi \frac{k}{\abs{k}} + \xi^\perp$ with $\xi^\perp \cdot k = 0$, so that
\bals
	\int \frac{\abs{\rho_\varepsilon(v - \xi)}}{\abs{1 + i(\eta + \xi \cdot k)\lambda^{-1}}^2}  \dd \xi \leq C_{d, \beta} \frac{1}{\varepsilon} \int \frac{\chi_{\abs{\frac{v \cdot k}{\abs k} - \tilde \xi}< \varepsilon}}{\abs{1 + i(\eta + \tilde\xi \cdot k)\lambda^{-1}}^2} \dd \tilde \xi \leq C_{d, \beta} \frac{\lambda}{\varepsilon \abs{k}}.
\eals
Thus
\bals
	\Big\Vert\left(\rho_\varepsilon \ast_v \hat f\right)(\eta, k, \cdot)\Big\Vert_{L^2(\R^d)} \leq C_{d, \beta} \Big(\frac{\lambda}{\varepsilon \abs{k}}\Big)^{\frac{1}{2}} \Big(\Norm{\hat f(\eta, k, \cdot)}_{L^2(\R^d)}+ \lambda^{-1}\Norm{\hat S(\eta, k, \cdot)}_{L^2(\R^d)}\Big).
\eals
Choosing 
\beqs
	\lambda = \frac{\Norm{\hat S(\eta, k, \cdot)}_{L^2(\R^d)}}{\Norm{\hat f(\eta, k, \cdot)}_{L^2(\R^d)}}
\eeqs
yields
\beq\label{eq:bouchut-aux2}
	\Big\Vert\left(\rho_\varepsilon \ast_v \hat f\right)(\eta, k, \cdot)\Big\Vert_{L^2(\R^d)} \leq  \frac{C_{d, \beta}}{\sqrt{\varepsilon\abs k}} \Norm{\hat f(\eta, k, \cdot)}_{L^2(\R^d)}^{\frac{1}{2}}\Norm{\hat S(\eta, k, \cdot)}_{L^2(\R^d)}^{\frac{1}{2}}.
\eeq
Combining \eqref{eq:decomposition} with \eqref{eq:bouchut-aux1} and \eqref{eq:bouchut-aux2} yields
\beqs
	\Norm{\hat f(\eta, k, \cdot)}_{L^2(\R^d)} \leq \frac{C_{d, \beta}}{\sqrt{\varepsilon\abs k}} \Norm{\hat f(\eta, k, \cdot)}_{L^2(\R^d)}^{\frac{1}{2}}\Norm{\hat S(\eta, k, \cdot)}_{L^2(\R^d)}^{\frac{1}{2}} + C_{d, \beta} \varepsilon^\beta\Norm{\abs{D_v}^\beta \hat f(\eta, k, \cdot)}_{L^2(\R^d)}.
\eeqs
We finally optimise $\varepsilon$ so that
\beqs
	\Norm{\hat f(\eta, k, \cdot)}_{L^2(\R^d)} \leq \Bigg(\frac{1}{\abs k}\Norm{\hat f(\eta, k, \cdot)}_{L^2(\R^d)}\Norm{\hat S(\eta, k, \cdot)}_{L^2(\R^d)}\Bigg)^{\frac{\beta}{1+2\beta}}\Norm{\abs{D_v}^\beta \hat f(\eta, k, \cdot)}^{\frac{1}{1+2\beta}}_{L^2(\R^d)}.
\eeqs
Dividing by $\Norm{\hat f(\eta, k, \cdot)}_{L^2(\R^d)}^{\frac{\beta}{1+2\beta}}$ yields
\beqs
	\Norm{\hat f(\eta, k, \cdot)}_{L^2(\R^d)} \leq \Bigg(\frac{1}{\abs k}\Norm{\hat S(\eta, k, \cdot)}_{L^2(\R^d)}\Bigg)^{\frac{\beta}{1+\beta}}\Norm{\abs{D_v}^\beta \hat f(\eta, k, \cdot)}^{\frac{1}{1+\beta}}_{L^2(\R^d)},
\eeqs
which concludes the proof of \eqref{eq:bouchut} after integrating over $(\eta, k)$.
\end{proof}

\bigskip
\noindent {\bf Acknowledgements.} 
We thank Clément Mouhot for his continuous support and intuitive discussions on the subject. This work was part of our PhD thesis at DPMMS, University of Cambridge, supported by the Cambridge International \& Newnham College Scholarship from the Cambridge Trust.

\bibliographystyle{plain.bst}
\bibliography{Schauder}

\end{document}